\documentclass[11pt]{amsart}

\usepackage[numbered]{bookmark}
\usepackage{changepage} 
\usepackage[T1]{fontenc}
\usepackage[utf8]{inputenc}
\usepackage{graphics, graphicx}
\usepackage{amsmath, amsfonts, amsthm, amssymb, mathrsfs, mathtools, mathabx}
\usepackage{verbatim}
\usepackage{mathrsfs}

\usepackage{enumerate}
\usepackage{textcase}

\usepackage{xcolor}
\usepackage{cite}
\usepackage{wrapfig}
\usepackage{subcaption}
\usepackage[font=small]{caption}

\newtheorem{theorem}{Theorem}[section]
\newtheorem{definition}[theorem]{Definition}
\newtheorem{lemma}[theorem]{Lemma}
\newtheorem{proposition}[theorem]{Proposition}
\newtheorem{corollary}[theorem]{Corollary}

\newtheorem{remark}{Remark}[section]

\newtheorem{theoremx}{Theorem}

\newtheorem{corollaryx}{Corollary}

\newtheorem{theoremxx}{Theorem}

\numberwithin{equation}{section}

\newcommand{\N}{\mathbb{N}}

\newcommand{\R}{\mathbb{R}}
\newcommand{\C}{\mathbb{C}} 
\newcommand{\T}{\mathbb{T}}
\newcommand{\Z}{\mathbb{Z}}

\newcommand{\Lo}{\ensuremath{\mathcal{L}_\omega}}

\newcommand{\LO}{\mathfrak{L}_{\omega,\Omega}}
\newcommand{\Mell}{\ensuremath{M_\Omega^{\mathrm{ell}}}}
\newcommand{\Mhyp}{\ensuremath{M_\Omega^{\mathrm{hyp}}}}
\newcommand{\Mpar}{\ensuremath{M_\Lambda^{\mathrm{par}}}}
\newcommand{\Spol}{\ensuremath{\mathcal{S}^{T}}}

\newcommand{\dec}{\ensuremath{l}}

\newcommand{\Lell}{\mathcal{L}}

\newcommand{\Function}[5]{\begin{array}{cccc} #1 : & #2 & \rightarrow & #3 \\ & #4 & \mapsto & #5 \end{array}}

\definecolor{modgreen}{rgb}{0.00,0.50,0.00}

\date{}
\author{Donato Scarcella}
\address{Departament de Matemàtiques i Informàtica, Universitat de Barcelona, Gran Via de les Corts Catalanes 585, 08007 Barcelona, Spain}
\email{donato.scarcella@ub.edu}

\author{Frank Trujillo}
\address{Centre de Mathématiques Laurent Schwartz (CMLS), CNRS, École polytechnique, Institut Polytechnique de Paris, Palaiseau, France.}
\email{frank.trujillo-amezquita@cnrs.fr}

\begin{document}
\title[Non-autonomous KAM theory II: Normally parabolic case]{Non-autonomous KAM theory for lower dimensional invariant tori (II):
\NoCaseChange{Normally parabolic case}}
\maketitle

\begin{abstract}
    Dynamical systems subject to non-autonomous perturbations decaying in time arise naturally in many physical contexts, including laser-molecule interactions, epidemiological models, nonlinear oscillatory systems, and celestial mechanics. In this  paper, we consider time-dependent perturbations decaying polynomially fast in time of Hamiltonian systems having a lower-dimensional isotropic normally parabolic invariant torus supporting quasiperiodic solutions. We prove the existence of invariant manifolds in the extended phase space, and determine the asymptotic behavior of the associated transverse dynamics. The above results are obtained for Hölder, smooth, and analytic Hamiltonian systems.
\end{abstract}

\section{Introduction}
\label{sc:introduction}
This paper is the second part of our study on non-autonomous KAM theory for lower-dimensional invariant tori. While the first part is devoted to the normally elliptic and hyperbolic cases~\cite{scarcella_KAMST}, here we consider the normally parabolic setting. We refer the reader to that paper for the motivation, the relevant background, and a comprehensive overview of the existing literature.

More precisely, we study time-dependent perturbations of Hamiltonians having a lower dimensional isotropic normally parabolic invariant torus with quasiperiodic solutions. Under suitable polynomial decay in time of the perturbations, we prove the existence of an \textit{asymptotic KAM torus} (see Definition \ref{def:hyp_asym_KAM_tori} below). Heuristically, it consists of an invariant manifold in the extended phase space whose dynamics converge, as time tends to infinity, to the quasiperiodic solutions of the unperturbed system.

Besides proving the existence of asymptotic KAM tori, the main objective of the present work is to investigate the asymptotic behaviour of the transverse dynamics along these invariant manifolds. More precisely, we study the solutions of the linearized Hamiltonian vector field along the trajectories contained in the asymptotic KAM torus and show that the corresponding cocycle is asymptotically conjugate to an explicit model cocycle displaying parabolic behavior. Consequently, the asymptotic transverse dynamics is of the same type as that of the unperturbed invariant torus. We call these objects \textit{asymptotically parabolic KAM tori} (see Definition \ref{def:ell_hyp_par_trasv_dyn_asym_KAM} below). The terminology normally parabolic refers to the fact that the linearized Hamiltonian vector field along the invariant torus has vanishing normal eigenvalues. In contrast with the normally elliptic and hyperbolic cases, this degeneracy leads to polynomial, rather than oscillatory or exponential, behavior of the transverse dynamics.

The purpose of this paper, together with~\cite{scarcella_KAMST}, is to develop a non-autonomous KAM theory for lower-dimensional invariant tori. The classical KAM theory, initiated by Kolmogorov~\cite{Kol54}, Arnold~\cite{Arn63a}, and Moser~\cite{M62} shows the persistence of quasiperiodic invariant tori in nearly-integrable Hamiltonian systems.  Over the past seventy years, this theory has undergone numerous developments, generalizations, and refinements, including persistence results for lower-dimensional invariant tori with normally elliptic and hyperbolic behavior~\cite{El88, G74, Zeh76, M18, M19, Ru01}.

To the best of our knowledge, there are no results in terms of persistence of normally parabolic invariant tori. Nevertheless, such invariant objects naturally arise in several problems of celestial mechanics, in particular in the study of parabolic motions at infinity in the $N$-body problem, where they play a fundamental role in the description of the asymptotic dynamics~\cite{BFM20, BFM24}. Moreover, they provide the geometric framework underlying several instability and chaotic phenomena, including Arnold diffusion~\cite{DKdlRS19, GPS25}.

A recent direction of research concerns perturbations depending explicitly on time and decaying as time tends to infinity. In a series of works~\cite{FW14,CdlL15, Sca22, Sca22b, Sca22c, Sca25}, time-dependent perturbations of Hamiltonians possessing a Lagrangian invariant torus supporting quasiperiodic dynamics are considered. In this non-autonomous framework, one cannot expect the persistence of invariant tori under the perturbation. Instead, it is possible to prove the existence of asymptotic KAM tori. We point out that, unlike classical KAM theory, this non-autonomous framework does not require non-degeneracy conditions or arithmetic assumptions on the frequency vector. This is due to the absence of small denominators in the homological equations arising in time-dependent problems.

As mentioned before, in this work, we consider time-dependent Hamiltonian vector fields converging polynomially fast in time to Hamiltonian vector fields having a lower-dimensional isotropic normally parabolic invariant torus that supports quasiperiodic solutions. We prove the existence of asymptotic KAM tori and asymptotically parabolic KAM tori in the Hölder (Theorems \ref{Thm:par_Csigma_0} and \ref{Thm:par_Csigma}), $C^\infty$ (Theorem \ref{thm:KAM_smooth}) and analytic (Theorems  \ref{Thm:par_Csigma_0_analy} and \ref{Thm:par_analy}) framework, the proofs in the analytic setting being very similar to those in the Hölder case.

The normally parabolic setting naturally splits into two regimes, although in both cases all transverse eigenvalues of the linearized unperturbed Hamiltonian vector field along the invariant torus vanish. Heuristically, the distinction lies in the transverse dynamics near the invariant torus of the unperturbed system. In one case, the transverse dynamics is degenerate, whereas in the other a non-trivial linear behavior appears. We refer the reader to Section~\ref{sc:analysis_transversal} for further details. For this reason, these two cases are treated separately (see Theorems \ref{Thm:par_Csigma_0} and \ref{Thm:par_Csigma_0_analy} for the degenerate case and Theorems \ref{Thm:par_Csigma} and  \ref{Thm:par_analy} for the general case). The general case requires stronger polynomial decay assumptions, making the proof technically more involved because of the presence of the non-trivial linear transverse components.

In both cases, first we prove the existence of an asymptotic KAM torus (see Item \ref{thm:C_existence} of Theorems \ref{Thm:par_Csigma_0} and \ref{Thm:par_Csigma}). The proof follows the same general strategy as in the normally elliptic and hyperbolic case~\cite{scarcella_KAMST}. The main difficulty is again to reformulate the invariance equation in a functional framework suitable for applying a quantitative version of the Implicit Function Theorem. This requires a careful analysis of the Banach spaces involved and of the corresponding linearized problem, which is solved through suitable cohomological equations. In the general case, however, the analysis of both the invariance equation and its linearization becomes considerably more technical and computationally involved. We point out that no smallness assumption on the perturbative terms are assumed. We prove the existence of asymptotic KAM tori for time sufficiently large. Then, if the flow is well-defined for all $t \in \R$ one could extend the set of definition using the flow (see Definition \ref{def:Csigma_KAM_torus} and~\eqref{eq:asymptotic_KAM_flow} below).

As in~\cite{scarcella_KAMST}, the novely of the present work is to prove that the above asymptotic KAM torus is asymptotically parabolic (see Item \ref{thm:C_transverse} of Theorems \ref{Thm:par_Csigma_0} and \ref{Thm:par_Csigma}). More precisely, assuming a slightly stronger polynomial decay in time of the perturbation, we prove that the transverse dynamics of the linearized Hamiltonian vector field along the trajectories contained in the asymptotic KAM torus has an asymptotically parabolic character. To this end, we construct a $C^1$ non-autonomous matrix that asymptotically conjugates the corresponding cocycle to the cocycle of a suitable linear model exhibiting parabolic transverse dynamics. The proof relies on the introduction of a nonlinear functional equation whose unknown is the conjugating matrix. The main difficulty consists in identifying a formulation for this equation whose linearization can be analyzed in appropriate weighted Banach spaces. This eventually reduces to solving a coupled system of cohomological equations with suitable decay estimates, allowing the application of a quantitative Implicit Function Theorem to establish the desired asymptotic conjugacy. While the overall strategy is the same in both parabolic regimes, the general case requires a considerably more delicate analysis due to the presence of a non-trivial linear transverse dynamics.

In addition, under slightly stronger polynomial decay assumptions, we show that the asymptotic transverse dynamics is completely determined by that of the limiting autonomous Hamiltonian system (see Corollaries \ref{cor:par_deg} and \ref{cor:par}). In particular, we establish a one-to-one asymptotic correspondence between the transverse solutions of the perturbed and limiting linearized Hamiltonian vector fields.

 The interest of this result lies, to our knowledge, in the fact that no analogous persistence theory currently exists for isotropic normally parabolic invariant tori within the classical KAM framework. Although the overall strategy follows the approach developed in~\cite{scarcella_KAMST}, extending it to the parabolic setting is far from straightforward: the general case requires substantially stronger decay assumptions and a considerably more delicate and computationally involved analysis.

The paper is organized as follows. In Section \ref{sc:results}, we state the main results of this work. Section \ref{sc:analysis_transversal} studies the asymptotic properties of a suitable linear model displaying parabolic transverse dynamics. In Section \ref{sc:criteria_asym_dyn}, we provide sufficient conditions determining the asymptotic character of the transverse dynamics associated with an asymptotic KAM torus. Section \ref{sec:Proof_Theorem_Par_0} containes the proofs of Theorems \ref{Thm:par_Csigma_0}, \ref{Thm:par_Csigma_0_analy} and Corollary \ref{cor:par_deg}, whereas the proofs of Theorems \ref{Thm:par_Csigma}, \ref{Thm:par_analy} and Corollary \ref{cor:par}  are given in Section \ref{sec:Proof_Theorem_Par}. 

\section{Main results}
\label{sc:results}

\subsection{Preliminaries}

Throughout this article, we will often consider compositions of the form $f(g(\cdot, t), t)$ where $f(\cdot, t)$ and $g(\cdot, t)$ are two (possibly) time-dependent maps. Notice that, using the notation above, this map is formally defined as $(\cdot, t) \mapsto f^t \circ g^t (\cdot)$. However, to avoid cumbersome notation and since we will not consider any kind of time reparametrizations, given $f(\cdot, t)$ and $g(\cdot, t)$ we will denote simply by $f \circ g$ the map $(\cdot, t) \mapsto f(g(\cdot, t), t)$, whenever it is well-defined. In particular, when $g$ is time-independent, the map $f \circ g$ is given by $(\cdot, t) \mapsto f(g(\cdot), t).$

Let $B \subset \R^{n+2m}$ be an open ball centered at the origin.  For any $T \geq 1$ we denote $I_T = [T, +
\infty)$. Given $f:\T^n \times B \times I_T \to \R$, we define
\begin{equation}\label{def:f0}
    f_0:\T^n \times I_T \to \R, \quad f_0(q,t) = f(q,0,0,t)
\end{equation}
whereas, for a fixed $t \in I_T$, let
\begin{equation}\label{def:ft}
        f^t:\T^n \times B \to \R, \quad f^t(q,p, z) = f(q,p,z,t).
\end{equation}
The notation extends naturally to real-valued and matrix-valued functions.

Let $\sigma \ge 1$, a positive integer $k \ge 0$ and a vector $\omega \in \R^n$, we consider time-dependent vector fields $X^t, X_0^t \in C^{\sigma+k}\left(\T^n \times B\right)$, for all fixed $t \in I_T$, continuous with respect to $t$, and a $C^\sigma$ embedding $\varphi_0:\T^n \to \T^n \times B$ such that   
\begin{equation}
\begin{aligned}\label{def:hyp_asym_KAM_tori}
    &\lim_{t \to +\infty} \left|X^t - X^t_0\right|_{C^{\sigma+k}} = 0,\\
  &X_0 \circ \varphi_0 = \partial_q\varphi_0 \cdot \omega. 
\end{aligned}
\end{equation}

\begin{definition}\label{def:Csigma_KAM_torus}
    We assume that $(X, X_0, \varphi_0)$, defined on $\T^n \times \R^{n + 2m} \times I_T$, satisfy~\eqref{def:hyp_asym_KAM_tori}. A family of embeddings $\varphi:\T^n \times I_{T'} \to \T^n \times B$, with $T' \geq T$, is a $C^\sigma$ \emph{asymptotic KAM torus} associated with $(X, X_0, \varphi_0)$, for some $\sigma \geq 1$, if
    \begin{align}\label{def:cond1_asymKAMtorus}
            &\lim_{t \to +\infty}\left|\varphi^t -\varphi_0\right|_{C^\sigma}=0,\\ \label{def:cond2_asymKAMtorus}
             &X \circ \varphi = \partial_q\varphi \cdot \omega + \partial_t \varphi.
    \end{align}
    
We say that $\varphi$ is a $C^\infty$ asymptotic KAM torus associated with $(X, X_0, \varphi_0)$ if it is a $C^\sigma$ asymptotic KAM torus associated with $(X, X_0, \varphi_0)$, for all $\sigma \geq 1$.
\end{definition}

Let us point out that \eqref{def:cond2_asymKAMtorus} is equivalent to
    \begin{equation}
    \label{eq:asymptotic_KAM_flow}
        \psi_{t_0, X}^t \circ \varphi^{t_0} (q) = \varphi^t(q + \omega(t-t_0)), \qquad \text{ for any } t_0, t \in I_{T'} \text{ and any } q\in \T^n,
    \end{equation}
where $\psi_{t_0, X}^t$ denotes the flow at time $t$ with initial time $t_0$ associated with $X$.

Let $\mathsf{r}, \mathsf{s} \in \N$, we denote by $\mathcal{M}_{\mathsf{r} \times \mathsf{s}}(\R)$ the space of $\mathsf{r} \times \mathsf{s}$ matrices with real coefficients and write $\mathcal{M}_\mathsf{r}(\mathbb{R})$ when $\mathsf{r} = \mathsf{s}$.   Given  $\mathsf{r} \in \N$, we define the standard symplectic matrix as
\begin{equation}\label{def:J}
     J_\mathsf{r} = \begin{pmatrix} 0 & \mathrm{Id}_\mathsf{r} \\ -\mathrm{Id}_\mathsf{r} & 0\end{pmatrix},
\end{equation}
where $\mathrm{Id}_\mathsf{r} \in \mathcal{M}_m(\R)$ denotes the $\mathsf{r} \times \mathsf{r}$ identity matrix. 

Given $\Lambda_1,\dots,\Lambda_m \in \R_{\ge 0}$, we define $\Lambda = \mathrm{diag}(\Lambda_1,\dots, \Lambda_m)$ and denote
\begin{equation}\label{Ms}
   \Mpar = \begin{pmatrix} \Lambda & 0 \\ 0 & 0 \end{pmatrix}.
\end{equation}

To describe a \emph{parabolic} asymptotic behaviour of the linearized system $DX$ along the solutions given by an asymptotic KAM torus $\varphi$ associated to $(X, X_0, \varphi)$, we introduce the following definition.

    \begin{definition}
    \label{def:ell_hyp_par_trasv_dyn_asym_KAM}
    Let $(X, X_0, \varphi_0)$ defined on $\T^n \times B \subseteq \T^n \times \R^{n + 2m}$ satisfying \eqref{def:hyp_asym_KAM_tori} and let $\varphi: \T^n \times [T', +\infty) \to \T^n \times \R^{n + 2m}$ be a $C^\sigma$ asymptotic KAM torus associated with $(X, X_0, \varphi_0)$.
    
    We say that $\varphi$ is said to be asymptotically \emph{parabolic}, if there exist a continuous map $A:\T^n \times [T'', +\infty) \to \mathcal{M}_{2n+2m}(\R)$ of the form
    \begin{equation}
     \label{def:A}
        A^t(q) =\begin{pmatrix} 0_{n \times n} & a_{12}^t(q) & a_{13}^t(q) \\
                                0_{n \times n} & 0_{n \times n} & 0_{n \times 2m} \\
                                0_{2m \times n} & a_{32}^t(q) & J_m \Mpar \end{pmatrix},
    \end{equation}
   with $a_{12}^t:\T^n \to \mathcal{M}_n(\R)$, $a_{13}^t :\T^n \to \mathcal{M}_{n\times 2m}(\R)$, and  $a_{32}^t:\T^n \to \mathcal{M}_{2m\times n}(\R)$,  satisfying
    \begin{equation}
    \label{def:cond1_ellhyppar_asymKAMtorus}
        \sup_{t \in I_{T^{''}}}|A^t|_{C^0}<\infty,
    \end{equation}
    such that the cocycles $\Phi(t; q, t_0)$ and $\Phi_A(t; q, t_0)$, given by the fundamental solutions of the systems
    \begin{gather}
    \label{eq:linearized_system}
    \dot \eta(t) = DX^t \circ \varphi^t (q + \omega t) \eta(t),  \qquad (q, t) \in \T^n \times [T'', +\infty),\\
    \label{eq:Trasv_dyn_AA}
    \dot \xi(t) = A^t(q + \omega t) \xi(t),  \qquad (q, t) \in \T^n \times [T'', +\infty),
    \end{gather}
are \emph{conjugated} by a $C^1$ map $S: \T^n \times [T'', +\infty) \to GL(2n + 2m, \R)$,  that is,
    \begin{equation}
\Phi(t; q, t_0) = S(q + \omega t, t)\Phi_A(t; q, t_0)S(q + \omega t_0, t_0)^{-1},
        \label{eq:conjugated_cocycles}
    \end{equation} 
    and satisfy
    \begin{equation}\label{def:trasv_dyn_limit} 
        \lim_{t \to +\infty}\big|\pi_z \Phi(t; q,t_0) - \pi_z \Phi_A(t; q, t_0)S(q + \omega t_0, t_0)^{-1}\big| = 0.
    \end{equation}
\end{definition}

As we shall see, the linear cocycle $\Phi_{A}(t; q, t_0)$ in the definition above for a map $A$ as in \eqref{def:A} displays a transverse parabolic behaviour similar to that of the linearized system of a Hamiltonian with a parabolic invariant torus (see Section \ref{sc:analysis_transversal}).

Let us recall that, denoting
\[  \Mell = \begin{pmatrix} \Omega & 0 \\ 0 & \Omega \end{pmatrix}, \qquad \Mhyp = \begin{pmatrix} \Omega & 0 \\ 0 & -\Omega \end{pmatrix},\]
for $\Omega = \mathrm{diag}(\Omega_1,\dots, \Omega_m)$ with $\Omega_1,\dots,\Omega_m \in \R_{>0}$, if we replace $\Mpar$ by $\Mell$ (resp. $\Mhyp$) in the definition above, we recover the notion of asymptotically \emph{elliptic} (resp. \emph{hyperbolic}) KAM torus, which was explored in detail in \cite{scarcella_KAMST}.

In this work, we consider the following perturbative Hamiltonian setting: $X_0$ and $X$ correspond to Hamiltonian vector fields $X_{H_0}, X_H$ associated to Hamiltonians $H_0$ and $H$, where $H$ is an appropriate perturbation of $H$. We will work on the extended phase space $\T^n \times \R^n \times \R^{2m} \times \R$, where $n \geq 1, \, m \geq 0$, and whose coordinates we denote by $(q,p,z,t)$. We sometimes denote $z = (x, y) \in \R^m \times \R^m$. Given $T \geq 0$, we introduce the interval $I_T=[T, +\infty) \subset \R$. We endow $\T^n \times \R^n \times \R^{2m}$ with the canonical symplectic form $\sum_{i = 1}^n dq_i\wedge dp_i + \sum_{j = 1}^m dx_i \wedge dy_i$ and denote 
\begin{equation}
    \label{eq:big_symplectic_matrix}
    J = \mathrm{diag}(J_n, J_m) = \begin{pmatrix}
        J_n & 0_{2m \times 2m} \\
        0_{2m \times 2n} & J_m
    \end{pmatrix}
\end{equation}
where we refer to~\eqref{def:J} for the definition of $J_n$ and $J_m$.
Notice that for any Hamiltonian $H$ of class $C^1$ on $(q, p, z)$, defined on an open subset of $\T^n \times \R^n \times \R^{2m} \times \R$, the associated Hamiltonian vector field (with respect to the canonical symplectic form), which throughout this work we denote by $X_H$, is given by
\begin{equation}
    X_H = J \nabla H,
\end{equation}
where $J$ is given by \eqref{eq:big_symplectic_matrix}.

Throughout this work, given $\mathsf{r}, \mathsf{s}\in\N$ and a bilinear form $M :\R^\mathsf{r}\times \R^\mathsf{s} \to \R$, we denote the action of $M$ on $(v_1, v_2)\in \R^\mathsf{r}\times \R^\mathsf{s}$ by $M \cdot (v_1, v_2)$. When, $\mathsf{r} = \mathsf{s}$ and $v_1 = v_2 = v$, we simply write $M \cdot v^2$. This notation extends naturally to $k$-linear forms.

\subsection{Statements of the Main Results}

We will consider a (possibly) non-autonomous Hamiltonian $H_0: \T^n \times B \times I_T \to \R$ of the form 
\begin{equation} \label{def:H_aut}
H_0(q, p, z, t) =  \omega \cdot p + \Mpar \cdot z^2 +  \mathcal{O}(p^2,pz,z^3), 
\end{equation}
where $B \subseteq \R^{n + 2m}$ is an open ball centred at the origin and the matrix $\Mpar$ is given by \eqref{Ms}, for some $\Lambda_1, \dots, \Lambda_n \in \R_{\geq 0}$ so that the \emph{trivial embedding} $\varphi_0$ given by
\begin{equation}\label{def:varphi0=(q,0,0)}
    \varphi_0 : \T^n \to \T^n \times \R^n \times \R^{2m}, \qquad \varphi_0(q) = (q,0,0),
\end{equation}
defines an invariant normally parabolic torus with restricted dynamics given by the continuous translation by $\omega$. We then consider $H = H_0 + P$ for some $P: \T^n \times B \times I_T \to \R$ exhibiting some appropriate decay in time as $t$ goes to infinity.

Let us recall that if we replace $\Mpar$ by $\Mell$ (resp. $\Mhyp$) in \eqref{def:H_aut}, we recover the setting of \emph{elliptic} (resp. \emph{hyperbolic}) invariant torus for the unperturbed system, which we treated in \cite{scarcella_KAMST}.

We will analyze the degenerate case $\Lambda_1,\dots, \Lambda_m =0$  and the general case $\Lambda_1,\dots, \Lambda_m \ge 0$, separately (we refer to Theorems \ref{Thm:par_Csigma_0}, \ref{Thm:par_Csigma_0_analy} and Theorem \ref{Thm:par_Csigma}, \ref{Thm:par_analy}, respectively).

\subsubsection{Hölder setting}\label{sec:Holder_Set}

In order to quantify the regularity of time-dependent smooth functions, we introduce the following Banach spaces. Given $\sigma \ge 0$, $k \in \Z_{\geq0}$, $T\ge 1$ and $\ell >0$, we define
\begin{align}\label{def:S}
\mathcal{S}^{ T}_{(\sigma,k),\ell}
    = \left\{
    f : \T^n \times B \times I_T \to \R \;\middle|\;
    \begin{aligned}
        &f^t \in C^{\sigma+k}(\T^n \times B), \text{ for } t \in I_T; \\
        & \sup_{t \in I_T} \bigl(|f^t|_{C^{\sigma+k}}\, t^\ell \bigr) < \infty; \\
        &\partial^i_{(q,p,z)} f \in C(\T^n \times B \times I_T), \text{ for } 0 \le |i| \le k
    \end{aligned}
    \right\}
\end{align}
and endow it with the norm
\begin{equation}\label{def:norm_S}
    |f|^{T}_{\sigma+k, \ell} = \sup_{t \in I_T}|f^t|_{C^{\sigma+k}}t^\ell,
\end{equation}
where $\partial^{i}_{(q,p,z)}=\partial^{i_1}_{q_1} \cdots\partial^{i_n}_{q_n}\partial_{p_1}^{i_{n+1}}\cdots \partial_{p_1}^{i_{2n}}\partial_{z_1}^{i_{2n+1}}\cdots \partial_{z_{2m}}^{i_{2n+2m}}$ and $|i|= |i|_1,$ for any $i \in \N^{2n + 2m}$ and, as a convention, we adopt $\partial^0_{(q,p,z)}f = f$.

In general, we will use the same notation if the function is defined on $\T^n \times I_T$, as well as for vector-valued or matrix-valued functions. More precisely, we say that a vector-valued or matrix-valued function belongs to $\Spol_{(\sigma, k), \ell}$ if it is the case for each of its components.  The norms of a real-valued function or a matrix are defined as the maximum of those norms of its components. It will be specified by the context. %
When the need arises to specify the codomain of the function, we will denote the space of functions in $\Spol_{(\sigma, k), \ell}$ taking values in $\R^{\mathsf{p}}$ and $\mathrm{Mat}_{{\mathsf{p} \times \mathsf{r}}}(\R)$ by $\Spol_{(\sigma, k), \ell}(\R^\mathsf{p})$ and $\Spol_{(\sigma, k), \ell}(\R^{\mathsf{p} \times \mathsf{r}})$, respectively.

The following theorem concerns the case  $\Lambda_1,\dots, \Lambda_m =0$ with Hölder regularity assumptions.

\begin{theoremx}\label{Thm:par_Csigma_0}
    Fix $\sigma \ge 1$, $\ell >1$ and $l \ge 0$. Let $H_0: \T^n  \times B \times [1, +\infty) \to \R$ of the form
\begin{equation}
\label{eq:initial_hamiltonian_par_deg}
 H_0(q, p, z, t) =  \omega \cdot p  + O(p^2,pz,z^3), \qquad \partial_{(p,z)}^2 H_0 \in \mathcal{S}_{(\sigma,3),0}^{1},
\end{equation}
and $P: \T^n  \times B \times [1, +\infty) \to \R$ of  the form
\begin{equation}\label{eq:perturbation_form_par_deg}
P(q, p, z, t) = a(q, t) + b(q, t) \cdot p + c(q, t) \cdot z + \frac{1}{2}d(q, t)\cdot z^2.
\end{equation}
Then, for $H := H_0 + P$ the following holds.
\begin{enumerate}
    \item \label{thm:Cdeg_existence} Suppose that
    \begin{equation*}
        a \in \mathcal{S}^{1}_{(\sigma, 2),0}, \hspace{2mm} \partial_qa \in \mathcal{S}_{(\sigma, 1), \ell+ \dec + 2}^{1}, \hspace{2mm}  b \in \mathcal{S}_{(\sigma, 2), \ell}^{1}, \hspace{2mm}  c \in \mathcal{S}_{(\sigma, 2), \ell+ \dec + 1}^{ 1}, \hspace{2mm} d \in \mathcal{S}_{(\sigma, 2), \ell}^{ 1}.
    \end{equation*}
      Then, there exists a $C^\sigma$ asymptotic KAM torus $\varphi$ associated with $(X_H, X_{H_0}, \varphi_0)$ of the form
    \begin{equation}
        \label{eq:asymp_KAM_torus_form_par0}
    \begin{gathered}
        \varphi:\T^n \times I_T \to \T^n \times \R^n \times \R^{2m}, \qquad \varphi^t = (\textup{id}_{\T^n} + u^t, v^t, w^t), \qquad T \geq 0, \\
     u \in \mathcal{S}_{(\sigma, 0), \ell-1}^{T}, \quad  v \in \mathcal{S}_{(\sigma, 0), \ell+ \dec + 1}^{T}, \quad w \in \mathcal{S}_{(\sigma, 0), \ell+ \dec}^{T}.
    \end{gathered}
        \end{equation}
        Moreover, any other $C^\sigma$ asymptotic KAM torus associated with $(X_H, X_{H_0}, \varphi_0)$ of the form \eqref{eq:asymp_KAM_torus_form_par0} coincides with $\varphi$ in the intersection of their domains.
    \medskip
    
    \item \label{thm:Cdeg_transverse}  Suppose that $\sigma \ge 2$ and
    \begin{equation}
    \label{eq:decay_pardeg_2}
        a \in \mathcal{S}^{1}_{(\sigma, 2),0}, \hspace{2mm}   \partial_qa \in \mathcal{S}_{(\sigma, 1), \ell+ \dec + 3}^{ 1}, \hspace{2mm}  b \in \mathcal{S}_{(\sigma, 2), \ell}^{ 1}, \hspace{2mm}  c \in \mathcal{S}_{(\sigma, 2), \ell+\dec+2}^{ 1}, \hspace{2mm}  d \in \mathcal{S}_{(\sigma, 2), \ell+\dec+1}^{ 1}.
    \end{equation}
    Then the $C^\sigma$ asymptotic KAM torus $\varphi$ associated with $(X_H, X_{H_0}, \varphi_0)$ given by the first statement is asymptotically parabolic.
\end{enumerate}
\end{theoremx}

The following corollary considers the case where the Hamiltonian $H$ in the statement of Theorem \ref{Thm:par_Csigma_0} converges, as time tends to infinity, to an autonomous Hamiltonian $H^\infty$ of the form~\eqref{def:H_aut}. 
Under slightly stronger assumptions than those of Item \ref{thm:Cdeg_transverse} of Theorem \ref{Thm:par_Csigma_0}, we prove a one-to-one correspondence between the solutions, restricted to the transverse $z$-coordinates, of the linearized vector field $DX_H$ along orbits contained in the asymptotic KAM torus $\varphi$ and those associated with $DX_{H^\infty} \circ \varphi_0(q+\omega t)$, where $\varphi_0$ is the trivial embedding in~\eqref{def:varphi0=(q,0,0)}. To state this result, we introduce the following projection \begin{equation}\label{def:Pi_p}
\Pi_p = \left(0_n \quad \mathrm{Id}_n \quad 0_{n \times 2m}\right) \in \mathcal{M}_{n \times 2n+2m}.
\end{equation}
\begin{corollaryx}\label{cor:par_deg}
    Under the hypotheses of Item \ref{thm:Cdeg_transverse} of Theorem \ref{Thm:par_Csigma_0}, we assume the existence of a $C^2$ autonomous Hamiltonian $H^\infty: \T^n  \times B  \to \R$ of the form~\eqref{def:H_aut}, with $M = \Mpar$ and $\Lambda_1, \dots, \Lambda_m =0$, such that  
\begin{equation}\label{hyp:cor_par0_decay}
        \int_1^{+\infty} \sup_{q \in \T^n}|\partial_z\partial_pH^t(q,0,0) - \partial_z\partial_pH^\infty(q,0,0)| \, dt < \infty.
\end{equation}
Let $\varphi: \T^n \times [T'', +\infty) \to \T^n \times \R^{n + 2m}$ be the $C^\sigma$ asymptotically elliptic KAM torus associated with $(X_H, X_{H_0}, \varphi_0)$ given by Theorem \ref{Thm:par_Csigma_0}, $S$ be the map defined in Definition \ref{def:ell_hyp_par_trasv_dyn_asym_KAM} and $\varphi_0$ be the trivial embedding in~\eqref{def:varphi0=(q,0,0)}. 
We consider the following systems 
\begin{align}\label{cor:syst_H_par_deg}
    \dot \eta(t) = DX_H^t \circ \varphi^t (q + \omega t) \eta(t), & \qquad (q, t) \in \T^n \times [T'', +\infty), \\ \label{cor:syst_Hinfty_par_deg}
    \dot \xi^\infty(t) = DX_{H^\infty}^t\circ \varphi_0(q + \omega t) \xi^\infty(t), & \qquad (q, t) \in \T^n \times [T'', +\infty).
\end{align}
    Then, if $\ell >2$, for every fixed $q\in\T^n$, $t_0\in I_{T^{''}}$ and $v_0 \in \R^n$ there exists a one-to-one correspondence between the sets 
     $$\left\{ \pi_z\eta \mid \eta \text{ is a solution  of ~\eqref{cor:syst_H_par_deg}} \text{ with } \pi_p\eta(t_0) = v_0\right\},$$
    $$\left\{ \pi_z\xi^\infty \mid \xi^\infty \text{ is a solution  of ~\eqref{cor:syst_Hinfty_par_deg}} \text{ with } \pi_p \xi^\infty(t_0) = \Pi_p S^{t_0}(q +\omega t_0)v_0\right\},$$ by the relation 
       \begin{equation*}
          \lim_{t \to +\infty}|\pi_z\eta(t)-\pi_z\xi^\infty(t)|=0.
    \end{equation*}
\end{corollaryx}

In the following theorem, we consider the general case when $\Lambda_1,\dots, \Lambda_m \ge 0$ under Hölder regularity assumptions. 
\begin{theoremx}\label{Thm:par_Csigma}
Fix $\sigma \ge 1$, $\ell >1$ and $\dec \ge 0$. Let $H_0:\T^n \times B \times [1,+\infty) \to \R$ of the form
\begin{equation}
\label{eq:initial_hamiltonian_par}
 H_0(q, p, x,y, t) =  \omega \cdot p + \frac{1}{2}\Lambda  \cdot x^2 + O(p^2,px, py,x^3, y^3), \qquad \partial_{(p,x,y)}^2 H_0 \in \mathcal{S}_{(\sigma,3),0}^{1},
\end{equation}
and $P: \T^n  \times B \times [1, +\infty) \to \R$ of  the form
\begin{equation}
\label{eq:perturbation_form_par}
\begin{aligned}
P(q, p, z, t) & = a(q, t) + b(q, t) \cdot p + c_1(q, t) \cdot x + c_2(q,t) \cdot y  \\ & \quad + \frac{1}{2}d_1(q, t)\cdot x^2 + d_2(q,t)\cdot (x,y) + {1 \over 2} d_3(q,t) \cdot y^2.
\end{aligned}
\end{equation}
Then, for $H:= H_0 + P$ the following holds.
\begin{enumerate}
    \item \label{thm:C_existence} Suppose that
    \begin{equation*}
        a \in \mathcal{S}^{ 1}_{(\sigma, 2),0}, \quad \partial_q a \in \mathcal{S}_{(\sigma, 1), \ell+ \dec+4}^{1}, \quad b \in \mathcal{S}_{(\sigma, 2), \ell}^{1}, \quad  c_1 \in \mathcal{S}_{(\sigma, 2), \ell+ \dec+2}^{ 1}, \quad  c_2 \in \mathcal{S}_{(\sigma, 2), \ell+ \dec+3}^{ 1}, 
        \end{equation*}
        \begin{equation*}
         d_1 \in \mathcal{S}_{(\sigma, 2), \ell-1}^{ 1}, \quad d_2 \in \mathcal{S}_{(\sigma, 2), \ell}^{ 1}, \quad d_3 \in \mathcal{S}_{(\sigma, 2), \ell+1}^{ 1}.
    \end{equation*}
    Then, there exists a $C^\sigma$ asymptotic KAM torus $\varphi$ associated with $(X_H, X_{H_0}, \varphi_0)$ of the form
    \begin{equation}
        \label{eq:asymp_KAM_torus_form_par}
    \begin{gathered}
        \varphi:\T^n \times I_T \to \T^n \times \R^n \times \R^{2m}, \qquad \varphi^t = (\textup{id}_{\T^n} + u^t, v^t, w^t), \qquad T \geq 0, \\
            u \in \mathcal{S}_{(\sigma, 0), \ell-1}^{T}, \quad  v \in \mathcal{S}_{(\sigma, 0), \ell+ \dec + 3}^{T}, \quad w_x \in \mathcal{S}_{(\sigma, 0), \ell+ \dec + 2}^{T} , \quad w_y \in \mathcal{S}_{(\sigma, 0), \ell+ \dec + 1}^{T}.
    \end{gathered}
        \end{equation}
        Moreover, any other $C^\sigma$ asymptotic KAM torus associated with $(X_H, X_{H_0}, \varphi_0)$ of the form \eqref{eq:asymp_KAM_torus_form_par} coincides with $\varphi$ in the intersection of their domains.
    \medskip
    
    \item \label{thm:C_transverse} Suppose that $\sigma \ge 2$ and
    \begin{equation} \label{eq:Hyp_decay_par_2}
    \begin{aligned}
        a \in \mathcal{S}^{ 1}_{(\sigma, 2),0}, \quad \partial_q a &\in \mathcal{S}_{(\sigma, 1), \ell+\dec+6}^{1}, \quad b \in \mathcal{S}_{(\sigma, 2), \ell}^{1}, \quad  c_1 \in \mathcal{S}_{(\sigma, 2), \ell+\dec+4}^{ 1}, \quad  c_2 \in \mathcal{S}_{(\sigma, 2), \ell+\dec+5}^{ 1}, \\
         d_1 &\in \mathcal{S}_{(\sigma, 2), \ell+\dec+1}^{ 1}, \quad d_2 \in \mathcal{S}_{(\sigma, 2), \ell+\dec+2}^{ 1}, \quad d_3 \in \mathcal{S}_{(\sigma, 2), \ell+\dec+3}^{ 1}.
        \end{aligned} 
    \end{equation}
    Then the $C^\sigma$ asymptotic KAM torus $\varphi$ associated with $(X_H, X_{H_0}, \varphi_0)$ given by the first statement is asymptotically parabolic.
\end{enumerate}
\end{theoremx}

The following corollary is the counterpart of Corollary \ref{cor:par_deg} in the general parabolic setting.
\begin{corollaryx}\label{cor:par}
    Under the hypotheses of Item \ref{thm:C_transverse} of Theorem \ref{Thm:par_Csigma}, we assume the existence of a $C^2$ autonomous Hamiltonian $H^\infty: \T^n  \times B  \to \R$ of the form~\eqref{def:H_aut}, such that  
\begin{equation}\label{hyp:cor_par_decay}
\begin{aligned}
      &\int_1^{+\infty} (t-1) \sup_{q \in \T^n}|\partial_x\partial_pH^t(q,0,0) - \partial_x\partial_pH^\infty(q,0,0)| \, dt < \infty,\\      
      &\int_1^{+\infty} \sup_{q \in \T^n}|\partial_y\partial_pH^t(q,0,0) - \partial_y\partial_pH^\infty(q,0,0)| \, dt < \infty.
\end{aligned}        
\end{equation}
Let $\varphi: \T^n \times [T'', +\infty) \to \T^n \times \R^{n + 2m}$ be the $C^\sigma$ asymptotically elliptic KAM torus associated with $(X_H, X_{H_0}, \varphi_0)$ given by Theorem \ref{Thm:par_Csigma}, $S$ be the map defined in Definition \ref{def:ell_hyp_par_trasv_dyn_asym_KAM} and $\varphi_0$ be the trivial embedding in~\eqref{def:varphi0=(q,0,0)}. 
We consider the following systems 
\begin{align}\label{cor:syst_H_par}
    \dot \eta(t) = DX_H^t \circ \varphi^t (q + \omega t) \eta(t), & \qquad (q, t) \in \T^n \times [T'', +\infty), \\ \label{cor:syst_Hinfty_par}
    \dot \xi^\infty(t) = DX_{H^\infty}^t\circ \varphi_0(q + \omega t) \xi^\infty(t), & \qquad (q, t) \in \T^n \times [T'', +\infty).
\end{align}
    Then, if $\ell >3$, for every fixed $q\in\T^n$, $t_0\in I_{T^{''}}$ and $v_0 \in \R^n$ there exists a one-to-one correspondence between the sets 
     $$\left\{ \pi_z\eta \mid \eta \text{ is a solution  of ~\eqref{cor:syst_H_par}} \text{ with } \pi_p\eta(t_0) = v_0\right\},$$
    $$\left\{ \pi_z\xi^\infty \mid \xi^\infty \text{ is a solution  of ~\eqref{cor:syst_Hinfty_par}} \text{ with } \pi_p \xi^\infty(t_0) = \Pi_p S^{t_0}(q +\omega t_0)v_0\right\},$$ by the relation 
       \begin{equation*}
          \lim_{t \to +\infty}|\pi_z\eta(t)-\pi_z\xi^\infty(t)|=0.
    \end{equation*}
   We refer to~\eqref{def:Pi_p} for the definition of the projection $\Pi_p$.
\end{corollaryx}

\subsubsection{Smooth setting}\label{sec:Smooth_Set} In this section, we show that Theorems \ref{Thm:par_Csigma_0} and \ref{Thm:par_Csigma} also hold when considering $C^\infty$ perturbations with appropriate decrease rates. More precisely, if we denote
\[ \Spol_{(\infty,k),\ell} = \bigcap_{\sigma \geq 1} \Spol_{(\sigma,k),\ell},\]
for any $k \in \Z \geq 0$, any $T \geq 1$ and any $\ell > 0$, we have the following.

\begin{theoremx}
\label{thm:KAM_smooth}
    Theorems \ref{Thm:par_Csigma_0} and \ref{Thm:par_Csigma} hold for $\sigma = +\infty$.
\end{theoremx}

Notice that although for $C^\infty$ perturbations Theorems \ref{Thm:par_Csigma_0} and \ref{Thm:par_Csigma} yield immediately the existence of $C^\sigma$ asymptotic KAM tori, for $\sigma \geq 1$ arbitrarily large, we cannot deduce right away Theorem \ref{thm:KAM_smooth} since the domain of definition of these asymptotic KAM tori may depend on $\sigma$. The proof of Theorem \ref{thm:KAM_smooth} follows verbatim the proof of the analogous result in \cite{scarcella_KAMST} for the elliptic and hyperbolic settings. For the sake of convinience, we reproduce it here. %

\begin{proof}[Proof of Theorem \ref{thm:KAM_smooth}]
Let us prove that Theorem \ref{Thm:par_Csigma_0} holds for $\sigma = +\infty$. The proof in the case of Theorem \ref{Thm:par_Csigma} is completely analogous.

Notice that, since the first part of Theorem \ref{Thm:par_Csigma_0} holds for $C^\infty$ perturbations with $\sigma = 2$ (satisfying the corresponding hypotheses) and as the assumptions for the second part of Theorem \ref{Thm:par_Csigma_0} imply those in the first part, it suffices to show that the $C^2$ asymptotic KAM torus obtained in the first part of Theorem \ref{Thm:par_Csigma_0} with $\sigma = 2$ is in fact a $C^\infty$ asymptotic KAM torus.

Let $P$ be a $C^\infty$ perturbation satisfying the hypotheses of the first part of Theorem \ref{Thm:par_Csigma_0}, for any $\sigma \geq 2$, and let us denote by $\varphi_\sigma: \T^n \times [T_\sigma, +\infty) \to \T^n \times B$ the $C^\sigma$ asymptotic KAM torus given by the theorem. We may assume WLOG that $\sigma \mapsto T_\sigma$ is a non-decreasing function.

We will now fix $\sigma \geq 2$ and show that $\varphi_2$ is in fact a $C^\sigma$ asympotic KAM torus. For this, it suffices to show that $\varphi_2$ satisfies \eqref{def:cond1_asymKAMtorus} and that $\varphi_2^t$ is of class $C^\sigma$, for any $t \in [T_2, +\infty)$.

By the uniqueness of the asymptotic KAM tori, it follows that
\[ \varphi_{2}\mid_{\T^n \times [T_\sigma, +\infty)} = \varphi_\sigma,\]
and thus $\varphi_2$ satisfies \eqref{def:cond1_asymKAMtorus}. Recalling that $\varphi_2$ satisfies \eqref{eq:asymptotic_KAM_flow}, the equation above yields
\[\varphi_2^t(q) = \Phi^t_{T_\sigma, X_H} \circ \varphi_\sigma^{T_\sigma}(q - \omega (t - T_\sigma)), \qquad \text{ for any } q \in \T^n \text{ and any } t \in [T_2, +\infty). \]

Hence, since $\Phi^t_{T_\sigma, X_H}$ is of class $C^\infty$, for any $t \in [T_2, +\infty)$ fixed, and $\varphi_\sigma^{T_\sigma}$ is of class $C^\sigma$, it follows that $\varphi_2^t$ is of class $C^\sigma$, for any $t \in [T_2, +\infty)$.

As $\sigma \geq 2$ was arbitrary, this shows that $\varphi_2$ is a $C^\infty$ asymptotic KAM torus.
\end{proof}

\subsubsection{Analytic setting}\label{sec:Anal_Set} This section contains the real analytic version of the results established in Section \ref{sec:Holder_Set}. To this end, for some $\sigma >0$, we introduce the following complex domains
\begin{equation*}
    \T_\sigma^n = \{q \in \C^n /\Z^n \hspace{1mm}:\hspace{1mm} |\mathrm{Im}\, q| \le \sigma\}, \quad B_\sigma =\{(p,z) \in \C^{n+2m}\hspace{1mm}:\hspace{1mm} |(p,z)| \le \sigma\}.
\end{equation*}
Given $f:\T^n_\sigma \times B_\sigma \to \C$, we consider the following norm
\begin{equation}\label{def:norm_analy}
    |f|_\sigma = \sup_{(q,p,z) \in \T^n_\sigma \times B_\sigma}|f(q,p,z)|.
\end{equation}
We will use the same notation for functions defined only on $\T^n_\sigma$, as well as for vector-valued and matrix-valued functions.

Below, we provide the definition of analytic asymptotic KAM torus, which is the analytic version of Definition \ref{def:Csigma_KAM_torus}.
Given $\sigma >0$ and $\omega \in \R^n$,  we consider time-dependent vector fields $X^t$ and $X_0^t$ real-analytic on $\T^n_\sigma \times B_\sigma$, for all fixed $t \in I_T$, and a real-analytic embedding $\varphi_0:\T^n_\sigma \to \T^n_\sigma \times B_\sigma$ such that
\begin{equation}
\begin{aligned}\label{def:hyp_asym_KAM_tori_analy}
    &\lim_{t \to +\infty} \left|X^t - X^t_0\right|_{\sigma} = 0,\\
   &X_0 \circ \varphi_0 = \partial_q\varphi_0 \cdot \omega.
\end{aligned}
\end{equation}

\begin{definition}\label{def:analytic_KAM_torus}
    We assume that $(X, X_0, \varphi_0)$, defined on $\T^n_\sigma \times B_\sigma \times I_T$, satisfy~\eqref{def:hyp_asym_KAM_tori_analy}. A family of embeddings $\varphi:\T^n \times I_{T'} \to \T^n \times B$, with $T' \geq T$, is an analytic \emph{asymptotic KAM torus} associated with $(X, X_0, \varphi_0)$, if for some $0<\sigma'< \sigma$
    \begin{align}\label{def:cond1_asymKAMtorus_analy}
            &\lim_{t \to +\infty}\left|\varphi^t -\varphi_0\right|_{\sigma'}=0,\\ \label{def:cond2_asymKAMtorus_analy}
             &X \circ \varphi = \partial_q\varphi \cdot \omega + \partial_t \varphi.
    \end{align}
\end{definition}

 Given $\sigma >0$, $T\ge 1$, and $\ell \ge 0$, to describe the regularity of time-dependent analytic functions decaying polynomially fast in time, we introduce the following Banach spaces
\begin{align}\label{def:S_pol_anal}
\mathscr{S}^{ T}_{\sigma,\ell}
    = \left\{
    f : \T^n_\sigma \times B_\sigma \times I_T \to \C \;\middle|\;
    \begin{aligned}
        &f^t \mbox{ is real-analytic on $\T^n_\sigma \times B_\sigma$ for each }  t \in I_T; \\
        & \sup_{t \in I_T} \bigl(|f^t|_{\sigma}\, t^\ell \bigr) < \infty; \\
        & f \in C(\T^n_\sigma \times B_\sigma \times I_T)
    \end{aligned}
    \right\}
\end{align}
endowed with the norm
\begin{equation}\label{def:norm_S_analy}
    |f|^T_{\sigma, \ell} = \sup_{t \in I_T}|f^t|_{\sigma}t^\ell.
\end{equation}
We will use the same notation for functions defined on $\T^n_\sigma \times I_T$, as well as vector-valued or matrix-valued functions. It will be specified by the context.

\begin{theoremxx}\label{Thm:par_Csigma_0_analy}
    Fix $\sigma > 0$, $\ell >1$ and $l \ge 0$. Let $H_0 : \T^n_\sigma \times B_\sigma \to \C$ be real analytic of the form  \eqref{eq:initial_hamiltonian_par_deg} satisfying $\partial_{(p,z)}^2 H_0 \in \mathscr{S}_{(\sigma,3),0}^{1},$ and let $P : \T^n_\sigma \times B_\sigma \to \C$ be real analytic of the form  \eqref{eq:perturbation_form_par_deg}. Then, for $H := H_0 + P$ the following holds.
\begin{enumerate}
    \item \label{thm:Cdeg_existence_analy} Suppose that
    \begin{equation*}
        a \in \mathscr{S}^{1}_{\sigma,0}, \hspace{2mm} \partial_qa \in \mathscr{S}_{\sigma, \ell+ \dec + 2}^{1}, \hspace{2mm}  b \in \mathscr{S}_{\sigma, \ell}^{1}, \hspace{2mm}  c \in \mathscr{S}_{\sigma, \ell+ \dec + 1}^{ 1}, \hspace{2mm} d \in \mathscr{S}_{\sigma, \ell}^{ 1}.
    \end{equation*}
      Then, there exists an analytic asymptotic KAM torus $\varphi$ associated with $(X_H, X_{H_0}, \varphi_0)$ of the form
    \begin{equation}
        \label{eq:asymp_KAM_torus_form_par0_analy}
    \begin{gathered}
        \varphi:\T^n_{\sigma/2} \times I_T \to \T^n_\sigma \times \C^n \times \C^{2m}, \qquad \varphi^t = (\textup{id}_{\T^n} + u^t, v^t, w^t), \qquad T \geq 0, \\
     u \in \mathscr{S}_{{\sigma \over 2}, \ell-1}^{T}, \quad  v \in \mathscr{S}_{{\sigma \over 2}, \ell+ \dec + 1}^{T}, \quad w \in \mathscr{S}_{{\sigma \over 2}, \ell+ \dec}^{T}.
    \end{gathered}
        \end{equation}
        Moreover, any other analytic asymptotic KAM torus associated with $(X_H, X_{H_0}, \varphi_0)$ of the form \eqref{eq:asymp_KAM_torus_form_par0_analy} coincides with $\varphi$ in the intersection of their domains.
    \medskip
    
    \item \label{thm:Cdeg_transverse_analy}  Suppose that
    \begin{equation}
    \label{eq:decay_pardeg_2_analy}
        a \in \mathscr{S}^{1}_{\sigma,0}, \hspace{2mm}   \partial_qa \in \mathscr{S}_{\sigma, \ell+ \dec + 3}^{ 1}, \hspace{2mm}  b \in \mathscr{S}_{\sigma, \ell}^{ 1}, \hspace{2mm}  c \in \mathscr{S}_{\sigma, \ell+\dec+2}^{ 1}, \hspace{2mm}  d \in \mathscr{S}_{\sigma, \ell+\dec+1}^{ 1}.
    \end{equation}
    Then the analytic asymptotic KAM torus $\varphi$ associated with $(X_H, X_{H_0}, \varphi_0)$ given by the first statement is asymptotically parabolic.
\end{enumerate}
\end{theoremxx}

\begin{theoremxx}\label{Thm:par_analy}
Fix $\sigma > 0$, $\ell >1$ and $\dec \ge 0$. Let $H_0 : \T^n_\sigma \times B_\sigma \to \C$ be real analytic of the form \eqref{eq:initial_hamiltonian_par} satisfying $\partial_{(p,x,y)}^2 H_0 \in \mathscr{S}_{\sigma,0}^{1}$ and let $P : \T^n_\sigma \times B_\sigma \to \C$ be real analytic of the form \eqref{eq:perturbation_form_par}. Then, for $H:= H_0 + P$ the following holds.
\begin{enumerate}
    \item \label{thm:C_existence_analy} Suppose that
    \begin{equation*}
        a \in \mathscr{S}^{ 1}_{\sigma,0}, \quad \partial_q a \in \mathscr{S}_{\sigma, \ell+ \dec+4}^{1}, \quad b \in \mathscr{S}_{\sigma, \ell}^{1}, \quad  c_1 \in \mathscr{S}_{\sigma, \ell+ \dec+2}^{ 1}, \quad  c_2 \in \mathscr{S}_{\sigma, \ell+ \dec+3}^{ 1}, 
        \end{equation*}
        \begin{equation*}
         d_1 \in \mathscr{S}_{\sigma, \ell-1}^{ 1}, \quad d_2 \in \mathscr{S}_{\sigma, \ell}^{ 1}, \quad d_3 \in \mathscr{S}_{\sigma, \ell+1}^{ 1}.
    \end{equation*}
    Then, there exists an analytic asymptotic KAM torus $\varphi$ associated with $(X_H, X_{H_0}, \varphi_0)$ of the form
    \begin{equation}
        \label{eq:asymp_KAM_torus_form_par_analy}
    \begin{gathered}
        \varphi:\T^n_{\sigma/2} \times I_T \to \T^n_\sigma \times \C^n \times \C^{2m}, \qquad \varphi^t = (\textup{id}_{\T^n} + u^t, v^t, w^t), \qquad T \geq 0, \\
            u \in \mathscr{S}_{{\sigma \over 2}, \ell-1}^{T}, \quad  v \in \mathscr{S}_{{\sigma \over 2}, \ell+ \dec + 3}^{T}, \quad w_x \in \mathscr{S}_{{\sigma \over 2}, \ell+ \dec + 2}^{T} , \quad w_y \in \mathscr{S}_{{\sigma \over 2}, \ell+ \dec + 1}^{T}.
    \end{gathered}
        \end{equation}
        Moreover, any other analytic asymptotic KAM torus associated with $(X_H, X_{H_0}, \varphi_0)$ of the form \eqref{eq:asymp_KAM_torus_form_par_analy} coincides with $\varphi$ in the intersection of their domains.
    \medskip
    
    \item \label{thm:C_transverse_analy} Suppose that
    \begin{equation} \label{eq:Hyp_decay_par_2_analy}
    \begin{aligned}
        a \in \mathscr{S}^{ 1}_{\sigma,0}, \quad \partial_q a &\in \mathscr{S}_{\sigma, \ell+\dec+6}^{1}, \quad b \in \mathscr{S}_{\sigma, \ell}^{1}, \quad  c_1 \in \mathscr{S}_{\sigma, \ell+\dec+4}^{ 1}, \quad  c_2 \in \mathscr{S}_{\sigma, \ell+\dec+5}^{ 1}, \\
         d_1 &\in \mathscr{S}_{\sigma, \ell+\dec+1}^{ 1}, \quad d_2 \in \mathscr{S}_{\sigma, \ell+\dec+2}^{ 1}, \quad d_3 \in \mathscr{S}_{\sigma, \ell+\dec+3}^{ 1}.
        \end{aligned} 
    \end{equation}
    Then the analytic asymptotic KAM torus $\varphi$ associated with $(X_H, X_{H_0}, \varphi_0)$ given by the first statement is asymptotically parabolic.
\end{enumerate}
\end{theoremxx}

\begin{remark}
It follows from the proofs of Theorems \ref{Thm:par_Csigma_0}, \ref{Thm:par_Csigma}, \ref{Thm:par_Csigma_0_analy} and \ref{Thm:par_analy} that we have precise control over the regularity of the conjugating map $S$ and the linear cocycle $A$ in Definition \ref{def:ell_hyp_par_trasv_dyn_asym_KAM}. These objects are constructed in each proof, with the help of the Implicit Function Theorem, in order to determine the asymptotic behaviour of the transverse dynamics.

Since our main goal was to establish the asymptotic convergence, when restricted to the normal coordinates, of the solutions of the linearized cocycle associated with the perturbed system along the asymptotic KAM torus to those of the cocycle $A$ (which exhibits parabolic behaviour in the present setting), namely, Equation \eqref{def:trasv_dyn_limit}, we did not explicitly include these regularity properties in the statements of our theorems.

However, it follows directly from the proofs that, in Theorems \ref{Thm:par_Csigma_0} and \ref{Thm:par_Csigma}, the maps $A$ and $S$ are of class $C^{\sigma-1}$ with respect to the variables $(q, p, z)$ (while in Theorems \ref{Thm:par_Csigma_0_analy} and \ref{Thm:par_analy} they are analytic) and of class $C^1$ with respect to $t$.
\end{remark}

\section{Analysis of Transverse Dynamics}
\label{sc:analysis_transversal}

In this section, we study the solutions of the time-dependent linear system~\eqref{eq:Trasv_dyn_AA}.
Given $T\ge 1$ and letting $\xi = (\xi_1, \xi_2, \xi_3, \xi_4) \in \R^n \times \R^n \times \R^m \times \R^m$, we rewrite~\eqref{eq:Trasv_dyn_AA} as
\begin{equation}
    \begin{cases}\label{eq:Trasv_dyn_A_2}
        \partial_t \xi_1 (t) = a^t_{12}(q+\omega t) \xi_2(t) +  a^{t}_{13}(q+\omega t) \begin{pmatrix} \xi_3(t) \\ \xi_4(t) \end{pmatrix},\\
        \partial_t \xi_2(t) = 0,\\
        \begin{pmatrix}\partial_t \xi_3 (t) \\ \partial_t \xi_4(t) \end{pmatrix} = a^{t}_{32}(q+\omega t) \xi_2(t) + J_m \Mpar \begin{pmatrix}\xi_3(t) \\ \xi_4(t) \end{pmatrix},
    \end{cases}
\end{equation}
for $(q,t) \in \T^n \times I_{T}$, where we refer to~\eqref{def:J} for the definition of $J_m$, and to~\eqref{def:A} for the terms $a^t_{12}, a^t_{13}$, and $a^t_{32}$.

Throughout this section, we denote by $C(\cdot)$ a generic positive constant depending on the parameter in brackets and by
\begin{equation}\label{not:proj_Pixy}
    \Pi_x = (\mathrm{Id}_m \quad 0_{m}), \, \Pi_y = (0_m \quad  \mathrm{Id}_m) \in \mathcal{M}_{m \times 2m}(\R)
\end{equation}
the projections onto the first and last $m$ rows.

Each solution $\xi(t;q,t_0)$ of the above system depends on the parameters $q\in\T^n$ and $t_0\in I_{T}$, and the initial condition $\xi(t_0;q)$ depends on $q\in\T^n$. To avoid cumbersome notation, throughout this section we omit these dependencies and simply write $\xi(t)$ and $\xi(t_0)$ whenever the meaning is clear from the context.

This section is divided into two parts. First, in Section \ref{sec:A_par_deg}, we consider the degenerate case 
\begin{equation*}
    \Mpar = \begin{pmatrix} 0 & 0 \\ 0 & 0
    \end{pmatrix}
\end{equation*}
which correspond to $\Lambda_1,\dots, \Lambda_m=0$ (see~\eqref{Ms} for the definition of $\Mpar$). Then, in Section \ref{sec:A_par_non_deg}, we analyze the general case in which $\Lambda_1,\dots, \Lambda_m \ge 0$.

\subsection{Case $\Lambda_1,\dots, \Lambda_m=0$}\label{sec:A_par_deg}
We rewrite the system~\eqref{eq:Trasv_dyn_A_2} as
\begin{equation}
    \begin{cases}\label{eq:Trasv_dyn_A_2_par_deg}
        \partial_t \xi_1 (t) = a^t_{12}(q+\omega  t) \xi_2(t) +  a^{t}_{13}(q+\omega t ) \begin{pmatrix} \xi_3(t) \\ \xi_4(t) \end{pmatrix},\\
        \partial_t \xi_2(t) = 0,\\
        \begin{pmatrix}\partial_t \xi_3 (t) \\ \partial_t \xi_4(t) \end{pmatrix} = a^{t}_{32}(q+\omega t) \xi_2(t).
    \end{cases}
\end{equation}
for all $(q,t) \in \T^n \times I_T$. In the following proposition, we provide growth estimates for the solutions of the system~\eqref{eq:Trasv_dyn_A_2_par_deg}. In Proposition \ref{prop:par_deg_1:1_asym_sol} below, we assume that $a_{32}^t$ converge in time to an autonomous matrix $a^\infty_{32}$. We analyze when there is an asymptotic correspondence between the transversal solutions of~\eqref{eq:Trasv_dyn_A_2_par_deg} and those of the limiting system.
\begin{proposition}
\label{prop:trasv_dyn_par_deg_case} 
    We assume that the non-autonomous linear system~\eqref{eq:Trasv_dyn_A_2_par_deg} satisfies
    \begin{equation*}
        a_{12}, \, a_{13}, \, a_{32} \in C^0(\T^n \times I_{T}), \qquad \sup_{t \in I_{T}}|a_{12}^t|_{C^0}, \, \sup_{t \in I_{T}}|a_{13}^t|_{C^0}, \, \sup_{t \in I_{T}}|a_{32}^t|_{C^0},< \infty.
    \end{equation*}
    Then, for every fixed $(q,t_0) \in \T^n \times I_{T}$ and any initial condition $\xi(t_0) = (\xi_1(t_0),\xi_2(t_0),\xi_3(t_0),\xi_4(t_0))$ there exists a unique solution $\xi(t) = (\xi_1(t),\xi_2(t),\xi_3(t),\xi_4(t))$ and a positive constant $C$ depending on $|\xi(t_0)|$, $\sup_{t \in I_{T}}|a_{12}^t|_{C^0}, \, \sup_{t \in I_{T}}|a_{13}^t|_{C^0}, \, \sup_{t \in I_{T}}|a_{32}^t|_{C^0}$, $n$, $m$ and $t_0$ such that 
    \begin{equation*}
        |\xi_1(t)| \le C t^2, \quad |\xi_2(t)| \le C,  \quad |\xi_3(t)|, \,|\xi_4(t)| \le C t,
    \end{equation*}
    for all $t \in I_{T}$.    
\end{proposition}
\begin{proof}
    The result follows from direct integration of~\eqref{eq:Trasv_dyn_A_2_par_deg}.
\end{proof}
\begin{remark}
    In Theorem \ref{Thm:par_Csigma_0}, the unperturbed Hamiltonian $H_0$ in~\eqref{eq:initial_hamiltonian_par_deg} has an invariant torus in $p=z=0$ supporting quasiperiodic solutions. If we consider solutions of~\eqref{eq:Trasv_dyn_A_2_par_deg} with initial condition $\xi(t_0)$ such that $\xi_2(t_0) =0$
    which consists in studying the transverse dynamics in a neighborhood of the invariant torus corresponding to $p=0$, we obtain the following degenerate transverse motions
    \begin{equation*}
        \xi_3(t) = \xi_3(t_0), \qquad \xi_4(t) = \xi_4(t_0).
    \end{equation*}
\end{remark}
We now consider the following autonomous system for all $(q,t) \in \T^n \times I_T$
\begin{equation}
    \begin{cases}\label{eq:Trasv_dyn_A_2_par_deg_syst_infty}
        \partial_t \xi^\infty_1 (t) = a^\infty_{12}(q+\omega t) \xi^\infty_2(t) +  a^\infty_{13}(q+\omega t) \begin{pmatrix} \xi^\infty_3(t) \\ \xi^\infty_4(t) \end{pmatrix},\\
        \partial_t \xi^\infty_2(t) = 0,\\
        \begin{pmatrix}\partial_t \xi^\infty_3 (t) \\ \partial_t \xi^\infty_4(t) \end{pmatrix} = a^\infty_{32}(q+\omega t) \xi^\infty_2(t) ,
    \end{cases}
\end{equation}
 where $a^\infty_{12}:\T^n \to \mathcal{M}_n(\R)$, $a^\infty_{13} :\T^n \to \mathcal{M}_{n\times 2m}(\R)$, and  $a^\infty_{32}:\T^n \to \mathcal{M}_{2m\times n}(\R)$.

\begin{proposition}\label{prop:par_deg_1:1_asym_sol}
    We consider the systems~\eqref{eq:Trasv_dyn_A_2_par_deg} and~\eqref{eq:Trasv_dyn_A_2_par_deg_syst_infty} and we assume that 
    \begin{equation}\label{hyp:limit_aut_sys_par_deg}
        \int_{T}^{+\infty} |a_{32}^s - a_{32}^\infty|_{C^0} \, ds < \infty.
    \end{equation}
     Then, for every fixed $q\in\T^n$, $t_0\in I_{T}$ and $\xi_2^0 \in \R^n$ there exists a one-to-one correspondence between the sets $$\{ \pi_z\xi^\infty \mid \xi^\infty \text{ is a solution  of ~\eqref{eq:Trasv_dyn_A_2_par_deg_syst_infty}} \text{ with } \xi_2^\infty(t_0) = \xi_2^0\},$$ $$\{ \pi_z\xi \mid \xi \text{ is a solution  of ~\eqref{eq:Trasv_dyn_A_2_par_deg}} \text{ with } \xi_2(t_0) = \xi_2^0\},$$ by the relation 
       \begin{equation*}
        \lim_{t \to +\infty}|\pi_z\xi^\infty(t)-\pi_z\xi(t)|=0.
    \end{equation*}
\end{proposition}
\begin{proof}
    We introduce the following notation
    \begin{equation*}
        \zeta(t) = \left(\xi_3(t),  \xi_4(t)\right)^\top, \qquad \zeta_\pm^\infty(t) = \left(\xi^\infty_3(t),  \xi^\infty_4(t)\right)^\top.
    \end{equation*}
    Thanks to~\eqref{hyp:limit_aut_sys_par_deg} we have that $$\int_{t_0}^{+\infty}\left(a_{32}^s(q+\omega t) - a_{32}^\infty(q+\omega t)\right) \xi_2(t_0) \, ds$$ is well defined. 
    Remembering that we are assuming $\xi_2(t_0) = \xi_2^\infty(t_0)$, by integration, we have that the difference of the solutions of the last equation of systems~\eqref{eq:Trasv_dyn_A_2_par_deg} and~\eqref{eq:Trasv_dyn_A_2_par_deg_syst_infty} is given by
    \begin{align*}
        \zeta(t) - \zeta^\infty(t) &= \zeta(t_0) - \zeta^\infty(t_0) + \int_{t_0}^{+\infty}\left(a_{32}^s(q+\omega t) - a_{32}^\infty(q+\omega t)\right) \xi_2(t_0) \, ds\\
        &- \int_{t}^{+\infty}\left(a_{32}^s(q+\omega t) - a_{32}^\infty(q+\omega t)\right) \xi_2(t_0) \, ds.
    \end{align*}
    Finally, using~\eqref{hyp:limit_aut_sys_par_deg} we have that 
    \begin{equation*}
        \lim_{t \to +\infty}|\zeta(t) - \zeta^\infty(t)|=0 \quad \mbox{if and only if} \quad \zeta(t_0) = \zeta^\infty(t_0) - \int_{t_0}^{+\infty}\left(a_{32}^s(q+\omega t) - a_{32}^\infty(q+\omega t)\right) \xi_2(t_0) \, ds.
    \end{equation*}
    This concludes the proof of this Proposition. 
\end{proof}

\subsection{Case $\Lambda_1,\dots, \Lambda_m\ge0$}\label{sec:A_par_non_deg}
We rewrite the system~\eqref{eq:Trasv_dyn_A_2} as
\begin{equation}
    \begin{cases}\label{eq:Trasv_dyn_A_2_par}
        \partial_t \xi_1 (t) = a^t_{12}(q+\omega  t) \xi_2(t) +  a^{t}_{13}(q+\omega t ) \begin{pmatrix} \xi_3(t) \\ \xi_4(t) \end{pmatrix},\\
        \partial_t \xi_2(t) = 0,\\
        \partial_t \xi_3 (t)  = \Pi_x a^{t}_{32}(q+\omega t) \xi_2(t), \\
        \partial_t \xi_4(t) = \Pi_y a^{t}_{32}(q+\omega t) \xi_2(t) - \Lambda \xi_3(t),
    \end{cases}
\end{equation}
for all $(q,t) \in \T^n \times I_T$, where, we refer to~\eqref{not:proj_Pixy} for the definition of the projections $\Pi_x$ and $\Pi_y$.  As in Section \ref{sec:A_par_deg}, we prove two results that are the counterparts of Propositions~\ref{prop:trasv_dyn_par_deg_case} and~\ref{prop:par_deg_1:1_asym_sol} for the present setting. The first concerns growth estimates for the solutions of~\eqref{eq:Trasv_dyn_A_2_par_deg}, whereas the second establishes their asymptotic correspondence with the solutions of a suitable limiting autonomous system.
\begin{proposition}
\label{prop:trasv_dyn_par}
    We assume that the non-autonomous linear system~\eqref{eq:Trasv_dyn_A_2_par} satisfies
    \begin{equation*}
        a_{12}, \, a_{13}, \, a_{32} \in C^0(\T^n \times I_{T}), \qquad \sup_{t \in I_{T}}|a_{12}^t|_{C^0}, \, \sup_{t \in I_{T}}|a_{13}^t|_{C^0}, \, \sup_{t \in I_{T}}|a_{32}^t|_{C^0},< \infty.
    \end{equation*}
    Then, for every fixed $(q,t_0) \in \T^n \times I_{T}$ and any initial condition $\xi(t_0) = (\xi_1(t_0),\xi_2(t_0),\xi_3(t_0),\xi_4(t_0))$ there exists a unique solution $\xi(t) = (\xi_1(t),\xi_2(t),\xi_3(t),\xi_4(t))$ and a positive constant $C$ depending on $\sup_{t \in I_{T}}|a_{12}^t|_{C^0}, \, \sup_{t \in I_{T}}|a_{13}^t|_{C^0}, \, \sup_{t \in I_{T}}|a_{32}^t|_{C^0}$, $n$, $m$, $\Lambda$ and $t_0$ such that 
    \begin{equation*}
        |\xi_1(t)| \le C t^3, \quad |\xi_2(t)| \le C,  \quad |\xi_3(t)|\le Ct, \quad |\xi_4(t)| \le C t^2,
    \end{equation*}
    for all $t \in I_{T}$.    
\end{proposition}
\begin{proof}
    The result follows from direct integration of~\eqref{eq:Trasv_dyn_A_2_par}.
\end{proof}
\begin{remark}
In this case, when we study solution of~\eqref{eq:Trasv_dyn_A_2_par} with initial condition $\xi(t_0)$ such that $\xi_2(t_0) =0$, we obtain the following transverse parabolic dynamics
\begin{equation*}
    \xi_3(t) = \xi_3(t_0), \qquad \xi_4(t) = \xi_4(t_0) - \Lambda \xi_3(t_0)(t-t_0).
\end{equation*}
\end{remark}

We consider the following autonomous system for all $(q,t) \in \T^n \times I_T$
\begin{equation}
    \begin{cases}\label{eq:Trasv_dyn_A_2_par_syst_infty}
        \partial_t \xi^\infty_1 (t) = a^\infty_{12}(q+\omega t) \xi^\infty_2(t) +  a^\infty_{13}(q+\omega t) \begin{pmatrix} \xi^\infty_3(t) \\ \xi^\infty_4(t) \end{pmatrix},\\
        \partial_t \xi^\infty_2(t) = 0,\\
        \partial_t \xi^\infty_3 (t)  = \Pi_x a^{\infty}_{32}(q+\omega t) \xi^\infty_2(t), \\
        \partial_t \xi^\infty_4(t) = \Pi_y a^{\infty}_{32}(q+\omega t) \xi^\infty_2(t) - \Lambda \xi^\infty_3(t)
    \end{cases}
\end{equation}
 where $a^\infty_{12}:\T^n \to \mathcal{M}_n(\R)$, $a^\infty_{13} :\T^n \to \mathcal{M}_{n\times 2m}(\R)$, and  $a^\infty_{32}:\T^n \to \mathcal{M}_{2m\times n}(\R)$. 
 \begin{proposition}\label{prop:par_1:1_asym_sol}
    We consider the systems~\eqref{eq:Trasv_dyn_A_2_par} and~\eqref{eq:Trasv_dyn_A_2_par_syst_infty} and we assume that 
    \begin{equation}\label{hyp:limit_aut_sys_par}
        \int_{T}^{+\infty} (s - T) |\Pi_x\left(a_{32}^s - a_{32}^\infty\right)|_{C^0} \, ds < \infty, \quad \int_{T}^{+\infty} |\Pi_y\left(a_{32}^s - a_{32}^\infty\right)|_{C^0} \, ds < \infty
    \end{equation}
    Then, for every fixed $q\in\T^n$, $t_0\in I_{T}$ and $\xi_2^0 \in \R^n$ there exists a one-to-one correspondence between the sets $$\{ \pi_z\xi^\infty \mid \xi^\infty \text{ is a solution  of ~\eqref{eq:Trasv_dyn_A_2_par_syst_infty}} \text{ with } \xi_2^\infty(t_0) = \xi_2^0\},$$ $$\{ \pi_z\xi \mid \xi \text{ is a solution  of ~\eqref{eq:Trasv_dyn_A_2_par}} \text{ with } \xi_2(t_0) = \xi_2^0\},$$ by the relation 
       \begin{equation*}
        \lim_{t \to +\infty}|\pi_z\xi^\infty(t)-\pi_z\xi(t)|=0.
    \end{equation*}
\end{proposition}
\begin{proof}
    For all $t \in I_T$, we introduce the following notation
    \begin{align*}
        &R_x^t(q +\omega t) = \Pi_x (a_{32}^t(q+\omega t) - a_{32}^\infty(q+\omega t))\xi_2(t_0),\\
        &R_y^t(q +\omega t) = \Pi_y\left(a_{32}^t(q+\omega t) - a_{32}^\infty(q+\omega t)\right)\xi_2(t_0)
    \end{align*}
     and we recall that we are assuming $\xi_2(t_0) = \xi_2^\infty(t_0)$. Hypothesis~\eqref{hyp:limit_aut_sys_par} implies
     \begin{equation*}
         \int_{t_0}^{+\infty} |R^s_x|_{C^0} \, ds < \infty.
     \end{equation*}
     Hence,
     as in the proof of Proposition \ref{prop:par_deg_1:1_asym_sol}, one can verify that 
    \begin{align*}
        \lim_{t \to +\infty}|\xi_3(t) - \xi_3^\infty(t)| = 0  \quad \mbox{if and only if} \quad \xi_3(t_0) = \xi_3^\infty(t_0) + \int_{t_0}^{+\infty} R^s_x (q+\omega s) \, ds.
    \end{align*}
    It remains to analyze the last equation of systems~\eqref{eq:Trasv_dyn_A_2_par} and~\eqref{eq:Trasv_dyn_A_2_par_syst_infty} where now $\xi_3(t)$ and $\xi_3^\infty(t)$ are known and satisfy $$\xi_3(t_0) = \xi_3^\infty(t_0) + \int_{t_0}^{+\infty} R^s_x (q+\omega s) \, ds.$$ To this end, we observe that, by hypothesis~\eqref{hyp:limit_aut_sys_par}
    \begin{align*}
        \int_{T}^{+\infty} (s - T) |R^s_x|_{C^0} \, ds < \infty, \quad \int_{T}^{+\infty} |R^s_y|_{C^0} \, ds < \infty
    \end{align*}
and hence, by integration, we have that the difference of the solutions of the last equations of systems~\eqref {eq:Trasv_dyn_A_2_par} and~\eqref{eq:Trasv_dyn_A_2_par_syst_infty} is equal to
\begin{align*}
    \xi_4(t) - \xi_4^\infty(t) &= \xi_4(t_0) - \xi_4^\infty(t_0) + \int_{t_0}^{+\infty} (s-t_0)\Lambda R^s_x(q +\omega s)\, ds + \int_{t_0}^{+\infty} R^s_y(q+\omega s) \, ds\\
    &-\int_{t}^{+\infty} (s-t)\Lambda R^s_x(q +\omega s)\, ds - \int_{t}^{+\infty} R^s_y(q+\omega s) \, ds. 
\end{align*}
    Thanks to the latter and~\eqref{hyp:limit_aut_sys_par} we obtain that 
    \begin{align*}
        &\lim_{t \to +\infty}|\xi_4(t) - \xi_4^\infty(t)| = 0 \qquad \mbox{if and only if}\\
        &\xi_4(t_0) = \xi_4^\infty(t_0) - \int_{t_0}^{+\infty} (s-t_0)\Lambda R^s_x(q +\omega s)\, ds - \int_{t_0}^{+\infty} R^s_y(q+\omega s) \, ds.
    \end{align*}
    This concludes the proof of this proposition. 
\end{proof}

\section{A criterion for asymptotically parabolic transverse dynamics}\label{sc:criteria_asym_dyn}

 To simplify the notation, in the following, given $\omega \in \R^n$ we will denote by $\Lo$ the directional derivative
\begin{equation}
    \label{def:LoO}
 \Lo  = \partial_t + \omega \cdot \partial_q = \partial_t + \sum_{i = 1}^n \omega_i\partial_{q_i}.
\end{equation}
It is not difficult to check that \eqref{eq:conjugated_cocycles} in Definition \ref{def:ell_hyp_par_trasv_dyn_asym_KAM} is equivalent to
\begin{equation}
    \label{eq:conjugated_cocycles_infinitesimal}
    A(q, t) =  S(q, t)^{-1}(DX \circ \varphi^t (q)  S(q, t)- \Lo S(q, t)).
\end{equation}

In what follows, we denote
\begin{equation}\label{def:Proj_Pi_xy}
    \Pi_{xy} = \begin{pmatrix} 0_{2n\times 2m} \\ \mathrm{Id}_{2m}\end{pmatrix}.
\end{equation}    

From \cite{scarcella_KAMST}, we have the following.

\begin{proposition}
\label{prop:precriterion}
Fix $n, m \geq 1$, $T > 0$, $\sigma \geq 2$, $\omega \in \R^n$.  Let $H_0: \T^n \times B \times I_1  \to \R$ given by
\[ H_0(q, p, z, t) = \omega \cdot p +O^2(p, z), \qquad \partial_{(p,z)}^2H_0 \in \mathcal{S}^{ 1}_{(\sigma, 2), 0},\]
where $B \subseteq \R^{n + 2m}$ is an open ball centered at the origin. Let $H: \T^n \times B \times I_1  \to \R$ such that $H^t$ is of class $C^2$, for each $t \in I_T$, and 
\[ \lim_{t \to +\infty} |X_{H}^t - X_{H_0}^t|_{C^1} = 0. \]
Let $\varphi = (\mathrm{id} + u, v, w): \T^n \times I_T \to \T^n \times \R^n \times \R^{2m}$ be a $C^\sigma$ asymptotic KAM torus associated with $(X_H, X_{H_0}, \varphi_0)$ and $G: \T^n \times I_T \to \mathfrak{sp}(2, \R)$ continuous and uniformly bounded.%

Then, the following holds.
\begin{enumerate}
    \item $\boldsymbol{S}(\varphi, G) : \T^n \times I_T \to \mathrm{Sp}(2n + 2m,\R, J)$ given by
    \begin{equation}
\label{eq:S_formula}
\boldsymbol{S}(\varphi, G)  = \left(
\begin{array}{ccc}
\mathrm{Id}_n +\partial_q  u & 0 & 0 \\
\partial_q v & (\mathrm{Id}_n +\partial_q  u )^{-\top} & -(\mathrm{Id}_n +\partial_q  u )^{-\top} \partial_q w^\top J_m e^G\\
\partial_q w & 0 & e^G
\end{array}
\right),
\end{equation}
is a well-defined continuous map.
\item  If $\boldsymbol{S}(\varphi, G)$ is of class $C^1$ then $\boldsymbol{A}(H, \varphi, G): \T^n \times I_T \to \mathfrak{sp}(2n + 2m, \R)$ given by %
\begin{equation}
\label{eq:def_A_Hamiltonian}
\boldsymbol{A}(H, \varphi, G) = \boldsymbol{S}(\varphi, G)^{-1}\big( (DX_H \circ \varphi )\boldsymbol{S}(\varphi, G) - \Lo \boldsymbol{S}(\varphi, G)),
\end{equation}
is a well-defined uniformly bounded continuous map of the form
     \begin{equation}
     \label{eq:form_A}
        \boldsymbol{A}(H, \varphi, G) =\begin{pmatrix} 0_{n \times n} & * & * \\
                                0_{n \times n} & 0_{n \times n} & 0_{n \times 2m} \\
                                0_{2m \times n} & * & J_m \boldsymbol{\zeta}(H, \varphi, G)\end{pmatrix},
    \end{equation}
    where $\boldsymbol{\zeta}(H, \varphi, G): \T^n \times I_{T'} \to \mathrm{Sym}(2m, \R)$ is given by
\begin{equation}
    \label{eq:formula_zeta}
    \boldsymbol{\zeta}(H, \varphi, G) = \Pi_{xy}^\top \boldsymbol{S}(\varphi, G)^\top\big( (D^2H \circ \varphi)  \boldsymbol{S}(\varphi, G) + J \Lo \boldsymbol{S}(\varphi, G)\big) \Pi_{xy}.
\end{equation}
\end{enumerate}
\end{proposition}

In particular, for $H, \varphi, G$ as in Proposition \ref{prop:precriterion} such that $\partial_q \varphi$ and $G$ are of class $C^1$, the cocycles given by the fundamental solutions of $\boldsymbol{A}(H, \varphi, G), DX_H \circ \varphi : \T^n \times I_T \to \mathcal{M}_{2n + 2n}(\R)$ over $R_\omega: \T^n \times I_T \to \T^n$, where $R_\omega(q, t) = (q + \omega t, t)$, are conjugated by $\boldsymbol{S}(\varphi, G)$.

\begin{remark}
\label{rmk:formula_zeta}
    The map $\boldsymbol{\zeta}(H, \varphi, G)$ given by \eqref{eq:formula_zeta} can be expressed explicitly as 
    \[\zeta = B^\top \partial_p^2 H\circ \varphi B + K^{\top} \partial_p\partial_zH\circ \varphi B + B^{\top} (\partial_p\partial_zH\circ \varphi)^\top K + K^{\top} \partial_z^2 H\circ \varphi K + K^{\top} J_m \Lo K,\]
where $K = \exp(G)$,  $B = -(\mathrm{Id}_{n} +\partial_q  u )^{-T} \partial_q w^\top J_m K$ and $\partial_p\partial_zH\circ \varphi$ stands for the $2m \times n$ matrix having components $\partial_{p_i}\partial_{z_j}H\circ \varphi$ for $1 \le i \le n$ and $1 \le j \le 2m$.
\end{remark}

Proposition \ref{prop:precriterion} and the results in Section \ref{sc:analysis_transversal} yield the following.

\begin{corollary}
\label{cor:criterion}
Let $H_0, H, \varphi = (\mathrm{Id} + u, v, w)$ and $G$ as in Proposition \ref{prop:precriterion}. Suppose that  $\boldsymbol{S}(\varphi, G)$ as in \eqref{eq:S_formula} is of class $C^1$ and let $
    \boldsymbol{\zeta}(H, \varphi, G)$ be given by \eqref{eq:formula_zeta}. %
Then the following holds.
\begin{enumerate}
\item If $\boldsymbol{\zeta}(H, \varphi, G) = 0_{2m \times 2m}$ and 
                \[  w \in \mathcal{S}^{ T}_{(\sigma, 0), \ell + 1}, \qquad G \in \mathcal{S}^{ T'}_{(\sigma - 1, 0), \ell},\]
                     then $\varphi$ is an asymptotically parabolic (degenerate) KAM torus.
\item If $\boldsymbol{\zeta}(H, \varphi, G) = \Mpar$  and 
                \[  w \in \mathcal{S}^{ T}_{(\sigma, 0), \ell + 2}, \qquad G = \begin{pmatrix}
                    G_{xx} & G_{xy} \\ G_{yx} & G_{yy}
                \end{pmatrix}, \qquad G_{xx}, G_{xy}, G_{yx}, G_{yy} \in \R^{m \times m}, \]
                \[ G_{xx}, G_{yy} \in \mathcal{S}^{ T'}_{(\sigma - 1, 0), \ell + 1}, \qquad G_{xy} \in \mathcal{S}^{ T'}_{(\sigma - 1, 0), \ell + 2}, \qquad G_{yx} \in \mathcal{S}^{ T'}_{(\sigma - 1, 0), \ell},\]
         then $\varphi$ is an asymptotically parabolic KAM torus.
\end{enumerate}
\end{corollary}

\begin{proof}
    Let  $A = \boldsymbol{A}(H, \varphi, G): \T^n \times I_{T} \to \mathfrak{sp}(2n + 2m, \R)$ as in Proposition \ref{prop:precriterion}, that is,  given by \eqref{eq:def_A_Hamiltonian}. Let us check that, in each case, $A$ and $S$ fulfill the conditions in Definition \ref{def:ell_hyp_par_trasv_dyn_asym_KAM}.
    
    By Proposition \ref{prop:precriterion}, $A$ is of the form \eqref{def:A} and satisfies  \eqref{def:cond1_ellhyppar_asymKAMtorus}. Moreover, by \eqref{eq:def_A_Hamiltonian},  \eqref{eq:conjugated_cocycles} holds. Thus, it suffices to check that in each case the assumptions on $w$ and $G$ imply \eqref{def:trasv_dyn_limit}. 
    
By \eqref{eq:conjugated_cocycles}, we have
    \begin{align*}
    \big|\pi_z & \Phi(t; q,t_0) -  \pi_z \Phi_A(t; q, t_0)S(q + \omega t_0, t_0)^{-1}\big| = \big|\pi_z (S(q + \omega t, t) - I)\Phi_A(t; q, t_0)S(q + \omega t_0, t_0)^{-1}\big| \\
    & \leq \big|(S_{zq}^t(q + \omega t) \quad S_{zp}^t(q + \omega t) \quad S_{zz}^t(q + \omega t) - \mathrm{Id}_{2m})   \Phi_A(t; q, t_0) \big| \big| S(q + \omega t_0, t_0)^{-1} \big|.
    \end{align*}
    
Notice that $S^{-1} = -JS^\top J$ is uniformly bounded. Moreover, by \eqref{eq:S_formula}, $S_{zq} = \partial_q w,$ $S_{zp} = 0$ and $S_{zz} = e^G$. By Propositions \ref{prop:trasv_dyn_par_deg_case} and \ref{prop:trasv_dyn_par}, in each scenario, we can bound the growth rate of the norm of each entry of $\Phi_A(t; q, t_0)$ as $t$ goes to infinity. The result follows by combining these bounds with the assumptions on $w$, $G$, and the bounds for $e^G - \mathrm{Id}_{2m}$ in Lemma \ref{prop:exp_bound_par}.
\end{proof}

\section{Asymptotically Parabolic Degenerate Case}
\label{sec:Proof_Theorem_Par_0}
In this section we prove Theorems \ref{Thm:par_Csigma_0}, \ref{Thm:par_Csigma_0_analy} and Corollary \ref{cor:par_deg}. 

The proof of Theorem A in Part I \cite{scarcella_KAMST}, which considers the asymptotically elliptic scenario, can be carried out, almost verbatim, in the degenerate setting of Theorem \ref{Thm:par_Csigma_0} by replacing $\Omega = \textup{diag}(\Omega_1, \dots, \Omega_m)$ and $\Mell$ by $0_{m \times m}$ and $0_{2m \times m}$, respectively. These degenerate assumptions greatly simplify the analysis of the cohomological equations appearing in the proof of Theorem A in \cite{scarcella_KAMST} and the proof can be easily adapted with very minor modifications. 

To avoid repetition, and since the proof of Theorem \ref{Thm:par_Csigma} will follow a similar structure but will be much more technically involved, we will only sketch the proof of Theorem \ref{Thm:par_Csigma_0}. We will describe in detail the steps to follow (without proving them) and refer to \cite{scarcella_KAMST} for details.

We point out that  Theorem \ref{Thm:par_Csigma} will not imply Theorem \ref{Thm:par_Csigma_0} as the time decay assumptions required to prove Theorem \ref{Thm:par_Csigma} will be stronger. The decay rates in Theorem \ref{Thm:par_Csigma_0} are possible due to the simplified setting $\Lambda_1 = \dots, \Lambda_m = 0.$

Theorem \ref{Thm:par_Csigma_0_analy}, which is simply the version of Theorem \ref{Thm:par_Csigma_0} adapted to real analytic Hamiltonians, can be proved along the same lines as Theorem \ref{Thm:par_Csigma_0}, with some minor modifications. Moreover, since Theorem \ref{Thm:par_Csigma_0} applies under the hypotheses of Theorem \ref{Thm:par_Csigma_0_analy}, it suffices to prove only the first part of Theorem \ref{Thm:par_Csigma_0_analy}. The proof of Corollary \ref{cor:par} is instead obtained by combining Theorem \ref{Thm:par_Csigma_0} with the asymptotic correspondence established in Proposition \ref{prop:par_deg_1:1_asym_sol}. Since the proof is exactly the same as that of Corollary A in Part I \cite{scarcella_KAMST} we will omit it.

For these reasons, except for Subsection \ref{sc:proof_analytic_par}, where we discuss the modifications needed to adapt the proof to the analytic setting, we will focus mostly on Theorem \ref{Thm:par_Csigma_0} and on Hamiltonians with Hölder regularity. 

 The remainder of this section is structured as follows. In Section \ref{sc:functional_setting_par0}, we recall the definition of several Banach spaces of functions with polynomial decay in time that were introduced in Part I \cite{scarcella_KAMST}. In Section \ref{sc:initial_ham_par0} we write the unperturbed Hamiltonian of Theorem \ref{Thm:par_Csigma_0} in a suitable form and set some notations. In Sections \ref{sec:proof_item_1_par_deg_Holder} and \ref{sec:proof_item_2_par_deg_Holder}, we sketch the proof of items \eqref{thm:Cdeg_existence} and \eqref{thm:C_transverse} of Theorem \ref{Thm:par_Csigma_0}, respectively, following the proof of Theorem A in Part I \cite{scarcella_KAMST}. In Section \ref{sc:proof_analytic_par0} we discuss the proof of Theorem \ref{Thm:par_Csigma_0_analy}.

\subsection{Functional setting}
\label{sc:functional_setting_par0}
Let us recall the definition of some of the spaces introduced in Part I \cite{scarcella_KAMST} that will be used in the present work.

\subsubsection{Real-valued and matrix-valued functions with polynomial decay}\label{sc:real_pol_decay} The following spaces will allow us to describe the regularity of the perturbation and the asymptotic KAM tori in the proof of our results. 

Recall that the Banach space $\mathcal{S}^{ T}_{(\sigma, k), \ell}$ is defined in~\eqref{def:S} and the associated norm $|\cdot|^{ T}_{\sigma+k, \ell}$ in~\eqref{def:norm_S}.  Letting $k \ge 1$, we introduce the space
\begin{equation}\label{def: S0ell}
    \mathcal{S}^{T}_{(\sigma,k), (0,\ell)} = \left\{f:\T^n \times B \times I_T \to \R  \hspace{1mm} \left| \hspace{1mm} f \in \mathcal{S}^{ T}_{(\sigma,k), 0}, \hspace{2mm} \partial_q f \in \mathcal{S}^{ T}_{(\sigma, k-1), \ell}\right\}\right.,
\end{equation}
endowed with the norm
\begin{equation}\label{def:norm_S0ell}
    |f|^{ T}_{(\sigma,k), (0,\ell)} = \max\{|f|^{ T}_{\sigma+k, 0}, |\partial_qf|^{ T}_{\sigma+k-1, \ell}\},
\end{equation}

To quantify the regularity of the components of the $C^\sigma$  asymptotic KAM torus, we define 
\begin{equation}\label{def:U}
    \mathcal{U}^{ T}_{\sigma, \omega, \ell} = \left\{u:\T^n \times I_T \to \R^n \hspace{1mm} \left| \hspace{1mm} u \in \mathcal{S}^{T}_{(\sigma,0), \ell}, \hspace{2mm}\Lo u \in \mathcal{S}^{ T}_{(\sigma,0), \ell+1}\right\}\right.,
\end{equation}
endowed with the norm
\begin{equation}\label{def:norm_U}
    |u|^{ T}_{\sigma, \omega, \ell} = \max \left\{|u|^{ T}_{\sigma, \ell}, |\Lo u|^{ T}_{\sigma, \ell +1}\right\}.
\end{equation}

Furthermore, we consider 
\begin{equation}\label{def:W}
    \mathcal{W}^{\mathrm{ell}, T}_{\sigma, \omega, \Omega, \ell} = \left\{w:\T^n \times I_T \to \R^{2m} \hspace{1mm} \left| \hspace{1mm} w \in \mathcal{S}^{ T}_{(\sigma,0), \ell}, \hspace{2mm}\mathcal{L}^{\mathrm{ell}}_{\omega,\Omega} w \in \mathcal{S}^{ T}_{(\sigma,0), \ell+1}\right\}\right.,
\end{equation}
endowed with the norm
\begin{equation}\label{def:norm_W}
    |w|^{ T}_{\sigma, \omega, \Omega, \ell} = \max \left\{|w|^{ T}_{\sigma, \ell}, |\mathcal{L}^{\mathrm{ell}}_{\omega,\Omega} w|^{ T}_{\sigma, \ell +1}\right\}.
\end{equation}

\begin{remark}\label{rmk:extra_regularity_ell}
    If $\sigma \geq 1$ then any $u \in \mathcal{U}^{ T}_{\sigma, \omega, \ell}$ (resp. $w \in \mathcal{W}^{\mathrm{ell}, T}_{\sigma, \omega, \Omega, \ell}$) is of class $C^1$. Moreover, if $\sigma \geq 2$ then $\partial_q u \in  \mathcal{U}^{ T}_{\sigma - 1, \omega, \ell}$ (resp.  $\partial_q w \in \mathcal{W}^{\mathrm{ell}, T}_{\sigma, \omega, \Omega, \ell}$) and, in particular, $\partial_q u$ (resp. $\partial_q w$) is of class $C^1$. See \cite{scarcella_KAMST}.
\end{remark}

\subsubsection{Matrix-valued functions with polynomial decay}
\label{sc:matrix_spaces_ell} 
We also introduce spaces of functions taking values in $\mathcal{M}_{2m}(\R)$. These spaces will be useful for describing the normal coordinates of the asymptotic KAM torus. 

Let
\begin{equation}
\begin{aligned}\label{def:BS_ell_M_Sym}
     \mathcal{M}^{ T}_{\sigma, \ell} &= %
     \mathcal{S}^{ T}_{(\sigma,0), \ell}(\R^{2m \times 2m}), \\
          \mathcal{S}ym^{ T}_{\sigma, \ell} &= \left\{M:\T^n \times I_T \to \textup{Sym}(2m, \R) \hspace{1mm} \left| \hspace{1mm}  M \in \mathcal{M}^{ T}_{\sigma, \ell} \right\}\right.,
    \end{aligned}
\end{equation}      
endowed with the norm,
\[ |M|^{ T}_{\sigma, \ell} := \max_{1 \leq i, j \leq 2m} |M_{ij}|^{ T}_{\sigma, \ell},\]
 and
 \begin{align}\label{def:BS_ell_Sp}
 \mathcal{S}p^{\mathrm{ell}, T}_{\sigma, \omega, \Omega, \ell} %
    &  = \left\{M:\T^n \times I_T \to \mathfrak{sp}(2m, \R) \hspace{1mm} \left| \hspace{1mm} M \in \mathcal{M}^{ T}_{\sigma, \ell}, \quad \mathfrak{L}_{\omega, \Omega}^\mathrm{ell} M \in \mathcal{M}^{ T}_{\sigma, \ell + 1} \right\} \right.,
\end{align}
endowed with the norm
\[ |M|^{ T}_{\sigma, \omega, \Omega, \ell} = \max\left \{ |M|_{\sigma, \ell}^{ T} ,|\mathfrak{L}_{\omega, \Omega}^\mathrm{ell} M|^{ T}_{\sigma, \ell + 1} \right\},\]
where $\textup{Sym}(2m, \R)$ denotes the space of $2m \times 2m$ symmetric matrices and $\mathfrak{sp}(2m, \R)$ denotes the Lie algebra associated with the space of $2m \times 2m$ symplectic matrices $\textup{Sp}(2m, \R),$ namely, 
\[ \mathfrak{sp}(2m, \R) := \{ G \in \mathcal{M}_{2m}(\R) \mid J_m G + G^{\top}J_m = 0 \}. \]

Notice that if $\sigma \geq 1$ then any $M \in  \mathcal{S}p^{\mathrm{ell}, T}_{\sigma, \omega, \Omega, \ell}$ is of class $C^1$. 

\subsubsection{Perturbations and torus embeddings with non-uniform polynomial decay}  \label{sc:perturbation_par0}
In order to describe the space of perturbations considered in Theorem \ref{Thm:par_Csigma_0}, as well as , given $\mathcal{L} = (\ell_1, \ell_2, \ell_3, \ell_4) \in \left(\R_{\ge 0}\right)^4$, $\mathcal{K} = (k_1, k_2, k_3) \in \left(\R_{\ge 0}\right)^3$ and $T\ge 1$, we define
\begin{equation}\label{proof_par_deg_1_def_spaces_P_E}
    \begin{aligned}
    &\mathcal{P}^{\mathrm{par.deg}, T}_{\sigma, \mathcal{L}} = \mathcal{S}^{T}_{(\sigma,2), (0,\ell_1)}(\R) \times \mathcal{S}^{T}_{(\sigma,2), \ell_2}(\R^n) \times \mathcal{S}^{T}_{(\sigma,2), \ell_3}(\R^{2m}) \times \mathcal{S}^{T}_{(\sigma,2), \ell_4}(\R^{2m \times 2m}),\\
    &\mathcal{E}^{\mathrm{par.deg}, T}_{\sigma, \mathcal{K}} = \mathcal{U}^{T}_{\sigma, \omega, k_1} \times \mathcal{U}^{T}_{\sigma, \omega, k_2} \times \mathcal{U}^{ T}_{\sigma, \omega, k_3},
    \end{aligned}
\end{equation}
each endowed with the sum norm associated with the corresponding product space, which we denote by $|\cdot|^{\mathrm{par.deg}, T}_{\sigma, \mathcal{L}}$ and $|\cdot|^{\mathrm{par.deg}, T}_{\sigma, \mathcal{K}}$, respectively.

By an abuse of notation, we denote the elements in the product spaces as $P=(a,b,c,d)\in\mathcal{P}^{\mathrm{par.deg},T}_{\sigma,\mathcal{L}}$ and $\varphi=(u,v,w)\in\mathcal{E}^{\mathrm{par.deg},T}_{\sigma,\mathcal{K}}$
with the same letters of the corresponding perturbation $P$ in~\eqref{eq:perturbation_form_par_deg} and family of torus embeddings $\varphi$ in \eqref{eq:asymp_KAM_torus_form_par0}. 

Given $\ell > 1$ and $\dec \ge 0$, we denote
\begin{equation}
\begin{gathered}\label{def:LK_par_deg}
    \mathcal{L}(\ell, \dec) = (\ell + \dec +2, \ell, \ell +\dec+1, \ell), \quad \mathcal{L}^*(\ell, \dec) = (\ell + \dec +3, \ell, \ell +\dec+2, \ell+1),\\ \mathcal{K}(\ell, \dec) = (\ell-1, \ell+\dec+1, \ell+\dec), \quad \mathcal{K}^*(\ell, \dec)=\mathcal{K}(\ell, \dec+1).
\end{gathered}
\end{equation}

With these notations, the spaces of perturbations associated with Items \eqref{thm:C_existence} and \eqref{thm:C_transverse} of Theorem \ref{Thm:par_Csigma_0} are given by  $\mathcal{P}^{\mathrm{par.deg}, 1}_{\sigma, \Lell(\ell, \dec)}$ and $\mathcal{P}^{\mathrm{par.deg}, 1}_{\sigma, \Lell^*(\ell, \dec)}$, respectively. Moreover, as we shall see in Section \ref{sec:proof_item_1_par_deg_Holder}, for any $P \in  \mathcal{P}^{\mathrm{par.deg}, 0}_{\sigma, \Lell(\ell, \dec)}$ there exists a $C^\sigma$ asymptotic KAM torus in $ \mathcal{E}^{\mathrm{par.deg}, T}_{\sigma, \mathcal{K}(\ell, \dec)}$ associated with $(X_{H_0 + P}, X_{H_0}, \varphi_0)$, for some $T \geq 1$.

\subsection{Initial Hamiltonian in Theorem \ref{Thm:par_Csigma_0} and conventions in notation}\label{sc:initial_ham_par0}
In the following, we fix $\omega \in \R^n$, $\sigma \geq 1$ and $H_0: \T^n \times B \times I_1 \to \R$ of the form \eqref{eq:initial_hamiltonian_par_deg} as in the statement of Theorem \ref{Thm:par_Csigma_0}, where $B \subseteq \R^{n +2m}$ is a ball centred at $0$ and $I_T = [T, + \infty)$, for any $T \geq 1$. For simplicity, we assume that $B = B_{\R^{n + 2m}}(1)$ is the unit ball in $\R^{n + 2m}$.  %

Since all the norms considered in this section concern (products of) function spaces with polynomial decay, for the sake of clarity and to simplify the notation, we omit the superscripts $\mathrm{par.deg}$ from all the norms whenever there is no risk of confusion.

Notice that we can express  $H_0$ in \eqref{eq:initial_hamiltonian_par_deg} as 
\begin{equation}
\label{eq:initial_Hamiltonian_NR}
        H_0(q,p,z,t) = N(p,z) + R(q,p,z,t), 
\end{equation}
where 
\begin{equation*}
\begin{array}{rcl}   N(p,z) &= & \omega \cdot p\\
    R(q,p,z,t) &= &M(q,p,z,t) \cdot p^2 + m(q,z,t) \cdot (p,z) + L(q,z,t) \cdot z^3,
\end{array}
\end{equation*}
and
\begin{equation}
\label{def:abcdMmL_terms}
\begin{aligned}
    M(q,p,z,t) &= \int_0^1 (1 -\tau) \partial_p^2 H_0(q, \tau p, z,t) d\tau ,\\
    m(q,z,t) &= \int_0^1 \partial_p\partial_z H_0(q, 0, \tau z,t) d\tau,\\
    L(q,z,t) &=  {1 \over 2} \int_0^1 (1-\tau)^2 \partial_z^3 H_0(q,0,\tau z, t) d\tau. 
\end{aligned}
\end{equation}
satisfying
\begin{equation}
  \label{def:const_Upsilon_hyp}
|M|^{ 1}_{\sigma + 3, 0}, |m|_{\sigma + 3, 0}^{1}, |L|_{\sigma + 2, 0}^{1} \le|\partial_{(p,z)}^2 H_0|^{ 1}_{\sigma+3, 0} \leq \Upsilon, \end{equation}
for some $\Upsilon > 0$ that we fix for the rest of the section. 

We recall that the trivial embedding $\varphi_0:\T^n \to \T^n \times \R^{n+2m}$ in~\eqref{def:varphi0=(q,0,0)} is an invariant torus for $H_0$ supporting quasiperiodic solutions fo frequecy vector $\omega$.

Finally, for the sake of simplicity, we fix  $\ell > 1$ and $\dec \ge 0$, and denote $\mathcal{L}(\ell, \dec)$, $\mathcal{L}^*(\ell, \dec)$, $\mathcal{K}(\ell, \dec)$, and $\mathcal{K}^*(\ell, \dec)$ simply by $\mathcal{L}$, $\mathcal{L}^*$, $\mathcal{K}$, and $\mathcal{K}^*$, respectively.

\subsection{Proof of Item \eqref{thm:Cdeg_existence} in Theorem \ref{Thm:par_Csigma_0}: Construction of $C^\sigma$ asymptotic KAM tori}\label{sec:proof_item_1_par_deg_Holder} In the following, we will use freely the notations introduced at the beginning of this section and in the previous two subsections. Recall that the space of perturbations associated with Item \eqref{thm:C_existence} of Theorem \ref{Thm:par_Csigma_0} is given by  $\mathcal{P}^{\mathrm{par.deg}, 1}_{\sigma, \Lell}$ (see Section \ref{sc:perturbation_par0}).

Given $r, \rho>0$, we denote by $B_{\mathcal{P}^{\mathrm{par.deg}, T}_{\sigma, \Lell}}(r)  \subseteq  \mathcal{P}^{\mathrm{par.deg}, T}_{\sigma, \Lell}$ and $B_{ \mathcal{E}^{\mathrm{par.deg}, T}_{\sigma, \mathcal{K}}}(\rho) \subseteq \mathcal{E}^{\mathrm{par.deg}, T}_{\sigma, \mathcal{K}}$ the open balls centered at the origin with radius $r$ and $\rho$, respectively.

The following analogous of Proposition 6.1 in Part I \cite{scarcella_KAMST} holds.

\begin{proposition}
\label{prop:C_deg_existence}
There exists $C_0 > 1$ such that for any $r > 0$ there exists $T_0 \geq 1$ satisfying the following. For any $T \geq T_0$, there exists a $C^1$-map
\[ \boldsymbol{\varphi}^T: B_{\mathcal{P}^{\mathrm{par.deg}, T}_{\sigma, \Lell}}(r)  \subseteq  \mathcal{P}^{\mathrm{par.deg}, T}_{\sigma, \Lell}  \to  B_{ \mathcal{E}^{\mathrm{par.deg}, T}_{\sigma, \mathcal{K}}}(C_0r) \subseteq \mathcal{E}^{\mathrm{par.deg}, T}_{\sigma, \mathcal{K}}\]
such that $\boldsymbol{\varphi}^T(P)$ defines an asymptotic KAM torus associated with $(X_{H_0 + P}, X_{H_0}, \varphi_0)$,  for any $P \in  \mathcal{P}^{\mathrm{par.deg}, T}_{\sigma, \Lell} $ with $|P|^T_{\sigma, \Lell} < r$.
\end{proposition}

Notice that Item \eqref{thm:C_existence} of Theorem \ref{Thm:par_Csigma_0} follows directly from the proposition above since
\begin{equation*}
    \big|P\big|_{\sigma, \Lell}^{T} \leq |P|_{\sigma, \Lell}^{1}, \qquad \text{ for any } T \geq 1\,  \text{ and any } \, P \in \mathcal{P}^{\mathrm{par.deg}, 1}_{\sigma, \Lell},
\end{equation*}
and by definition of $\mathcal{K}$ the asymptotic KAM torus has the desired decay rates.

Proposition~\ref{prop:C_deg_existence} follows by applying an appropriate version of the the Implicit Function Theorem (see Theorem \ref{thm:QIFT}) to the following functional
\begin{equation}\label{proof:Csigma_def_F_par_deg_1}
    \Function{\mathcal{F}^T}{ \mathcal{P}^{\mathrm{par.deg}, T}_{\sigma,\Lell} \times \mathcal{E}^{\mathrm{par.deg}, T}_{\sigma, \mathcal{K}}}{ \mathcal{S}^{T}_{(\sigma,0), \ell}(\R^n) \times \mathcal{S}^{T}_{(\sigma,0), \ell+\dec+2}(\R^n) \times \mathcal{S}^{T}_{(\sigma,0), \ell+\dec+1}(\R^{2m})}{(P, \varphi)}{X_{H_0 + P} \circ \varphi - \Lo \varphi },
\end{equation}
where $T \geq 1$ and we endow the codomain $\mathcal{S}^{T}_{(\sigma,0), \ell}(\R^n) \times \mathcal{S}^{T}_{(\sigma,0), \ell+\dec+2}(\R^n) \times \mathcal{S}^{T}_{(\sigma,0), \ell+\dec+1}(\R^{2m})$ with the sum norm associated with the product, which we denote by $|\cdot|^T_{\boldsymbol{\sigma, \ell, \dec}}$. Here,  we used bold symbols to differentiate this norm from the norm $|\cdot|^T_{\sigma, \ell}$.

Indeed, with these definitions, any $\varphi \in \mathcal{E}^{\mathrm{par.deg}, T}_{\sigma, \mathcal{K}}$ satisfies \eqref{def:cond1_asymKAMtorus}, while $\mathcal{F}^T(P, \varphi) = 0$ if and only if \eqref{def:cond2_asymKAMtorus} is satisfied. Therefore,
\begin{equation}
\label{eq:KAM_torus_characterization}
    \mathcal{F}^T(P, \varphi) = 0 \quad \mbox{if and only if $\quad \varphi$ defines a $C^\sigma$ asymptotic KAM torus for $H_0 + P$.}
\end{equation}

Analyzing the differential of the functional $\mathcal{F}^T$ with respect to the variable $\varphi=(u,v,w)$ evaluated at $(0,0)$ which we denote by $D_\varphi\mathcal{F}^T(0,0)$, the following analogous of Proposition 6.2 in \cite{scarcella_KAMST} holds. 

\begin{proposition}
\label{prop:F_properties_par_deg}
    The operator $\mathcal{F}^T$ given by \eqref{proof:Csigma_def_F_par_deg_1} satisfies the following.
    \begin{enumerate}
    \item \label{prop:F_well_defined_par_deg} For any $T \geq 1$, there exists a neighbourhood $\mathcal{U}^T \subseteq \mathcal{P}^{\mathrm{par.deg}, T}_{\sigma,\Lell} \times \mathcal{E}^{\mathrm{par.deg}, T}_{\sigma, \mathcal{K}}$ of $(0, 0)$ such that $\mathcal{F}^T\mid_{\mathcal{U}^T}$ and $D_\varphi\mathcal{F}^T\mid_{\mathcal{U}^T}$ are well-defined continuous maps. 
    
    \noindent Moreover, for any $r, \rho \geq 0$ there exists $T_0 \geq 1$ such that, for any $T \geq T_0$,  
    \[B_{\mathcal{P}^{\mathrm{par.deg}, T}_{\sigma,\Lell}}(r) \times B_{\mathcal{E}^{\mathrm{par.deg}, T}_{\sigma, \mathcal{K}}}(\rho) \subseteq \mathcal{U}^T.\]

         \item \label{prop:F_bounded_par_deg}There exists a constant $C$,
     such that for any $r > 0$ and any $T \geq 1$,
\[ \sup \left\{ |\mathcal{F}^T(P, 0)|_{\boldsymbol{\sigma, \ell, \dec}}^T \,\left|\, P \in B_{\mathcal{P}^{\mathrm{par.deg}, T}_{\sigma,\Lell}}(r); \,\, (P, 0) \in \mathcal{U}^T \right\}\right. \leq Cr.\] %

    \item \label{prop:DF_invertible_par_deg} $D_\varphi \mathcal{F}^T(0, 0)$ is invertible, for any $T \geq 1$. Moreover, there exists a constant $\bar C$, independent of $T$, such that $\|D_\varphi \mathcal{F}^T(0, 0)^{-1}\| \leq \bar C$, where $\| \cdot\|$ denotes the operator norm. %

        \item \label{prop:DF_limit_par_deg} are identical to that of  Let $r, \rho > 0$. Then  
        \[\lim_{T \to +\infty} \| D_\varphi\mathcal{F}^T(P, \varphi) - D_\varphi\mathcal{F}^T(0, 0) \| = 0, \qquad \text{uniformly on } \quad B_{\mathcal{P}^{\mathrm{par.deg}, T}_{\sigma,\Lell}}(r) \times B_{\mathcal{E}^{\mathrm{par.deg}, T}_{\sigma, \mathcal{K}}}(\rho),\]%
where $\| \cdot\|$ denotes the operator norm. %
    \end{enumerate}
\end{proposition}

Assuming Proposition \ref{prop:F_properties_par_deg}, it is easy to show Proposition \ref{prop:C_deg_existence}.

\begin{proof}[Proof of Proposition \ref{prop:C_deg_existence}]
Let $r > 0$ and define $\rho = 2C\bar C r,$ where $C$ and $\bar C$ are constants given by Items \eqref{prop:F_bounded_par_deg}, and \eqref{prop:DF_invertible_par_deg} of Proposition \ref{prop:F_properties_par_deg}.

By Items \eqref{prop:F_well_defined_par_deg} and \eqref{prop:DF_limit_par_deg} of Proposition \ref{prop:F_properties_par_deg} there exists $T_0 \geq 1$ such that, for any $T \geq T_0$,  $B_{\mathcal{P}^{\mathrm{ell}, T}_{\sigma,\Lell}}(r) \times B_{\mathcal{E}^{\mathrm{ell}, T}_{\sigma, \mathcal{K}}}(\rho)$ is contained in the domain of definition of $\mathcal{F}^T$ and 
\[  \| D\mathcal{F}^T(P, \varphi) - D\mathcal{F}^T(0, 0) \| < \frac{1}{2\bar C}, \qquad \text{ for } \quad |P|_{\sigma, \Lell}^T < r, \quad |\varphi|_{\sigma, \omega, \Omega, \mathcal{K}}^T < \rho.\]

Let $T \geq T_0$. By Items \eqref{prop:F_bounded_par_deg} and \eqref{prop:DF_invertible_par_deg} of Proposition \ref{prop:F_properties_par_deg},
\[\|D_\varphi \mathcal{F}^T(0, 0)^{-1}\| \leq \bar C, \qquad \sup \left\{ |\mathcal{F}^T(P, 0)|_{\sigma, \Lell}^T \,\left|\, P \in B_{\mathcal{P}^{\mathrm{ell}, T}_{\sigma,\Lell}}(r)  \right\}\right. \leq \frac{\rho}{2\bar C}.\]

Noticing that $\mathcal{F}^T(0, 0) = 0$ and recalling \eqref{eq:KAM_torus_characterization}, Proposition \ref{prop:C_deg_existence} follows by Theorem \ref{thm:QIFT}. 
\end{proof}

\subsection{Proof of Item \eqref{thm:Cdeg_transverse} in Theorem \ref{Thm:par_Csigma_0}: Existence of asymptotic parabolic degenerate transversal dynamics}\label{sec:proof_item_2_par_deg_Holder}
Recall that the space of perturbation in Item \eqref{thm:Cdeg_transverse} is given by $\mathcal{P}^{\mathrm{par.deg}, 1}_{\sigma, \mathcal{L}^*}$ (see Section \ref{sc:perturbation_par0}). Furthermore, we observe that, for any $T \ge 1$, 
\begin{equation*}
    \mathcal{P}^{\mathrm{par.deg}, T}_{\sigma, \mathcal{L}^*} \subset \mathcal{P}^{\mathrm{par.deg}, T}_{\sigma, \mathcal{L}}, \quad \mbox{and} \hspace{3mm} |P|^T_{\sigma, \mathcal{L}} \le |P|^T_{\sigma, \mathcal{L}^*} \hspace{2mm} \mbox{ for every $P \in \mathcal{P}^{\mathrm{par.deg}, 1}_{\sigma, \mathcal{L}^*}$.}
\end{equation*}
Thus, the hypotheses of Item \eqref{thm:Cdeg_existence} of Theorem \ref{Thm:par_Csigma_0} are satisfied.

In what follows, we fix $r>0$ and let $T_0 \ge 1$ be given by Proposition \ref{prop:C_deg_existence}. By Proposition \ref{prop:C_deg_existence}, for every $T \ge T_0$, there exists a map
\[\boldsymbol{\varphi}^T: B_{\mathcal{P}^{\mathrm{par.deg}, T}_{\sigma, \Lell^*}}(r)  \to \mathcal{E}^{\mathrm{par.deg}, T}_{\sigma, \mathcal{K}^*}\]
for which $\boldsymbol{\varphi}^T(P)$ defines a $C^\sigma$ asymptotic KAM torus associated with $(X_{H_0 + P}, X_{H_0}, \varphi_0)$,  for any $P \in  \mathcal{P}^{\mathrm{par.deg}, T}_{\sigma, \Lell^*} $ with $|P|^T_{\sigma, \Lell^*} < r$.

Recall that the asymptotic KAM torus $\boldsymbol{\varphi}^T(P)$ associated to the perturbation $P$ is said to be \emph{parabolic} if there exist a family of invertible matrices $S:\T^n \times I_{T'} \to \mathcal{M}_{2n+2m}(\R)$, for some $T' \geq T$ as in Definition \ref{def:ell_hyp_par_trasv_dyn_asym_KAM}.  We will construct these invertible matrices using Proposition \ref{prop:precriterion}, which to any continuous map $G: \T^n \times I_T \to \mathfrak{sp}(2m, \R)$ associates a continuous family of symplectic matrices $\boldsymbol{S}(\boldsymbol{\varphi}^T(P), G): \T^n \times I_T \to Sp(2m + 2n, \R, J)$ of the form \eqref{eq:S_formula}. Moreover, it follows from \eqref{eq:S_formula} that if  $\boldsymbol{\varphi}^T(P)$ and $G$ are of class $C^1$ then $\boldsymbol{S}(\boldsymbol{\varphi}^T(P), G)$ is also of class $C^1$.

Therefore, by Corollary \ref{cor:criterion}, in order for $S = \boldsymbol{S}(\boldsymbol{\varphi}^T(P), G)$, given by Proposition \ref{prop:precriterion}, to satisfy the assumptions of Definition \ref{def:ell_hyp_par_trasv_dyn_asym_KAM}, it suffices to assume that $G$ and $\varphi = \boldsymbol{\varphi}^T(P)$ are sufficiently regular and satisfy suitable decay conditions, and to show that the function $\boldsymbol{\zeta}(H_0 + P, \varphi, G)$, defined by \eqref{eq:formula_zeta}, is equal to $0_{2n \times 2n}$, namely,
\[
\Pi_{xy}^\top S^\top\big( (D^2 (H_0 + P) \circ \varphi) S + J \Lo S\big) \Pi_{xy} = 0_{2n \times 2n}
\]
where we refer to~\eqref{def:Proj_Pi_xy} for the definition of the projection $\Pi_{xy}$.

The following analogous of Proposition 6.7 in Part I \cite{scarcella_KAMST} holds.

\begin{proposition}
\label{prop:Cdeg_transverse}
Fix $r > 0$ and let $T_0 \geq 1$ be given by Proposition \ref{prop:C_deg_existence}. There exist $T_1 \geq T_0$ such that, for any $T \geq T_1$,  there exists a $C^1$-map
\[ \boldsymbol{G}^T: B_{\mathcal{P}^{\mathrm{par.deg}, T}_{\sigma, \Lell^*}}(r)    \subseteq  \mathcal{P}^{\mathrm{par.deg}, T}_{\sigma, \Lell^*} \to \mathcal{S}p^{ T}_{\sigma-1, \omega, \ell + \dec}  \]
such that
\[\boldsymbol{\zeta}(H_0 + P, \boldsymbol{\varphi}^T(P),  \boldsymbol{G}^T(P)) =  0_{2m \times 2m}, \qquad \text{for any } P \in B_{\mathcal{P}^{\mathrm{par.deg}, T}_{\sigma, \Lell^*}}(r),\]
where $\boldsymbol{\zeta}$ is given by \eqref{eq:formula_zeta} and  $\boldsymbol{\varphi}^T: B_{\mathcal{P}^{\mathrm{par.deg}, T}_{\sigma, \Lell^*}}(r)   \to \mathcal{E}^{\mathrm{par.deg}, T}_{\sigma, \mathcal{K}^*}$  is given by Proposition \ref{prop:C_deg_existence}.

In particular, by Corollary \ref{cor:criterion}, for any $P \in B_{\mathcal{P}^{\mathrm{par.deg}, T}_{\sigma, \Lell^*}}(r)$, the asymptotic KAM torus $\boldsymbol{\varphi}^T(P)$ is asymptotically parabolic degenerate.
\end{proposition}

Taking into account the explicit formula for $\boldsymbol{\zeta}$ given in Remark \ref{rmk:formula_zeta}, Proposition~\ref{prop:Cdeg_transverse} follows by applying an appropriate version of the the Implicit Function Theorem (see Theorem \ref{thm:QIFT}) to the following functional
$$\mathcal{G}^T: B_{\mathcal{P}^{\mathrm{par.deg}, T}_{\sigma,\Lell^*}}(r) \times \mathcal{S}p^{ T}_{\sigma-1, \omega, \ell+\dec}   \to  \mathcal{S}ym^{ T}_{\sigma-1, \ell+\dec+1}$$ 
as
\begin{equation}
\label{eq:formula_F_par.deg_2}
    \begin{aligned}
    \mathcal{G}^T(P, G) =&  B^\top \partial_p^2 H\circ \varphi B + K^{\top} \partial_p\partial_zH\circ \varphi B + B^{\top}  \left(\partial_p\partial_zH\circ \varphi\right)^\top K \\
    & + K^{\top} \partial_z^2 H\circ \varphi K + K^{\top} J_m \Lo K,
    \end{aligned}
\end{equation}
where $T \geq 1$, $H = H_0 + P$, $\varphi = \boldsymbol{\varphi}^T(P) = (u, v, w)$, $K = \exp(G)$, and  $B = -(\mathrm{Id}_{2n} +\partial_q  u )^{-\top} (\partial_q w)^\top J_m K$. 

Analyzing the differential of the functional $\mathcal{G}^T$ with respect to $G$ evaluated at $(0,0)$ which we denote by $D_\varphi\mathcal{G}^T(0,0)$, the following analogous of Proposition 6.8 in \cite{scarcella_KAMST} holds.

\begin{proposition}
\label{prop:G_properties_par_deg}
    The operator $\mathcal{G}^T$ given by \eqref{eq:formula_F_par.deg_2} satisfies the following.
    \begin{enumerate}
    \item \label{prop:G_well_defined_par_deg} $\mathcal{G}^T$ and $D_G\mathcal{G}^T$ are well-defined continuous maps, for any $T \geq T_0$. 
    
    \item \label{prop:G_bounded_par_deg} There exists a constant $C$ such that, for any $T \geq T_0$,
\[ \sup \left\{ |\mathcal{G}^T(P, 0)|^T_{\sigma - 1, \ell+\dec+1} \,\left|\, P \in B_{\mathcal{P}^{\mathrm{par.deg}, T}_{\sigma,\Lell^*}}(r) \right\}\right. \leq Cr.\] %

    \item \label{prop:DG_invertible_par_deg} $D_G \mathcal{G}^T(0, 0)$ is invertible, for any $T \geq T_0$. Moreover, there exists a constant $\bar C$, independent of $T$, such that $\|D_G \mathcal{G}^T(0, 0)^{-1}\| < \bar C$. 
        \item \label{prop:DG_limit_par_deg} Let $r, \rho > 0$. Then  
        \[\lim_{T \to +\infty} \| D_G \mathcal{G}^T(P, G) - D_G\mathcal{G}^T(0, 0) \| = 0, \qquad \text{uniformly on } \quad B_{\mathcal{P}^{\mathrm{par.deg}, T}_{\sigma,\Lell^*}}(r) \times B_{\mathcal{S}p^{ T}_{\sigma-1, \omega,  \ell+\dec}}(\rho),\]
where $\| \cdot\|$ denotes the operator norm  and $B_{\mathcal{S}p^{ T}_{\sigma-1, \omega, \ell +\dec}}(\rho) \subset \mathcal{S}p^{ T}_{\sigma-1, \omega,  \ell+\dec}$ stands for a ball of radius $\rho$ centered at the origin.
    \end{enumerate}
\end{proposition}

Assuming Proposition \ref{prop:G_properties_par_deg}, the proof of Proposition \ref{prop:Cdeg_transverse} follows the same lines as that of Proposition \ref{prop:F_properties_par_deg},  with Proposition~\ref{prop:G_properties_par_deg} replacing Proposition~\ref{prop:F_properties_par_deg}.

Clearly, Propositions \ref{prop:C_deg_existence} and \ref{prop:Cdeg_transverse} imply Theorem \ref{Thm:par_Csigma_0}. Let us mention that in proving these propositions, more precisely, in analyzing the differential of the functionals $\mathcal{F}^T$ and $\mathcal{G}^T$ the only cohomological equations involved are of the form 
\[  \Lo \chi := (\partial_t + \omega \cdot \partial_q) \chi = g, \qquad  g: \T^n \times I_T \to \R,\]
in the unknowns $\chi: \T^n \times I_T \to \R$ and restricted to appropriate Banach spaces of functions with polynomial decay in time. We refer to \cite{scarcella_KAMST} for additional details.

\subsection{Proof of Theorem \ref{Thm:par_Csigma_0_analy}}
\label{sc:proof_analytic_par0} 
 As mentioned at the beginning of Section \ref{sec:Proof_Theorem_Par_0}, it suffices to prove the first part of Theorem \ref{Thm:par_Csigma_0_analy} (since the second part will follow immediately from Theorem \ref{Thm:par_Csigma_0}). 

Moreover, since the proof of the first part of Theorem \ref{Thm:par_Csigma_0_analy} follows the same lines and structure as that of Item \eqref{thm:Cdeg_existence} in Theorem \ref{Thm:par_Csigma_0}, but considering spaces of analytic instead of Hölder functions, we will only define the spaces and discuss the very minor changes needed to carry out the proof. Let us point out that, for the sake of completeness, the cohomological equations appearing in the proof of Theorem \ref{Thm:par_Csigma_0} are also considered when restricted to the analytic setting in Section \ref{sc:cohom_eqs_par} as well as the analytic counterparts of the results concerning the norms and the Banach spaces appearing in Section \ref{sc:functional_setting_par0} (see Sections \ref{sc:cohom_eqs_par} and \ref{sc:norm_properties_par}).

Recall that the space of real analytic functions with polynomial decay in time $\mathscr{S}^{T}_{\sigma,\ell}$ and its associated norm $    |\cdot |^T_{\sigma, \ell}$ (which as an abuse of notation we denote as the one associated to $\mathcal{S}^{T}_{\sigma,\ell}$) were defined in \eqref{def:S_pol_anal} and \eqref{def:norm_S_analy}, respectively.

Given $\sigma >0$, $\ell\ge 0$ and $T\ge 1$, where $\sigma$ will now denote the width of the complex neighbourhood, we define the counterparts of the spaces $\mathcal{S}^{T}_{(\sigma, \ast), (0,\ell)}$, $\mathcal{U}^{ T}_{\sigma, \omega, \ell}$ and $\mathcal{W}^{\mathrm{ell}, T}_{\sigma, \omega, \Omega, \ell}$, given by \eqref{def: S0ell},~\eqref{def:U} and~\eqref{def:W}, as
\begin{equation}\label{def: S0ell_analy}
    \mathscr{S}^{T}_{\sigma, (0,\ell)} = \left\{f:\T^n_\sigma \times B_\sigma \times I_T \to \C  \hspace{1mm} | \hspace{1mm} f \in \mathscr{S}^{ T}_{\sigma, 0}, \hspace{2mm} \partial_q f \in \mathscr{S}^{ T}_{\sigma, \ell}\right\},
\end{equation}
\begin{equation}\label{def:U_analy}
    \mathscr{U}^{ T}_{\sigma, \omega, \ell} = \left\{u:\T^n_\sigma \times I_T \to \C \hspace{1mm} | \hspace{1mm} u \in \mathscr{S}^{T}_{\sigma, \ell}, \hspace{2mm}\Lo u \in \mathscr{S}^{ T}_{\sigma, \ell+1}\right\},
\end{equation}
\begin{equation}\label{def:W_analy}
    \mathscr{W}^{\mathrm{ell}, T}_{\sigma, \omega, \Omega, \ell} = \left\{w:\T^n_\sigma \times I_T \to \C^{2m} \hspace{1mm} | \hspace{1mm} w \in \mathscr{S}^{ T}_{\sigma, \ell}, \hspace{2mm}\mathcal{L}^{\mathrm{ell}}_{\omega,\Omega} w \in \mathscr{S}^{ T}_{\sigma, \ell+1}\right\},
\end{equation}
endowed with the norms
\begin{equation}\label{def:norm_S0ell_analy}
    |f|^T_{\sigma, (0,\ell)} = \max\{|f|^T_{\sigma, 0}, |\partial_qf|^T_{\sigma, \ell}\},
\end{equation}
\begin{equation}\label{def:norm_U_analy}
    |u|^T_{\sigma, \omega, \ell} = \max \left\{|u|^T_{\sigma, \ell}, |\Lo u|^T_{\sigma, \ell +1}\right\},
\end{equation}
\begin{equation}\label{def:norm_W_analy}
    |w|^{T}_{\sigma, \omega, \Omega, \ell} = \max \left\{|w|^T_{\sigma, \ell}, |\mathcal{L}^{\mathrm{ell}}_{\omega,\Omega} w|^T_{\sigma, \ell +1}\right\},
\end{equation}
respectively.

With these notations, the proof follows the exact same lines as the proof of Item \eqref{thm:Cdeg_existence} in Theorem \ref{Thm:par_Csigma_0} , by defining the counterparts $\mathscr{P}^{\mathrm{ell}, T}_{\sigma, \Lell}$, $\mathscr{E}^{\mathrm{ell}, T}_{\sigma, \mathcal{K}}$ of the spaces $\mathcal{P}^{\mathrm{ell}, T}_{\sigma, \Lell}$, $\mathcal{E}^{\mathrm{ell}, T}_{\sigma, \mathcal{K}}$ in the obvious way and by considering the functional $\mathcal{F}^T$ in an appropriate subset of $\mathscr{P}^{\mathrm{ell}, T}_{\sigma, \Lell} \times \mathscr{E}^{\mathrm{ell}, T}_{\sigma, \mathcal{K}}$, for example $ \mathscr{P}^{\mathrm{ell}, T}_{\sigma,\Lell} \times B_{\mathscr{E}^{\mathrm{ell}, T}_{\sigma/2, \mathcal{K}}}\big(T^{(\ell - 1)/2}\big)$ (compare with Proposition \ref{prop:F_properties_par_deg}). 

Let us mention that the choice of this subset is the only major difference in the proof, as for $\mathcal{F}^T$ to be well-defined we will have to consider a smaller open neighbourhood of the origin in $\mathscr{E}^{\mathrm{ell}, T}_{\sigma, \mathcal{K}}$ with a smaller analiticity width (for example $\sigma/2$).

\section{Asymptotically Parabolic Case}
\label{sec:Proof_Theorem_Par}

In this section we prove Theorems \ref{Thm:par_Csigma}, \ref{Thm:par_analy} and Corollary \ref{cor:par}, which are the general versions of the results proven in Section \ref{sec:Proof_Theorem_Par_0}.

We will focus mostly in the proof of Theorem \ref{Thm:par_Csigma}, which follows the same proof strategy of Theorem \ref{Thm:par_Csigma_0} and relies on a version of the Implicit Function Theorem applied to appropriately defined Banach spaces of time-dependent functions exhibiting polynomial decay in time. By doing this, we will be naturally led to consider the existence of solutions to the following (cohomological) equations
\begin{equation}
\label{eq:coh_eqs_par}
\left\{ \begin{array}{lll}
    \Lo \chi := (\partial_t + \omega \cdot \partial_q) \chi = g, & & g: \T^n \times I_T \to \R,\\
    \mathcal{L}^{\mathrm{par}}_{\omega, \Lambda} \boldsymbol{\chi} :=  J_mM^{\mathrm{par}}_\Lambda \boldsymbol{\chi} - \Lo \boldsymbol{\chi}  = \boldsymbol{g}, & & \boldsymbol{g}: \T^n \times I_T \to \R^{2m}, \\
    \mathfrak{L}^{\mathrm{par}}_{\omega, \Lambda} X := X^{\top}M^{\mathrm{par}}_\Lambda + M^{\mathrm{par}}_\Lambda X + J_m \Lo X =  G, & & G: \T^n \times I_T \to \mathcal{M}_{2m}(\R),
    \end{array} \right.
\end{equation}
in the unknowns $\chi: \T^n \times I_T \to \R$, $\boldsymbol{\chi}: \T^n \times I_T \to \R^{2m}$ and $X: \T^n \times I_T \to \mathcal{M}_{2m}(\R)$ (restricted to appropriate Banach spaces of functions with polynomial decay),  where $\omega \in \R^n$, $\Lambda = \mathrm{diag}(\Lambda_1,\dots,\Lambda_m)$ with $\Lambda_1,\dots,\Lambda_m \in \R_{\geq 0}$, and the matrices $M^{\mathrm{par}}_\Lambda$ and $J_m$ are as in \eqref{Ms} and~\eqref{def:J}, respectively. Recall that  $I_T = [T, + \infty)$, for any $T \geq 1$.

Let mention that, the first equation in \eqref{eq:coh_eqs_par} was already considered in Part I \cite{scarcella_KAMST}. The other two equations will be treated separately in Propositions \ref{prop:L_w_par} and \ref{prop:L_w_Lambda}.

We stress that, even though we will closely follow the proof scheme developed in Part I and that we recalled in Section \ref{sec:Proof_Theorem_Par_0} for the parabolic degenerate case, the form of the matrix $\Mpar$ forces us to consider a more detailed development of the unperturbed Hamiltonian than the one used in Part I and in  Section \ref{sec:Proof_Theorem_Par_0} (see Section \ref{sc:initial_ham_par}), as we are no longer able to treat jointly the terms in the expansion that depend on $z = (x, y)$. In particular, we have to consider different decay rates for the terms in the perturbation (see \eqref{eq:perturbation_form_par}) depending on $x$ and $y$. These differences make it much harder to define an appropriate functional setting to implement the approach described in Section \ref{sec:Proof_Theorem_Par_0}, and make every step of the proof considerably more involved.

 The remainder of this section is structured as follows. In Section \ref{sc:functional_setting_par}, we introduce several Banach spaces of matrix-valued functions with non-uniform polynomial decay in time. In Section \ref{sc:initial_ham_par} we write the unperturbed Hamiltonian of Theorem \ref{Thm:par_Csigma} in a suitable form and set some notations. In Sections \ref{sec:proof_item_1_par_Holder}  and \ref{sec:proof_item_2_par_Holder}, we prove items \eqref{thm:C_existence} and \eqref{thm:C_transverse} of Theorem \ref{Thm:par_Csigma}, respectively, assuming several results concerning the cohomological equations mentioned above which, for the sake of clarity of exposition, we prove later in Section \ref{sc:cohom_eqs_par}. In Section \ref{sc:proof_analytic_par} we prove Theorem \ref{Thm:par_analy}.  Section \ref{sec:proof_cor_par} is dedicated to the proof of Corollary \ref{cor:par}. Finally, Section \ref{sc:norm_properties_par} contains several technical results about the norms and spaces introduced in Section \ref{sc:functional_setting_par}.%

\subsection{Functional setting}\label{sc:functional_setting_par}   In the following, we let $\omega \in \R^n$, $\Lambda = \mathrm{diag}(\Lambda_1,\dots,\Lambda_m)$ with $\Lambda_1,\dots,\Lambda_m \in \R_{\geq 0}$, $\sigma \geq 1$, $k \in \Z_{>0}$, $\ell \geq 0$ and $T \geq 1$, and we use the notations in \eqref{eq:coh_eqs_par} for the operators $\Lo, \mathcal{L}^{\mathrm{par}}_{\omega, \Lambda}$ and $\mathfrak{L}^{\mathrm{par}}_{\omega, \Omega}.$

\subsubsection{Matrix-valued functions with non-uniform polynomial decay} \label{sc:matrix_nonuniform_decay}
 We consider spaces of functions taking values in $\mathcal{M}_{2m}(\R)$ whose components exhibit different time-decay properties. For this purpose, we introduce the following notation.  Given $M \in \mathcal{M}_{2m}(\R)$, we denote
\begin{equation}\label{def:Matrix_block}
    M=\begin{pmatrix} M_1 & M_2 \\ M_3 & M_4 \end{pmatrix},
\end{equation}
with $M_k \in \mathcal{M}_m(\R)$ for $k=1,\dots,4$.

For any $\mathsf{p} \geq 1$ we define
\begin{equation}
\label{eq:general_matrix_space}
     \mathcal{M}^{ T}_{\sigma, \ell}(\mathsf{p}) = \left\{M:\T^n \times I_T \to\mathcal{M}_{\mathsf{p}}(\R) \hspace{1mm} \left| \hspace{1mm}  \forall 1 \leq i, j \leq \mathsf{p}, \quad M_{ij}  \in \mathcal{S}^{ T}_{(\sigma,0), \ell} \right\}\right.,
\end{equation}
endowed with the norm
\[ |M|^{ T}_{\sigma, \ell} := \max_{1 \leq i, j \leq N} |M_{ij}|^{ T}_{\sigma, \ell},\]
where for the sake of simplicity we omit the dependence on $\mathsf{p}$ in the notation for the norm. We will only consider these spaces in the cases $\mathsf{p} = m, 2m$. For this reason, when there is no risk of confusion and it is clear to what space we are referring to, we will omit $\mathsf{p}$ in the notation of these spaces.

The notation introduced in \eqref{def:Matrix_block} extends naturally to functions in the spaces $\mathcal{M}^{ T}_{\sigma, \ell}(\mathsf{p})$.

 Given $\mathcal{D} = (d_1, d_2, d_3, d_4) \in \left(\R_{\ge 0}\right)^4$,  we define 
\begin{equation}\label{def:BS_par_M_Sym}
    \begin{aligned}
        \mathcal{M}^{ T}_{\sigma, \mathcal{D}} &= \left\{M : \T^n \times I_T \to \mathcal{M}_{2m}(\R)\hspace{1mm} \left| \hspace{1mm}  \forall 1\leq k \leq 4,  \quad  M_k  \in \mathcal{M}^{ T}_{\sigma, d_k}(m)\right\}\right. ,\\
         \mathcal{S}ym^{ T}_{\sigma, \mathcal{D}} &= \left\{M:\T^n \times I_T \to \textup{Sym}(2m, \R) \hspace{1mm} \left| \hspace{1mm}  M \in \mathcal{M}^{ T}_{\sigma, \mathcal{D}} \right\}\right.,
    \end{aligned}
\end{equation}
endowed with the norm
\begin{equation}\label{def:norm_M_D}
    |M|^{ T}_{\sigma, \mathcal{D}} := \max_{1 \leq k \leq 4} |M_k|^{ T}_{\sigma, d_k},
\end{equation}
and
\begin{equation}\label{def:BS_par_Sp}
 \mathcal{S}p^{ T}_{\sigma, \omega, \mathcal{D}}   = \left\{M:\T^n \times I_T \to \mathfrak{sp}(2m, \R) \hspace{1mm} \left| \hspace{1mm} M \in \mathcal{M}^{ T}_{\sigma, \mathcal{D}}, \quad \mathcal{L}_{\omega} M \in \mathcal{M}^{ T}_{\sigma, \mathcal{D}+1}  \right\} \right.,
\end{equation}
endowed with the norm
\[ |M|^{ T}_{\sigma, \omega, \mathcal{D}} = \max\left \{ |M|_{\sigma, \mathcal{D}}^{ T} ,|\mathcal{L}_{\omega} M|^{ T}_{\sigma, \mathcal{D}+1} \right\}.\]
where in the latter $\mathcal{D} + 1= (d_1+1, d_2+1, d_3+1, d_4+1)$. 

\begin{remark}
    We point out that the space $\mathcal{S}p^{ T}_{\sigma, \omega, \mathcal{D}} $ in~\eqref{def:BS_par_Sp} is not merely an extension, to the setting where the matrix components exhibit different polynomial decay in time, of the analogous space used in the elliptic case in Part I \cite{scarcella_KAMST}. Indeed, the definition of $\mathcal{S}p^{ T}_{\sigma, \omega, \mathcal{D}} $ involves the linear operator $\Lo$, whereas its elliptic counterpart involves $\LO^\mathrm{ell}$.
    The reason for defining them this way will become clear later.
\end{remark}

\subsubsection{Perturbations with non-uniform polynomial decay on $z = (x,y)$} \label{sc:perturbation_par}
To describe the space of perturbations considered in Theorem \ref{Thm:par_Csigma}, given $\mathcal{L} = (\ell_1, \ell_2, \ell_3, \ell_4, \ell_5, \ell_6, \ell_7) \in \left(\R_{\ge 0}\right)^7$ and $T \ge 1$, we denote the space of perturbations by
\begin{equation}\label{def:Per_space_P_par_Holder}
\begin{aligned}
    \mathcal{P}^{\mathrm{par},T}_{\sigma, \Lell}  & =  \mathcal{S}^{T}_{(\sigma,2), (0,\ell_1)}(\R) \times \mathcal{S}^{T}_{(\sigma,2), \ell_2}(\R^n) \times \mathcal{S}^{T}_{(\sigma,2), \ell_3}(\R^m) \times \mathcal{S}^{T}_{(\sigma,2), \ell_4}(\R^m) \\ & \quad \times \mathcal{S}^{T}_{(\sigma,2), \ell_5}(\R^{m \times m}) \times \mathcal{S}^{T}_{(\sigma,2), \ell_6}(\R^{2m}) \times \mathcal{S}^{T}_{(\sigma,2), \ell_7}(\R^{m \times m})
\end{aligned}
\end{equation}
and we point out that $\mathcal{P}^{\mathrm{par},T}_{\sigma, \Lell} $, endowed with the sum norm associated with the product, which we denote by $|\cdot|^{\mathrm{par}, T}_{\sigma, \mathcal{L}}$, is a Banach space. We refer to Section~\ref{sc:real_pol_decay} for the definition of the spaces $\big(\mathcal{S}^{T}_{(\sigma,k), \ell}, |\cdot|^{ T}_{\sigma + k, \ell}\big)$ and $\big(\mathcal{S}^{T}_{(\sigma,k), (0,\ell)}, |\cdot|^{ T}_{\sigma + k, (0,\ell)}\big)$ appearing in the definition of $\mathcal{P}^{\mathrm{par},T}_{\sigma, \Lell}$. 

By a slight abuse of notation, we write $$P = (a, b, c_1, c_2, d_1, d_2, d_3) \in \mathcal{P}^{\mathrm{par},T}_{\sigma, \Lell},$$ and use the same symbol $P$ to denote the corresponding Hamiltonian defined in~\eqref{eq:perturbation_form_par}, namely,
\[\begin{aligned}
P(q, p, z, t) & = a(q, t) + b(q, t) \cdot p + c_1(q, t) \cdot x + c_2(q,t) \cdot y  \\ & \quad + \frac{1}{2}d_1(q, t)\cdot x^2 + d_2(q,t)\cdot (x,y) + {1 \over 2} d_3(q,t) \cdot y^2.
\end{aligned}
\]

For any $\ell, \dec \ge 0$, we denote
\begin{equation}
    \begin{aligned}
    \label{eq:index_regularity_perturbations_par}
    \mathcal{L}(\ell, \dec) &:=  (\ell + \dec +4, \ell, \ell + \dec +2, \ell+ \dec + 3, \ell - 1, \ell, \ell+1), \\
     \mathcal{L}^*(\ell, \dec)  &:= (\ell + \dec + 6, \ell, \ell + \dec + 4, \ell+\dec + 5, \ell+\dec, \ell+\dec+2, \ell+\dec+3).
    \end{aligned}
\end{equation} 

With these notations, the spaces of perturbations associated with Items Items~\eqref{thm:C_existence} and~\eqref{thm:C_transverse} of Theorem \ref{Thm:par_Csigma} are given by  $\mathcal{P}^{\mathrm{par}, 1}_{\sigma, \Lell(\ell, \dec)}$ and $\mathcal{P}^{\mathrm{par}, 1}_{\sigma, \Lell^*(\ell, \dec)}$, respectively

\subsubsection{Torus embeddings with non-uniform polynomial decay on $z = (x,y)$} \label{sc:embeddings_par}
to describe the space of torus embeddings (to which the asymptotic KAM tori given by the theorem will belong), for any $\mathcal{K} = (k_1, k_2, k_3, k_4) \in \left(\R_{\ge 0}\right)^4$ and $T\ge 1$, we define the following Banach space
\begin{equation*}
    \mathcal{E}^{\mathrm{par},T}_{\sigma, \mathcal{K}} = \mathcal{U}^{T}_{\sigma, \omega, k_1} \times \mathcal{U}^{T}_{\sigma, \omega, k_2} \times \mathcal{U}^{T}_{\sigma, \omega, k_3} \times \mathcal{U}^{T}_{\sigma, \omega, k_4},
\end{equation*}
endowed with the sum norm associated with the product that we denote by $|\cdot|^{\mathrm{par},T}_{\sigma, \mathcal{K}}$.

By a slight abuse of notation, we will use the same symbol for the element $\varphi = (u,v,w_x, w_y) \in \mathcal{E}^{\mathrm{par},T}_{\sigma, \mathcal{K}}$ and the associated family of embedded tori $\varphi : \T^n \times I_T \to \T^n \times \R^n \times \R^n \times \R^n$ defined by $\varphi(q,t) = (q + u(q,t), v(q,t), w_x(q,t), w_y(q,t))$ defining the $C^\sigma$ asymptotic KAM torus that we are looking for. 

Given $k \ge 1$ and $\mathsf{k} \ge 0$, we denote 
\begin{equation}\label{eq:index_regularity_uwx_par} 
    \mathcal{K}(k, \mathsf{k}) := (k-1, k + \mathsf{k} + 3, k + \mathsf{k} + 2, k + \mathsf{k} + 1)
\end{equation}

As we shall see in Section \ref{sec:proof_item_1_par_Holder}, for any $P \in  \mathcal{P}^{\mathrm{par}, 1}_{\sigma, \Lell(\ell, \dec)}$ with $\ell > 1$ and $\dec > 0$ there exists a $C^\sigma$ asymptotic KAM torus in $ \mathcal{E}^{\mathrm{par}, T}_{\sigma, \mathcal{K}(\ell, \dec)}$ associated with $(X_{H_0 + P}, X_{H_0}, \varphi_0)$, for some $T \geq 1$.

\subsection{Initial Hamiltonian in Theorem \ref{Thm:par_Csigma}}\label{sc:initial_ham_par}

For the rest of this section,  we consider the Hamiltonian $H_0:\T^n \times B \times I_1 \to \R$ of the form~\eqref{eq:initial_hamiltonian_par} as in the statement of Theorem \ref{Thm:par_Csigma}. We recall that $I_T=[T, +\infty)$, for a given $T>0$,  and $B \subset \R^{n +2m}$ is an open ball centered at the origin that, for simplicity, we assume to have radius $1$. The two statements of Theorem \ref{Thm:par_Csigma} are proved separately in Sections \ref{sec:proof_item_1_par_Holder} and \ref{sec:proof_item_2_par_Holder}, respectively. We assume $\sigma \ge 1$ in Section \ref{sec:proof_item_1_par_Holder} and $\sigma \ge 2$ in Section \ref{sec:proof_item_2_par_Holder}.

Since all the norms considered in this section concern (products of) function spaces with polynomial decay, for the sake of clarity and to simplify the notation, we omit the superscripts $\mathrm{par}$ from all the norms whenever there is no risk of confusion.

We can rewrite the Hamiltonian $H_0:\T^n \times B \times I_1 \to \R$ in~\eqref{eq:initial_hamiltonian_par} in as
\begin{equation}\label{eq:initial_Hamiltonian_NR_par}
        H_0(q,p,x,y,t) = N(p,x) + R(q,p,x,y,t)
\end{equation}
where 
\begin{align*}
    N(p,x) &= \omega \cdot p + {1 \over 2} \Omega \cdot x^2,\\
    R(q,p,x,y,t) &= M(q,p,x,y,t) \cdot p^2 + m_1(q,x,y,t) \cdot (p,x) + m_2(q,y,t) \cdot (p,y)\\
    &+ L_1(q,x,y,t) \cdot x^3 + L_2(q,y,t) \cdot (x, x, y) + L_3(q,y,t) \cdot (x, y, y) + L_4(q,y,t) \cdot y^3 
\end{align*}
with 
\begin{equation}\label{eq:barMmL_par}
\begin{aligned}
     M(q,p,x,y,t) &= \int_0^1 (1 -\tau) \partial_p^2 H(q, \tau p, x,y,t) d\tau ,\\
    m_1(q,x,y,t) &= \int_0^1 \partial_p\partial_x H(q, 0, \tau x, y,t) d\tau, \\
    m_2(q,x,y,t) &= \int_0^1 \partial_p\partial_y H(q, 0, 0,\tau y,t) d\tau ,\\
    L_1(q,x,y,t) &= {1 \over 2} \int_0^1 (1-\tau)^2 \partial_x^3 H (q,0,\tau x, y, t) d\tau, \\
    L_2(q,y,t) &= {1 \over 2} \int_0^1 \partial_y\partial_x^2 H (q,0,0,\tau y, t) d\tau,\\
    L_3(q,x,y,t) &= \int_0^1 (1-\tau) \partial_y^2\partial_x H (q,0,0,\tau y, t) d\tau, \\
    L_4(q,y,t) &= {1 \over 2} \int_0^1(1-\tau)^2 \partial_y^3 H (q,0,0,\tau y, t) d\tau.
\end{aligned}
\end{equation}
 We recall that the RHS of each identity in~\eqref{eq:barMmL_par} is unchanged when $H_0$ is replaced by $R$.
We stress that, in the latter, $L_2(q,y,t) \cdot (x, x, y)$ and $L_3(q,y,t) \cdot (x, y, y)$ stand for the trilinear forms $L_2(q,y,t)$ and $L_3(q,y,t)$ evaluated in $(x,x,y)$ and $(x,y,y)$, respectively. 

By assumption, we have that $\partial_{(p,x,y)}^2 H_0 \in \mathcal{S}_{(\sigma,3),0}^{1}$. Hence, we fix $\Upsilon>0$ in such a way that 
 \begin{equation}\label{def:const_Upsilon_par}
     |\partial_{(p,x,y)}^2 H_0 |_{\sigma+3, 0}^1 \le \Upsilon
 \end{equation}
and using the definition~\eqref{eq:barMmL_par}, one can verify that  
\begin{equation*}
    |M|^1_{\sigma+3, 0}, \, |m_1|^1_{\sigma+3, 0},  \, |m_2|^1_{\sigma+3, 0},  \, |L_1|^1_{\sigma+2, 0},  \, |L_2|^1_{\sigma+2, 0}, \, |L_3|^1_{\sigma+2, 0},  \, |L_4|^1_{\sigma+2, 0} \le |\partial_{(p,x,y)}^2 H_0 |_{\sigma+3, 0}^1 \le \Upsilon.
\end{equation*}

We recall that the trivial embedding $\varphi_0$ in~\eqref{def:varphi0=(q,0,0)} is an invariant torus for $H_0$ supporting quasiperiodic solutions of frequency vector $\omega$.  

 In order to control the regularity of the higher-order terms of the Hamiltonian~\eqref{eq:initial_Hamiltonian_NR_par}, we have the following
\begin{lemma}\label{lemma:def_bar_M_m_L_par}
    For all $(q,p,z,t) \in \T^n \times B \times I_{1}$, we define
    \begin{align*}
       \overline{M}(q,p,x,y,t) &= \int_0^1 \partial^2_p H(q, \tau p , x,y, t) d \tau, \hspace{10mm} \overline{\overline{M}}(q,p,x,y,t) = \partial_p^2 H(q,p,x,y,t),\\
       \overline{m}_1(q,x,y,t) &=  \partial_p\partial_x H(q, 0 , x,y, t), \hspace{20mm} \overline{m}_2(q,y,t) =  \partial_p\partial_y H(q, 0 , 0,y, t), \\
       \overline{L}_1(q,x,y,t) &= \int_0^1 (1 -\tau)\partial_x^3 H(q, 0 , \tau x, y, t) d \tau, \quad \overline{\overline{L}}_1 (q,x,y,t) = \int_0^1 \partial_x^3 H(q,0, \tau x, y, t)d\tau,\\
       \overline{L}_2(q,y,t) &= {1 \over 2} \partial_y \partial_x^2 H(q,0,0,y,t), \hspace{20mm} \overline{L}_3(q,y,t) = \int_0^1 \partial_y^2\partial_x H(q, 0 , 0, \tau y, t) d \tau, \\
       \overline{\overline{L}}_3(q,y,t) &=\partial_y^2\partial_x H(q, 0 , 0, y, t), \hspace{20mm} \overline{L}_4(q,y,t) = \int_0^1 (1 -\tau)\partial_y^3 H(q, 0 , 0, \tau y, t) d \tau,\\
       \overline{\overline{L}}_4(q,y,t) &= \int_0^1\partial_y^3 H(q, 0 , 0, \tau y, t) d \tau.
    \end{align*}
    Then,
    \begin{equation*}
        \begin{aligned}
            &\overline{M}(q,p,x,y,t)\cdot p = \partial_p\left(M(q,p,x,y,t) \cdot p^2\right), \hspace{10mm} \overline{\overline{M}}(q,p,x,y,t) = \partial_p \left(\overline{M}(q,p,x,y,t)\cdot p\right)\\
            &\overline{m}_1(q,x,y,t)\cdot p = \partial_x\left(m_1(q,x,y,t) \cdot (p,x)\right),\quad \overline{m}_2(q,y,t)\cdot p = \partial_y\left(m_2(q,y,t) \cdot (p,y)\right),\\
            &\overline{L}_1(q,x,y,t)\cdot x^2 = \partial_x\left(L_1(q,x,y,t) \cdot x^3\right), \hspace{8mm}  \overline{\overline{L}}_1(q,x,y,t)\cdot x = \partial_x\left(\overline{L}_1(q,x,y,t) \cdot x^2\right)\\
            &\overline{L}_2(q,y,t)\cdot x^2 = \partial_y\left(L_2(q,y,t) \cdot (x, x, y)\right), \hspace{8mm} \overline{\overline{L}}_3(q,y,t)\cdot x = \partial_y\left(\overline{L}_3(q,y,t) \cdot (x, y)\right)\\
            &\overline{L}_3(q,y,t)\cdot (x, y) = \partial_y\left(L_3(q,y,t) \cdot (x, y, y)\right), \quad  \overline{L}_4(q,y,t)\cdot y^2 = \partial_y\left(L_4(q,y,t) \cdot  y^3\right),\\
            &\overline{\overline{L}}_4(q,y,t)\cdot y = \partial_y\left(\overline{L}_4(q,y,t) \cdot  y^2\right),
        \end{aligned}
    \end{equation*}
    for all $(q,p,x,y,t) \in \T^n \times B \times I_1$. Moreover, letting $\sigma \ge 0$, $T \ge 1$ and an integer $k \ge 1$, if $\partial_{(p,x,y)}^2 R \in \mathcal{S}_{(\sigma,k),0}^{T}$ 
    then $\overline{M} ,\ \overline{\overline{M}} ,\ \overline{m}_1 ,\ \overline{m}_2 \in \mathcal{S}_{(\sigma,k),0}^{T}$, and $\overline{L}_1,\ \overline{\overline{L}}_1 ,\ \overline{L}_2 ,\ \overline{L}_3 ,\ \overline{\overline{L}}_3 ,\ \overline{L}_4 ,\ \overline{\overline{L}}_4  \in \mathcal{S}_{(\sigma,k-1),0}^{T}$ with 
    \begin{equation*}
     \begin{aligned}
        &|\overline{M}|^T_{\sigma+k, 0} ,\ |\overline{\overline{M}}|^T_{\sigma+k, 0} ,\ |\overline{m}_1|^T_{\sigma+k, 0} ,\ |\overline{m}_2|^T_{\sigma+k, 0}  \le |\partial_{(p,z)}^2 H|^T_{\sigma+k, 0},\\
        &|\overline{L}_1|^T_{\sigma+k-1, 0} ,\ |\overline{\overline{L}}_1|^T_{\sigma+k-1, 0} ,\ |\overline{L}_2|^T_{\sigma+k-1, 0} ,\ |\overline{L}_3|^T_{\sigma+k-1, 0} \le |\partial_{(p,z)}^2 H|^T_{\sigma+k, 0},\\
        &|\overline{\overline{L}}_3|^T_{\sigma+k-1, 0} \,  |\overline{L}_4|^T_{\sigma+k-1, 0}  ,\ |\overline{\overline{L}}_4|^T_{\sigma+k-1, 0} \le |\partial_{(p,z)}^2 H|^T_{\sigma+k, 0}.
        \end{aligned}
    \end{equation*}
      In addition, letting $0<\sigma'<\sigma$ and $T \ge 1$, if $\partial_{(p,x,y)}^2 R \in \mathscr{S}_{\sigma,0}^{T}$, then $\overline{M} ,\ \overline{\overline{M}} ,\ \overline{m}_1 ,\ \overline{m}_2 \in \mathscr{S}_{\sigma,0}^{T}$, and $\overline{L}_1,\ \overline{\overline{L}}_1 ,\ \overline{L}_2 ,\ \overline{L}_3 ,\ \overline{\overline{L}}_3 ,\ \overline{L}_4 ,\ \overline{\overline{L}}_4  \in \mathscr{S}_{\sigma',0}^{T}$ with
     \begin{equation*}
     \begin{aligned}
        &|\overline{M}|^T_{\sigma, 0} ,\ |\overline{\overline{M}}|^T_{\sigma, 0} ,\ |\overline{m}_1|^T_{\sigma, 0} ,\ |\overline{m}_2|^T_{\sigma, 0}  \le |\partial_{(p,z)}^2 H|^T_{\sigma, 0},\\
        &|\overline{L}_1|^T_{\sigma', 0} ,\ |\overline{\overline{L}}_1|^T_{\sigma', 0} ,\ |\overline{L}_2|^T_{\sigma', 0} ,\ |\overline{L}_3|^T_{\sigma', 0}, \, |\overline{\overline{L}}_3|^T_{\sigma', 0} \,  |\overline{L}_4|^T_{\sigma', 0}  ,\ |\overline{\overline{L}}_4|^T_{\sigma', 0} \le {1 \over \sigma-\sigma'} |\partial_{(p,z)}^2 H|^T_{\sigma, 0}.
        \end{aligned}
    \end{equation*}
\end{lemma}
\begin{proof}
    The proof is similar to that of Lemma 5.1 in Part I \cite{scarcella_KAMST}.
\end{proof}

In what follows, we fix $\ell >1$ and $\dec \ge 0$, and for the sake of simplicity, we denote $\mathcal{L}(\ell, \dec)$, $\mathcal{L}^*(\ell, \dec)$ and $\mathcal{K}(\ell, \dec)$ (as in \eqref{eq:index_regularity_perturbations_par} and \eqref{eq:index_regularity_uwx_par} ) simply by $\mathcal{L}$, $\mathcal{L}^*$ and $\mathcal{K}$, respectively.

As mentioned before, we will prove the two statements of Theorem \ref{Thm:par_Csigma} separately in the following two sections. The structure of the proof is similar to that of Section \ref{sec:Proof_Theorem_Par_0}. %

\subsection{Proof of Item \eqref{thm:C_existence} of Theorem \ref{Thm:par_Csigma}: Construction of the $C^\sigma$ asymptotic KAM torus} \label{sec:proof_item_1_par_Holder}
In the following, we will use freely the notations introduced at the beginning of this section and in the previous two subsections.

Given $r, \rho>0$, let $B_{\mathcal{P}^{\mathrm{par},T}_{\sigma, \mathcal{L}}}(r)  \subseteq  \mathcal{P}^{\mathrm{par},T}_{\sigma, \mathcal{L}}$ and $B_{ \mathcal{E}^{\mathrm{par},T}_{\sigma, \mathcal{K}}}(\rho) \subseteq \mathcal{E}^{\mathrm{par},T}_{\sigma, \mathcal{K}}$ denote the balls of radius $r$ and $\rho$, respectively, centered at the origin. In the spirit of the proof of Item \eqref{thm:Cdeg_existence} of Theorem \ref{Thm:par_Csigma_0}, contained in Section \ref{sec:Proof_Theorem_Par_0}, we will show the following.

\begin{proposition}
\label{prop:C_existence}
There exists $C_0 > 1$ such that for any $r > 0$ there exists $T_0 \geq 1$ satisfying the following. For any $T \geq T_0$, there exists a $C^1$-map
\[ \boldsymbol{\varphi}^T: B_{\mathcal{P}^{\mathrm{par},T}_{\sigma, \mathcal{L}}}(r)  \subseteq  \mathcal{P}^{\mathrm{par},T}_{\sigma, \mathcal{L}}  \to  B_{ \mathcal{E}^{\mathrm{par},T}_{\sigma, \mathcal{K}}}(C_0r) \subseteq \mathcal{E}^{\mathrm{par},T}_{\sigma, \mathcal{K}}\]
such that $\boldsymbol{\varphi}^T(P)$ defines an asymptotic KAM torus associated with $(X_{H_0 + P}, X_{H_0}, \varphi_0)$,  for any $P \in  \mathcal{P}^{\mathrm{par},T}_{\sigma, \mathcal{L}} $ with $|P|^T_{\sigma, \mathcal{L}} < r$.
\end{proposition}

Notice that Item \eqref{thm:C_existence} of Theorem \ref{Thm:par_Csigma} follows directly from the proposition above since
\begin{equation*}
    \big|P\big|_{\sigma, \Lell}^{T} \leq |P|_{\sigma, \Lell}^{1}, \qquad \text{ for any } T \geq 1\,  \text{ and any } \, P \in \mathcal{P}^{\mathrm{par}, 1}_{\sigma, \Lell},
\end{equation*}
and by definition of $\mathcal{K}$ the asymptotic KAM torus has the desired decay rates.

To prove Proposition \ref{prop:C_existence}, we consider the following functional
\begin{equation}\label{proof:Csigma_def_F_par_1}
\Function{\mathcal{F}^T}{ \mathcal{P}^{\mathrm{par},T}_{\sigma,\mathcal{L}} \times \mathcal{E}^{\mathrm{par},T}_{\sigma, \mathcal{K}}}{ \mathcal{R}^{T}_{\sigma, \mathcal{L}}}{(P, \varphi)}{X_{H_0 + P} \circ \varphi - \Lo \varphi },
\end{equation}
where 
\begin{equation}
    \label{def:F_codomain}
\mathcal{R}^{T}_{\sigma, \mathcal{L}} = \mathcal{S}^{T}_{(\sigma,0), \ell}(\R^n) \times \mathcal{S}^{T}_{(\sigma,0), \ell+\dec+4}(\R^n) \times \mathcal{S}^{T}_{(\sigma,0), \ell+\dec+3}(\R^m) \times \mathcal{S}^{T}_{(\sigma,0), \ell+\dec+2}(\R^m),
\end{equation}
which we endow 
with the sum norm associated with the product that we denote $|\cdot|_{\boldsymbol{{\sigma, \ell, \dec}}}^T$, where the parameters are written in bold symbols.

Notice that this functional is defined by the same formula as the one introduced in Section \ref{sec:Proof_Theorem_Par_0} and given in \eqref{proof:Csigma_def_F_par_deg_1}, but with a different unperturbed Hamiltonian $H_0$, and a different codomain of functions with polynomial decay. Moreover, by the arguments given in Section \ref{sec:Proof_Theorem_Par_0} it follows that, for $(P, \varphi)$ in the domain of definition, $\mathcal{F}^T(P, \varphi) = 0 $ if and only if $\varphi$ defines a $C^\sigma$ asymptotic KAM torus for $H_0 + P$.

The following proposition contains a series of properties satisfied by the functional $\mathcal{F}^T$ defined in~\eqref{proof:Csigma_def_F_par_1}. For the rest of this section, we denote by $D_\varphi \mathcal{F}^T$ the differential of the functional $\mathcal{F}^T$ with respect to $\varphi=(u,v,w_x, w_y)$. 

\begin{proposition}
\label{prop:F_properties_par}
    The operator $\mathcal{F}^T$ given by \eqref{proof:Csigma_def_F_par_1} satisfies the following.
    \begin{enumerate}
    \item \label{prop:F_well_defined_par} For any $T \geq 1$, there exists a neighbourhood $\mathcal{U}^T \subseteq \mathcal{P}^{\mathrm{par},T}_{\sigma,\mathcal{L}} \times \mathcal{E}^{\mathrm{par},T}_{\sigma, \mathcal{K}}$ of $(0, 0)$ such that $\mathcal{F}^T\mid_{\mathcal{U}^T}$ and $D_\varphi\mathcal{F}^T\mid_{\mathcal{U}^T}$ are well-defined continuous maps. 
    
    \noindent Moreover, for any $r, \rho \geq 0$ there exists $T_0 \geq 1$ such that, for any $T \geq T_0$,  
    \[B_{\mathcal{P}^{\mathrm{par},T}_{\sigma,\mathcal{L}}}(r) \times B_{\mathcal{E}^{\mathrm{par},T}_{\sigma, \mathcal{K}}}(\rho) \subseteq \mathcal{U}^T.\]

         \item \label{prop:F_bounded_par}There exists a constant $C$,
     such that for any $r > 0$ and any $T \geq 1$,
\[ \sup \left\{ |\mathcal{F}^T(P, 0)|^T_{\boldsymbol{{\sigma, \ell, \dec}}} \,\left|\, P \in B_{\mathcal{P}^{\mathrm{par},T}_{\sigma,\Lell}}(r); \,\, (P, 0) \in \mathcal{U}^T \right\}\right. \leq Cr.\] 

    \item \label{prop:DF_invertible_par} $D_\varphi \mathcal{F}^T(0, 0)$ is invertible, for any $T \geq 1$. Moreover, there exists a constant $\bar C$, independent of $T$, such that $\|D_\varphi \mathcal{F}^T(0, 0)^{-1}\| \leq \bar C$, where $\| \cdot\|$ denotes the operator norm.

        \item \label{prop:DF_limit_par} Let $r, \rho > 0$. Then  
        \[\lim_{T \to +\infty} \| D_\varphi\mathcal{F}^T(P, \varphi) - D_\varphi\mathcal{F}^T(0, 0) \| = 0, \qquad \text{uniformly on } \quad B_{\mathcal{P}^{\mathrm{par},T}_{\sigma,\mathcal{L}}}(r) \times B_{\mathcal{E}^{\mathrm{par},T}_{\sigma, \mathcal{K}}}(\rho),\]
where $\| \cdot\|$ denotes the operator norm. 
    \end{enumerate}
\end{proposition}

The proof of Proposition \ref{prop:C_existence} is a straightforward consequence of Proposition \ref{prop:F_properties_par} and the quantitative version of the Implicit Function Theorem provided by Theorem \ref{thm:QIFT}. It is similar to the proof of Proposition \ref{prop:C_deg_existence} presented in Section \ref{sec:proof_item_1_par_deg_Holder}. For this reason, it is omitted.

The rest of this section is dedicated to the proof of Proposition \ref{prop:F_properties_par}, %
which is divided in the following four lemmas (Lemmas \ref{lemma:F_well_defined_par}, \ref{lemma:F_bounded_par}, \ref{lemma:DF_invertible_par} and \ref{lem:DF_limit_par}) each addressing one of the properties appearing in the statement.

    \begin{remark}
        We point out that we can prove Lemmas \ref{lemma:F_well_defined_par}, \ref{lemma:F_bounded_par}, and \ref{lemma:DF_invertible_par} under the slightly weaker decay assumption 
        $d_1 \in \mathcal{S}_{(\sigma, 2), 0}^{ 1}, \, d_2 \in \mathcal{S}_{(\sigma, 2), 1}^{ 1}, \, d_3 \in \mathcal{S}_{(\sigma, 2), 2}^{ 1}$ instead of $d_1 \in \mathcal{S}_{(\sigma, 2), \ell-1}^{ 1}, \, d_2 \in \mathcal{S}_{(\sigma, 2), \ell}^{ 1}, \, d_3 \in \mathcal{S}_{(\sigma, 2), \ell+1}^{ 1}$. We need to assume this stronger decay assumption in order to prove Lemma \ref{lem:DF_limit_par}.
    \end{remark}

\begin{lemma}\label{lemma:F_well_defined_par}
   Let $T \geq 1$. The functional $\mathcal{F}^T$ given by \eqref{proof:Csigma_def_F_par_1} and its differential $D_\varphi \mathcal{F}^T$ are well-defined continuous maps on $$\mathcal{U}^T =  \mathcal{P}^{\mathrm{par},T}_{\sigma,\mathcal{L}} \times B_{\mathcal{E}^{\mathrm{par},T}_{\sigma, \mathcal{K}}}(T^{\ell-1}).$$
\end{lemma}
\begin{proof}
    We fix $(P, \varphi) \in \mathcal{P}^{\mathrm{par},T}_{\sigma,\mathcal{L}}\times B_{\mathcal{E}^{\mathrm{par},T}_{\sigma, \mathcal{K}}}(T^{\ell-1})$ and recalling the definition of the norm $|\cdot|^T_{\sigma, \omega, \ell}$ in~\eqref{def:norm_U}, we have that 
    \begin{equation*}
        |u|_{C^0} < 1, \quad |v|_{C^0} < {1 \over T^{4 + \dec}}, \quad |w_x|_{C^0} < {1 \over T^{3 +\dec}}, \quad |w_y|_{C^0} < {1 \over T^{2 + \dec}}
    \end{equation*}
    and hence the composition $X_{H_0 + P} \circ \varphi$ (and therefore $\mathcal{F}^T(P, \varphi)$) is well-defined.

    We observe that we can explicitly write the functional $\mathcal{F}^T$ as
    \begin{equation*}
    \mathcal{F}^T  = \begin{pmatrix}\partial_p H \circ \varphi - \omega -\Lo u\\
    -\partial_q H\circ \varphi - \Lo v\\
    \partial_y H \circ \varphi  -\Lo w_x\\
    -\partial_x H\circ \varphi - \Lo w_y\\\end{pmatrix}
\end{equation*}
with 
\begin{equation}
\begin{aligned}\label{eq:terms_F_par_1}
    \partial_p H \circ \varphi - \omega -\Lo u &= b\circ \varphi_q + \overline{M} \circ \varphi \cdot  v + m_1 \circ \varphi_{qxy}\cdot w_x + m_2 \circ \varphi_{qy}\cdot w_y - \Lo u,\\
    -\partial_q H\circ \varphi - \Lo v &= -\partial_q a \circ \varphi_q - \partial_q b\circ \varphi_q\cdot v -  \partial_q c_1\circ \varphi_q \cdot w_x -  \partial_q c_2\circ \varphi_q \cdot w_y\\
    &- {1 \over 2}\partial_q d_1\circ \varphi_q \cdot (w_x)^2 - \partial_q d_2\circ \varphi_q \cdot (w_x, w_y) - {1 \over 2}\partial_q d_3\circ \varphi_q \cdot (w_y)^2\\
    &- \partial_q M\circ \varphi
    \cdot (v)^2 - \partial_q m_1 \circ \varphi_{qxy}\cdot (v,w_x) - \partial_q m_2 \circ \varphi_{qy}\cdot (v,w_y)\\
    &- \partial_q L_1 \circ \varphi_{qxy} \cdot (w_x)^3 - \partial_q L_2 \circ \varphi_{qy} \cdot (w_x, w_x, w_y) \\
    &- \partial_q L_3 \circ \varphi_{qy} \cdot (w_x, w_y, w_y) - \partial_q L_4 \circ \varphi_{qy} \cdot (w_y)^3 - \Lo v \\ 
    \partial_y H\circ \varphi - \Lo w_x &= c_2 \circ \varphi_q + d_2  \circ \varphi_q \cdot w_x + d_3  \circ \varphi_q \cdot w_y + \partial_y M  \circ \varphi \cdot (v)^2 \\
    &+ \partial_y m_1 \circ \varphi_{qxy}\cdot (w_x, v) + \overline{m}_2 \circ \varphi_{qy} \cdot v + \partial_y L_1 \circ \varphi_{qxy}\cdot (w_x)^3\\
    &+\overline{L}_2 \circ \varphi_{qy}\cdot (w_x)^2 + \overline{L}_3\circ \varphi_{qy} \cdot (w_x, w_y) + \overline{L}_4\circ \varphi_{qy} \cdot (w_y)^2 - \Lo w_x,\\
    -\partial_x H\circ \varphi - \Lo w_y &= - c_1 \circ \varphi_q - d_1\circ \varphi_q\cdot w_x - d_2 \circ \varphi_q \cdot w_y - \partial_x M \circ \varphi \cdot (v)^2\\
    &- \overline{m}_1 \circ \varphi_{qxy} \cdot v - \overline{L}_1 \circ \varphi_{qxy} \cdot (w_x)^2 -  2 L_2\circ \varphi_{qy}\cdot (w_x, w_y)\\
    &-L_3\circ \varphi_{qy} \cdot (w_y)^2 - \Lo w_y - \Lambda \cdot w_x. 
    \end{aligned}
\end{equation}
where $\overline{M}, \, \overline{m}_1, \, \overline{m}_2, \, \overline{L}_1, \, \overline{L}_2, \, \overline{L}_3, \, \overline{L}_4$ are defined in Lemma \ref{lemma:def_bar_M_m_L_par}. 
We recall that we denote the projections of $\varphi$ onto the $q$, $p$, $x$, and $y$ components by indices. Furthermore, we recall that we are adopting the convention that $\circ \varphi$ is used as a shorthand for $\circ (\varphi(q,t), t)$.

Using the properties of Proposition \ref{prop:Csigma_prop_norms_pol}, it is easy to check that the image of $\mathcal{F}^T$ is contained in $\mathcal{R}^T_{\sigma, \mathcal{L}} = \mathcal{S}^{T}_{(\sigma,0), \ell}(\R^n) \times \mathcal{S}^{T}_{(\sigma,0), \ell+\dec+4}(\R^n) \times \mathcal{S}^{T}_{(\sigma,0), \ell+\dec+3}(\R^m) \times \mathcal{S}^{T}_{(\sigma,0), \ell+\dec+2}(\R^m)$. 

Concerning the regularity of $\mathcal{F}^T$ and $D_\varphi \mathcal{F}^T$ it is a straightforward consequence of~\eqref{eq:terms_F_par_1} and the regularity assumptions on the perturbative terms and the components of the $C^\sigma$ asymptotic KAM torus we are looking for  $P=(a,b,c_1,c_2,d_1,d_2,d_3) \in \mathcal{P}^{\mathrm{par},T}_{\sigma,\mathcal{L}}$ and $\varphi = (u,v,w_x, w_y) \in \mathcal{E}^{\mathrm{par},T}_{\sigma, \mathcal{K}}$.
\end{proof}

\begin{lemma}
    \label{lemma:F_bounded_par}
   There exists a constant $C$, such that for any $r > 0$ and any $T \geq 1$,
\[ \sup \left\{ |\mathcal{F}^T(P, 0)|^T_{\boldsymbol{\sigma, \ell, \dec}} \mid P \in \mathcal{P}^{\mathrm{par},T}_{\sigma,\mathcal{L}}, \, |P|_{\sigma, \Lell}^T < r \right\} < Cr .\]
\end{lemma}
\begin{proof}
Since $\mathcal{F}^T(P, 0) = X_{H_0 + P} \circ \varphi_0$ (where $\varphi_0$ is the trivial embedding \eqref{def:varphi0=(q,0,0)}), for any $P \in \mathcal{P}^{\mathrm{par}, T}_{\sigma, \Lell}$, it follows that there exists a constant $C$, independent of $T$, such that
\[ \sup \left\{ |\mathcal{F}^T(P, 0)|_{\boldsymbol{\sigma, \ell, \dec}}^T \mid P \in \mathcal{P}^{\mathrm{par}, T}_{\sigma,\Lell}, \, |P|_{\sigma, \Lell}^T < r \right\} < Cr .\]
\end{proof}

The differential of  $\mathcal{F}^T$ with respect to $\varphi = (u, v, w_x, w_y)$ evaluated at $(0,0)$ is equal to 
\begin{equation}\label{proof:Csigma_def_par_DF_1}
\Function{D_{\varphi}\mathcal{F}^T(0,0)}{\mathcal{E}^{\mathrm{par},T}_{\sigma,\mathcal{K}}}{\mathcal{R}^T_{\sigma, \mathcal{L}}}{(\hat u, \hat v, \hat w_x, \hat w_y) }{\begin{pmatrix} \overline{M}_0 \cdot \hat v + (m_1)_0 \cdot \hat w_x + (m_2)_0 \cdot \hat w_y - \Lo \hat u\\
     -\Lo \hat v\\
     (\overline{m}_2)_0\cdot \hat v - \Lo \hat w_x\\
     -(\overline{m}_1)_0\cdot \hat v - \Lo \hat w_y - \Lambda \cdot \hat w_x
    \end{pmatrix}},
\end{equation}
where $\mathcal{R}^T_{\sigma, \mathcal{L}}$ is given by \eqref{def:F_codomain} and we use the notations introduced in~\eqref{def:f0} (for the subscript $0$), ~\eqref{eq:barMmL_par} and Lemma~\ref{lemma:def_bar_M_m_L_par}.%

We recall that $\Upsilon$ is the positive constant defined by~\eqref{def:const_Upsilon_par}, and that $C(\cdot)$ denotes a generic positive constant depending on the parameter in brackets.
\begin{lemma}\label{lemma:DF_invertible_par}
    The linear operator $D_{\varphi} \mathcal{F}^T(0,0)$ given by ~\eqref{proof:Csigma_def_par_DF_1} is invertible. Moreover, there exits a constant constant $\bar C$ depending on $\sigma$, $\ell$, $\dec$, $\Upsilon$, and $\Omega_i$ for $i=1,\dots m$, such that
    \[ \| D_{\varphi} \mathcal{F}^T(0,0)^{-1}\| < \bar C(\sigma, \ell, \Omega, \Upsilon),\]
    where $\| \cdot \|$ denotes the operator norm.
\end{lemma}
\begin{proof}
    Given $g = (g_1, g_2, g_3, g_4) \in  \mathcal{S}^{T}_{(\sigma,0), \ell}(\R^n) \times \mathcal{S}^{T}_{(\sigma,0), \ell+\dec+4}(\R^n) \times \mathcal{S}^{T}_{(\sigma,0), \ell+\dec+3}(\R^m) \times \mathcal{S}^{T}_{(\sigma,0), \ell+\dec+2}(\R^m)$ the proof reduces to finding the unique solution %
    $(\hat u, \hat v, \hat w_x, \hat w_y) \in \mathcal{E}^{\mathrm{par},T}_{\sigma,\mathcal{K}}$ of $D_{\varphi}\mathcal{F}^T(0,0)(\hat u, \hat v, \hat w_x, \hat w_y) = (g_1, g_2, g_3, g_4)$, which corresponds to solving the system 
    \begin{equation}
    \begin{cases}\label{proof:lemma_Csigma_invDF_eq_par}
        &\overline{M}_0 \cdot \hat v + (m_1)_0 \cdot \hat w_x + (m_2)_0 \cdot \hat w_y - \Lo \hat u = g_1,\\ 
    &-\Lo \hat v=g_2,\\ 
    &  -(\overline m_1, \overline m_2)_0 \cdot  \hat v +\mathcal{L}^{\mathrm{par}}_{\omega, \Lambda}(\hat w_x, \hat w_y) = (g_3, g_4),
    \end{cases}
    \end{equation}
    where we rewrote the third and fourth lines in the formula for $D_{\varphi}\mathcal{F}^T(0,0)$ given by \eqref{proof:Csigma_def_par_DF_1} using the operator $\mathcal{L}^{\mathrm{par}}_{\omega, \Lambda}$ in \eqref{eq:coh_eqs_par}.
    
    Thanks to Proposition \ref{prop:Csigma_HEomega},  a unique solution $\hat v \in \mathcal{U}^{ T}_{\sigma, \omega, \ell +\dec +3}$ to the second equation of~\eqref{proof:lemma_Csigma_invDF_eq_par} exists and
    \begin{equation}\label{proof:lemma_Inv_par_1_est_v}
        |\hat v|^T_{\sigma, \omega, \ell+\dec+3} \le C(\ell, \dec) |g_2|^T_{\sigma, \ell +\dec+4},
    \end{equation}
    where we refer to~\eqref{def:U} and~\eqref{def:norm_U} for the definition of the space $\mathcal{U}^{ T}_{\sigma, \omega, \ell +\dec +3}$ and the norm $|\cdot|^T_{\sigma, \omega, \ell}$, respectively. 
    In the third line of~\eqref{proof:lemma_Csigma_invDF_eq_par}, $\hat v$ is known. Hence, by Proposition \ref{prop:L_w_par} there exists a unique solution $(\hat w_x, \hat w_y) \in \mathcal{U}^{ T}_{\sigma, \omega, \ell +\dec +2} \times \mathcal{U}^{ T}_{\sigma, \omega, \ell +\dec +1}$ to the third line of of~\eqref{proof:lemma_Csigma_invDF_eq_par}. Moreover, recalling the bounds for $\hat v$ and $(\overline m_1)_0,  (\overline m_2)_0$ in  \eqref{proof:lemma_Inv_par_1_est_v} and Lemma \ref{lemma:def_bar_M_m_L_par}, by Proposition \ref{prop:L_w_par} and the norm properties in Proposition \ref{prop:Csigma_prop_norms_pol}, it follows that
    \begin{equation}
    \label{proof:lemma_Inv_par_1_est_w_xy}
    |\hat w_x|^T_{\sigma, \omega, \ell+\dec+2}, |\hat w_y|^T_{\sigma, \omega, \ell+\dec+1}  \leq C(\sigma, \ell,\dec, \Lambda, \Upsilon) \left(|g_4|^T_{\sigma, \ell +\dec+2} + |g_3|^T_{\sigma, \ell +\dec+3} + |g_2|^T_{\sigma, \ell + \dec+4}\right).
    \end{equation}

    Finally, we can solve the first equation of~\eqref{proof:lemma_Csigma_invDF_eq_par} where $\hat v$, $\hat w_x$ and $\hat w_y$ are known. Thanks to Proposition \ref{prop:Csigma_HEomega} a unique solution $\hat u \in \mathcal{U}^{ T}_{\sigma, \omega, \ell -1}$ of the first equation of~\eqref{proof:lemma_Csigma_invDF_eq_par} exists and satisfies 
    \begin{align*}
        |\hat u|^T_{\sigma, \omega, \ell-1} &\le C(\ell,\dec) |g_1  - \overline{M}_0 \cdot \hat v - (m_1)_0 \cdot \hat w_x - (m_2)_0 \cdot \hat w_y|^T_{\sigma, \ell}  \\
        &\le C(\sigma, \ell, \dec, \Upsilon) \left(|g_1|^T_{\sigma, \ell}  + |\hat v|^T_{\sigma, \ell +\dec+3} + |\hat w_x|^T_{\sigma, \ell +\dec+2} + |\hat w_y|^T_{\sigma, \ell +\dec+1}\right)\\
        &\le C(\sigma, \ell, \dec,\Lambda, \Upsilon) \left(|g_1|^T_{\sigma, \ell}  + |g_2|^T_{\sigma, \ell +\dec+4} + |g_3|^T_{\sigma, \ell +\dec+3} + |g_4|^T_{\sigma, \ell +\dec+2}\right),
    \end{align*}
    where we used Proposition \ref{prop:Csigma_prop_norms_pol}, Lemma \ref{lemma:def_bar_M_m_L_par}, Definition~\eqref{def:norm_U},~\eqref{proof:lemma_Inv_par_1_est_v},%
    and~\eqref{proof:lemma_Inv_par_1_est_w_xy}. This concludes the proof of this lemma. 
    \end{proof}

\begin{lemma}
    \label{lem:DF_limit_par}
    Let $P = (a, b, c_1, c_2, d_1, d_2, d_3) \in  \mathcal{P}^{\mathrm{par},T}_{\sigma,\mathcal{L}}$ and $\varphi_* = (u_*,v_*,w_{x*}, w_{y*}) \in \mathcal{E}^{\mathrm{par},T}_{\sigma, \mathcal{K}}$ with $| \varphi_*|_{\sigma, \mathcal{K}}^T < T^{\ell-1}$. Then there exists a constant $C$, depending continuously on $\Upsilon, |P|^T_{\sigma, \mathcal{L}},$ and $| \varphi_*|_{\sigma, \mathcal{K}}^T$, such that the operator
    \[  D_\varphi \mathcal{F}^T(P, \varphi_*) - D_\varphi \mathcal{F}^T(0, 0):   \mathcal{E}^{\mathrm{par},T}_{\sigma, \mathcal{L}} \to \mathcal{R}^T_{\sigma, \mathcal{L}}, \]
   where $\mathcal{R}^T_{\sigma, \mathcal{L}}$ is given by \eqref{def:F_codomain}, satisfies
\[\| D\mathcal{F}^T(0, 0) -  D\mathcal{F}^T(P, \varphi_*) \| \leq \frac{C}{T^{\ell - 1}},\]
where $\| \cdot\|$ denotes the operator norm. 

In particular, Item \eqref{prop:DF_limit_par} of Proposition \ref{prop:F_properties_par} holds.
\end{lemma}
\begin{proof}
Let $\varphi \in \mathcal{E}^{\mathrm{par},T}_{\sigma,\mathcal{K}}$ and we observe that 
\begin{equation*}
    \left(D_\varphi\mathcal{F}^T(0,0) - D_\varphi\mathcal{F}^T(P, \varphi_*)\right)\varphi = \begin{pmatrix}
        I_1 & I_2 & I_3 & I_4
    \end{pmatrix}^\top,
\end{equation*}
with
\begin{align*}
    I_1 &= -\partial_q b \circ \varphi_{*, q} u - \partial_q \overline{M} \circ \varphi_* \cdot (v_*, u) - \partial_q m_1 \circ \varphi_{*,qxy} \cdot (w_{x*}, u) - \partial_q m_2 \circ \varphi_{*,qy} \cdot (w_{y*}, u) \\
    &-\left(\overline{M}\circ \varphi_* - \overline{M}_0\right)v - \partial_p \overline{M}\circ \varphi_*\cdot (v_*, v) - \partial_x M\circ \varphi_* \cdot (v_*, w_x) - \left(m_1 \circ \varphi_{*, qxy} - (m_1)_0\right)w_x \\
    &- \partial_x m_1 \circ \varphi_{*, qxy} \cdot (w_{x*}, w_x) - \partial_y M \circ \varphi_* \cdot (v_*. w_y) - \left(m_2 \circ \varphi_{*,qy} - (m_2)_0\right) w_y\\
    &- \partial_y m_1 \circ \varphi_{*,qxy}\cdot (w_{x*}, w_y) - \partial_y m_2 \circ \varphi_{*, qy}\cdot (w_{*y}, w_y),
    \end{align*}
    \begin{align*}
    I_2 &= \partial_q^2 a \circ \varphi_{*, q} u + \partial_q^2 b\circ \varphi_{*, q} \cdot (v_*, u) + \partial_q^2 c_1 \circ \varphi_{*, q}\cdot (w_{x*}, u) + \partial_q^2 c_2 \circ \varphi_{*, q}\cdot (w_{y*}, u)\\
    &+ {1 \over 2}\partial_q^2 d_1 \circ \varphi_{*, q} \cdot (w_{x*}, w_{x*}, u) + \partial_q^2 d_2 \circ \varphi_{*, q} \cdot (w_{x*}, w_{y*}, u) + {1 \over 2}\partial_q^2 d_3 \circ \varphi_{*, q} \cdot (w_{y*}, w_{y*}, u)\\
    &+\partial_q^2 M \circ \varphi_*\cdot(v_*, v_*, u) + \partial_q^2 m_1 \circ \varphi_{*, qxy} \cdot (v_*, w_{x*}, u) + \partial^2_q m_2 \circ \varphi_{*, qy} \cdot (v_*, w_{y*}, u)\\
    &+ \partial_q^2 L_1 \circ \varphi_{*, qxy}\cdot (w_{x*}, w_{x*}, w_{x*}, u) +  \partial_q^2 L_2 \circ \varphi_{*, qy}\cdot (w_{x*}, w_{x*}, w_{y*}, u)\\
    &+ \partial_q^2 L_3 \circ \varphi_{*, qy}\cdot (w_{x*}, w_{y*}, w_{y*}, u) +  \partial_q^2 L_4 \circ \varphi_{*, qy}\cdot (w_{y*}, w_{y*}, w_{y*}, u)\\
    &+ \partial_q b \circ \varphi_{*, q} v + 2\partial_q M \circ \varphi_* \cdot (v_*, v) + \partial_p \partial_q M \circ \varphi_* \cdot (v_*, v_*, v) + \partial_q m_1 \circ \varphi_{*, qxy}\cdot (w_{x*}, v) \\
    &+ \partial_q m_2 \circ \varphi_{*, qy}\cdot (w_{y*}, v) + \partial_q c_1 \circ \varphi_{*, q} w_x + \partial_q d_1 \circ \varphi_{*, q} \cdot (w_{x*}, w_x) + \partial_q d_2 \circ \varphi_{*, q} \cdot (w_{y*}, w_x)   \\
    &+ \partial_x \partial_q M \circ \varphi_* \cdot (v_*, v_*, w_x) + \partial_x \partial_q m_1 \circ \varphi_{*, qxy}\cdot (v_*, w_{x*}, w_x) + \partial_q m_1 \circ \varphi_{*, qxy}\cdot (v_*, w_x)\\
    &+ 3\partial_q L_1\circ \varphi_{*, qxy}\cdot (w_{x*}, w_{x*}, w_x) + \partial_x \partial_q L_1\circ \varphi_{*, qxy} \cdot (w_{x*}, w_{x*}, w_{x*}, w_x)\\
    &+2 \partial_q L_2 \circ \varphi_{*, qy}\cdot (w_{x*}, w_{y*}, w_x) + \partial_q L_3 \circ \varphi_{*, qy}\cdot (w_{y*}, w_{y*}, w_{x}) + \partial_q c_2 \circ \varphi_{*q} w_y\\
    &+ \partial_q d_2 \circ \varphi_{*, q} \cdot (w_{x*}, w_y) + \partial_q d_3 \circ \varphi_{*, q} \cdot (w_{y*}, w_y) + \partial_y \partial_q M \circ \varphi_* \cdot (v_*, v_*, w_y) \\
    &+ \partial_y \partial_q m_1 \circ \varphi_{*, qxy}\cdot (v_*, w_{x*}, w_y) + \partial_q m_2 \circ \varphi_{*, qy}\cdot (v_*, w_y) + \partial_y \partial_q m_2 \circ \varphi_{*, qy} \cdot (v_*, w_{y*}, w_y)\\
    &+ \partial_y\partial_q L_1\circ \varphi_{*, qxy}\cdot (w_{x*}, w_{x*}, w_{x*}, w_y) +  \partial_y\partial_q L_2\circ \varphi_{*, qy}\cdot (w_{x*}, w_{x*}, w_{y*}, w_y)\\
    &+\partial_q L_2\circ \varphi_{*, qy}\cdot (w_{x*}, w_{x*}, w_y) + \partial_y\partial_q L_3\circ \varphi_{*, qy}\cdot (w_{x*}, w_{y*}, w_{y*}, w_y)\\
    &+2\partial_q L_3\circ \varphi_{*, qy}\cdot (w_{x*}, w_{y*}, w_y) + \partial_y\partial_q L_4\circ \varphi_{*, qy}\cdot (w_{y*}, w_{y*}, w_{y*}, w_y)\\
    &+ 3\partial_q L_4\circ \varphi_{*, qy}\cdot (w_{y*}, w_{y*}, w_y),
    \end{align*}
    \begin{align*}
    I_3 &= - \partial_q c_2\circ\varphi_{*, q} u - \partial_q d_2 \circ\varphi_{*, q} \cdot (w_{x*}, u) - \partial_q d_3 \circ\varphi_{*, q} \cdot (w_{y*}, u) - \partial_q \partial_y M\circ \varphi_*\cdot (v_*, v_*, u) \\
    &- \partial_q \partial_ym_1\circ\varphi_{*, qxy}\cdot (w_{x*}, v_*, u) - \partial_q \overline{m}_2\circ\varphi_{*, qy}\cdot (v_*, u) - \partial_q \partial_yL_1\circ \varphi_{*, qxy}\cdot (w_{x*}, w_{x*}, w_{x*}, u)\\
    &-\partial_q \overline{L}_2\circ \varphi_{*, qy}\cdot (w_{x*}, w_{x*}, u) -\partial_q \overline{L}_3\circ \varphi_{*, qy}\cdot (w_{x*}, w_{y*}, u) -\partial_q \overline{L}_4\circ \varphi_{*, qy}\cdot (w_{y*}, w_{y*}, u)\\
    &-\partial_p \partial_y M\circ\varphi_* \cdot (v_*, v_*, v) - 2 \partial_y M\circ\varphi_* \cdot (v_*, v) -  \partial_y m_1\circ\varphi_{*, qxy} \cdot (w_{x*}, v)\\
    &- \left(\overline{m}_2\circ \varphi_{*, qy} - \left(\overline{m}_2\right)_0\right)v - d_2 \circ\varphi_{*, q} w_{x} - \partial_x\partial_y M \circ\varphi_* \cdot (v_*, v_*, w_x) \\
    &- \partial_x\partial_y m_1 \circ\varphi_{*, qxy} \cdot (w_{x*}, v_*, w_x) - \partial_y m_1 \circ\varphi_{*, qxy} \cdot (v_*, w_x) \\
    &- \partial_x \partial_y L_1 \circ \varphi_{*, qxy} \cdot (w_{x*}, w_{x*}, w_{x*}, w_x) - 3 \partial_y L_1 \circ \varphi_{*, qxy} \cdot (w_{x*}, w_{x*}, w_x) \\
    &- 2  \overline{L}_2 \circ \varphi_{*, qy} \cdot (w_{x*}, w_x) - \overline{L}_3 \circ \varphi_{*, qy} \cdot (w_{y*}, w_x) - d_3 \circ\varphi_{*, q} w_{y} - \partial^2_y M \circ\varphi_* \cdot (v_*, v_*, w_y)\\
    &- \partial^2_y m_1 \circ\varphi_{*, qxy} \cdot (w_{x*}, v_*, w_y) - \partial_y \overline{m}_2 \circ\varphi_{*, qy} \cdot (v_*, w_y) - \partial^2_y L_1 \circ \varphi_{*, qxy} \cdot (w_{x*}, w_{x*}, w_{x*}, w_y)\\
    &- \partial_y \overline{L}_2 \circ \varphi_{*, qy} \cdot (w_{x*}, w_{x*}, w_y) - \partial_y \overline{L}_3 \circ \varphi_{*, qy} \cdot (w_{x*}, w_{y*}, w_y)  - \overline{L}_3 \circ \varphi_{*, qy} \cdot (w_{x*}, w_y)\\
    &- \partial_y \overline{L}_4 \circ \varphi_{*, qy} \cdot (w_{y*}, w_{y*}, w_y) - 2 \overline{L}_4 \circ \varphi_{*, qy} \cdot (w_{y*}, w_y),
\end{align*}
\begin{align*}
    I_4 &= \partial_q c_1\circ\varphi_{*, q} u + \partial_q d_1 \circ\varphi_{*, q} \cdot (w_{x*}, u) + \partial_q d_2 \circ\varphi_{*, q} \cdot (w_{y*}, u) + \partial_q \partial_x M\circ \varphi_*\cdot (v_*, v_*, u) \\
    &+ \partial_q \overline{m}_1\circ\varphi_{*, qxy}\cdot (v_*, u) +\partial_q \overline{L}_1\circ \varphi_{*, qxy}\cdot (w_{x*}, w_{x*}, u) +2 \partial_q L_2\circ \varphi_{*, qy}\cdot (w_{x*}, w_{y*}, u) \\
    &+\partial_q L_3\circ \varphi_{*, qy}\cdot (w_{y*}, w_{y*}, u)+\partial_p \partial_x M\circ\varphi_* \cdot (v_*, v_*, v) + 2 \partial_x M\circ\varphi_* \cdot (v_*, v) \\
    &+  \left(\overline{m}_1\circ\varphi_{*, qxy}  - \left(\overline{m}_1\right)_0\right)v + d_1 \circ\varphi_{*, q} w_x + \partial^2_x M \circ\varphi_* \cdot (v_*, v_*, w_x) + \partial_x\overline{m}_1 \circ\varphi_{*, qxy} \cdot (v_*, w_x) \\
    &+ \partial_x \overline{L}_1 \circ \varphi_{*, qxy} \cdot (w_{x*}, w_{x*}, w_x) + 2\overline{L}_1 \circ \varphi_{*, qxy} \cdot (w_{x*}, w_x) + 2L_2 \circ \varphi_{*, qy} \cdot (w_{y*}, w_x)\\
    &+ d_2 \circ\varphi_{*, q} w_{y} + \partial_y\partial_x M \circ\varphi_* \cdot (v_*, v_*, w_y) + \partial_y \overline{m}_1 \circ\varphi_{*, qxy} \cdot (v_*, w_y)\\
    &+\partial_y \overline{L}_1 \circ \varphi_{*, qxy} \cdot (w_{x*}, w_{x*}, w_y) + 2  \partial_y L_2 \circ \varphi_{*, qy} \cdot (w_{x*}, w_{y*}, w_y) + 2  L_2 \circ \varphi_{*, qy} \cdot (w_{x*}, w_y)\\
    &+ \partial_y L_3 \circ \varphi_{*, qy} \cdot (w_{y*}, w_{y*}, w_y) + 2\partial_y L_3 \circ \varphi_{*, qy} \cdot (w_{y*}, w_y),
    \end{align*}
where we used the notation introduced in~\eqref{def:f0} and we denoted by a subscript the projection of $\varphi_*$ on the $q, p, x$ and $y$.

Similarly to the proof of Lemma \ref{lem:DF_limit_par}, using Proposition \ref{prop:Csigma_prop_norms_pol} and Lemma \ref{lemma:def_bar_M_m_L_par}, one can prove that   
\begin{equation*}
    I_1 \in \mathcal{S}^{ T}_{(\sigma, 0), 2\ell - 1}, \quad I_2 \in \mathcal{S}^{ T}_{(\sigma, 0), 2\ell +\dec+ 3}, \quad I_3 \in \mathcal{S}^{ T}_{(\sigma, 0), 2\ell +\dec+ 2}, \quad I_4 \in \mathcal{S}^{ T}_{(\sigma, 0), 2\ell +\dec + 1}.
\end{equation*}
Moreover, there exists a positive constant $C$ depending continuously on $|P|^T_{\sigma, \mathcal{L}}$ and $|\varphi_*|^T_{\sigma, \mathcal{K}}$ such that 
\begin{equation*}
    |I_1|^T_{\sigma, 2\ell-1}, \, |I_2|^T_{\sigma, 2\ell+\dec +3}, \,|I_3|^T_{\sigma, 2\ell+\dec +2}, \,|I_4|^T_{\sigma, 2\ell+\dec +1} \le C|\varphi|^T_{\sigma, \mathcal{K}},
\end{equation*}
and hence 
\begin{align*}
    |\left(D_\varphi\mathcal{F}^T(0,0) - D_\varphi\mathcal{F}^T(P, \varphi_*)\right)\varphi|_{\boldsymbol{{\sigma, \ell, \dec}}}^T &\le   |I_1|^T_{\sigma, \ell} +   |I_2|^T_{\sigma, \ell+\dec+4} +   |I_3|^T_{\sigma, \ell+\dec+3} +   |I_4|^T_{\sigma, \ell+\dec+2}\\
    &\le {1 \over T^{\ell-1}}\left(|I_1|^T_{\sigma, 2\ell-1} +   |I_2|^T_{\sigma, 2\ell+\dec+3} +   |I_3|^T_{\sigma, 2\ell+\dec+2} +   |I_4|^T_{\sigma, 2\ell+\dec+1}\right)\\
    &\le {C \over T^{\ell-1}}|\varphi|^T_{\sigma, \mathcal{K}}. 
\end{align*}
\end{proof}

\subsection{Proof of Item \eqref{thm:C_transverse} of Theorem \ref{Thm:par_Csigma}: Existence of asymptotic parabolic transversal dynamic}\label{sec:proof_item_2_par_Holder} 

In this section, we prove that the asymptotic KAM torus obtained in Item \eqref{thm:C_existence} of Theorem \ref{Thm:par_Csigma} is asymptotically parabolic under slightly stronger assumptions on the regularity (more precisely, on $\sigma$) and on the decay rate of the perturbations (see \eqref{eq:Hyp_decay_par_2}).

In the following, we continue using the notation and parameters from the previous subsections, with the only exception that we now assume $\sigma \geq 2$.

Recall that the perturbations considered in Item \eqref{thm:C_transverse} of Theorem \ref{Thm:par_Csigma} can be identified with the space $\mathcal{P}^{\mathrm{par}, 1}_{\sigma, \mathcal{L}^*}$ (see Section \ref{sc:perturbation_par}). Moreover, for any $T \geq 1$, we have $\mathcal{P}^{\mathrm{par}, 1}_{\sigma, \mathcal{L}^*} \subseteq \mathcal{P}^{\mathrm{par}, 1}_{\sigma, \mathcal{L}(\ell, \dec + 2)}$, where $\mathcal{L}(\ell, \dec + 2) \in \left(\R_{\ge 0}\right)^7$ is given by \eqref{eq:index_regularity_perturbations_par}, and
\[  |P|^T_{\sigma, \Lell(\ell, \dec + 2)}\leq |P|^T_{\sigma, \Lell^*} , \qquad \text{ for any }P \in \mathcal{P}^{\mathrm{par}, T}_{\sigma, \Lell^*}.\] 

Thus, denoting $\mathcal{K}^* = \mathcal{K}(\ell, \dec+2)$ where $\mathcal{K}(\ell, \dec+2)$ is defined in~\eqref{eq:index_regularity_uwx_par}, by Proposition \ref{prop:C_existence}, for any $r>0$ there exists $T_0 \ge 1$ such that, for any $T\ge T_0$, there exists a $C^1$-map
\begin{equation*}
    \boldsymbol{\varphi}^T: B_{\mathcal{P}^{\mathrm{par},T}_{\sigma, \mathcal{L}^*}}(r)    \to  \mathcal{E}^{\mathrm{par},T}_{\sigma, \mathcal{K}^*}
\end{equation*}
for which $\boldsymbol{\varphi}^T(P)$ defines a $C^\sigma$ asymptotic KAM torus  associated with $(X_{H_0 + P}, X_{H_0}, \varphi_0)$, for any $P \in \mathcal{P}^{\mathrm{par}, 1}_{\sigma, \mathcal{L}^*}$ with $|P|^T_{\sigma, \mathcal{L}^*} < r$.

 In this section, we will deal with matrices having components exhibiting different time-decay properties. For this reason, we define 
\begin{align}\label{eq:index_regularity_G_par} 
    \mathcal{D}(\ell + \dec)  &:= (\ell + \dec+1 , \ell + \dec+2, \ell + \dec, \ell + \dec +1),\\ \label{eq:index_regularity_G_par_Sym} 
    \mathfrak{D}(\ell + \dec)  &:= (\ell + \dec +1 , \ell + \dec+2, \ell + \dec+2, \ell + \dec+3).
\end{align}
and, for simplicity, we denote  $\mathcal{D}=\mathcal{D}(\ell + \dec)$ and $\mathfrak{D} = \mathfrak{D}(\ell + \dec)$.

In the spirit of the proof of Item \eqref{thm:C_transverse} of Theorem \ref{Thm:par_Csigma_0}, contained in Section \ref{sec:proof_item_2_par_deg_Holder}, we will show the following.

\begin{proposition}
\label{prop:C_transverse}
Fix $r > 0$ and let $T_0 \geq 1$ be given by Proposition \ref{prop:C_existence}. There exist $T_1 \geq T_0$ such that, for any $T \geq T_1$,  there exists a $C^1$-map
\[ \boldsymbol{G}^T: B_{\mathcal{P}^{\mathrm{par}, T}_{\sigma, \Lell^*}}(r)    \subseteq  \mathcal{P}^{\mathrm{par}, T}_{\sigma, \Lell^*} \to \mathcal{S}p^{ T}_{\sigma-1, \omega, \mathcal{D}}   \]
for which $\boldsymbol{S}(\boldsymbol{\varphi}^T(P),  \boldsymbol{G}^T(P))$ given by \eqref{eq:S_formula} is of class $C^1$ and
\begin{equation*}
\label{eq:IFT_condition_par}
\boldsymbol{\zeta}(H_0 + P, \boldsymbol{\varphi}^T(P),  \boldsymbol{G}^T(P)) =  \Mpar, \qquad \text{for any } P \in B_{\mathcal{P}^{\mathrm{par}, T}_{\sigma, \Lell^*}}(r),
\end{equation*}
where $\boldsymbol{\zeta}$ is given by \eqref{eq:formula_zeta} and  $\boldsymbol{\varphi}^T: B_{\mathcal{P}^{\mathrm{par}, T}_{\sigma, \Lell^*}}(r)   \to \mathcal{E}^{\mathrm{par}, T}_{\sigma, \mathcal{K}^*}$  is given by Proposition \ref{prop:C_existence}.

In particular, by Corollary \ref{cor:criterion}, for any $P \in B_{\mathcal{P}^{\mathrm{par}, T}_{\sigma, \Lell^*}}(r)$, the asymptotic KAM torus $\boldsymbol{\varphi}^T(P)$ is asymptotically parabolic.
\end{proposition}

We refer the reader to Section \ref{sc:matrix_nonuniform_decay} for the definition of the space $\big( \mathcal{S}p^{ T}_{\sigma-1, \omega, \mathcal{D}}, |\cdot|^T_{\sigma - 1, \omega, \mathcal{D}} \big)$.

Notice that Item \eqref{thm:C_transverse} of Theorem \ref{Thm:par_Csigma} follows directly from the proposition above, since
\begin{equation*}
    \big|P\big|_{\sigma, \Lell^*}^{T} \leq |P|_{\sigma, \Lell^*}^{1}, \qquad \text{ for any } T \geq 1\,  \text{ and any } \, P \in \mathcal{P}^{\mathrm{par}, 1}_{\sigma, \Lell^*}, 
\end{equation*}
and, by definition of $\mathcal{K}^*$, the asymptotic KAM torus has the desired decay rates.

In the following, we fix $r > 0$ and let $T_0 \geq 0$ be given by Proposition \ref{prop:C_existence}. Moreover, for any $T \geq T_0$, we will denote by $ \boldsymbol{\varphi}^T: B_{\mathcal{P}^{\mathrm{par},T}_{\sigma, \mathcal{L}^*}}(r)   \to  \mathcal{E}^{\mathrm{par},T}_{\sigma, \mathcal{K}^*}$ the map given by Proposition \ref{prop:C_existence}. 

To prove Proposition \ref{prop:C_transverse}, we consider the functional
\begin{equation*}
    \mathcal{G}^T: B_{\mathcal{P}^{\mathrm{par}, T}_{\sigma,\Lell^*}}(r) \times \mathcal{S}p^{ T}_{\sigma-1, \omega, \mathcal{D}}   \to  \mathcal{S}ym^{ T}_{\sigma-1, \mathfrak{D}}
\end{equation*}
as
\begin{equation}
\label{eq:formula_G_par_2}
    \begin{aligned}
    \mathcal{G}^T(P, G) =&  B^\top \partial_p^2 H\circ \varphi B + K^{\top} \partial_p\partial_zH\circ \varphi B + B^{\top} \left(\partial_p\partial_zH\circ \varphi\right)^\top K \\
    & + K^{\top} \partial_z^2 H\circ \varphi K + K^{\top} J_m \Lo K - \Mpar,
    \end{aligned}
\end{equation}
where $z=(x,y) \in \R^m \times \R^m$, $H = H_0 + P$, $\varphi = \boldsymbol{\varphi}^T(P) = (u, v, w_x, w_y)$, $K = \exp(G)$, $\Mpar$ is defined by~\eqref{Ms}, $B = -(\mathrm{Id}_{n} +\partial_q  u )^{-\top} \left(\left(\partial_q w_x\right)^\top, \left(\partial_q w_y\right)^\top\right) J_m K$ and $\partial_p\partial_zH\circ \varphi$ stands for the $2m \times n$ matrix having coomponents $\partial_{p_i}\partial_{z_j}H\circ \varphi$ for $1 \le i \le n$ and $1 \le j \le 2m$. We refer the reader to Section \ref{sc:matrix_nonuniform_decay} for the definition of the space $\big(\mathcal{S}ym^{ T}_{\sigma-1, \mathfrak{D}}, |\cdot|^T_{\sigma - 1, \mathfrak{D}}  \big)$ taking values in the set of $2m \times 2m$ symmetric matrices. 

We point out that in the definition of the matrix $B$, the element $\left(\left(\partial_q w_x\right)^\top, \left(\partial_q w_y\right)^\top\right)$ stands for the $n \times 2m$ matrix defined by blocks by the $n \times m$ matrices given by $\left(\partial_q w_x\right)^\top$ and $\left(\partial_q w_y\right)^\top$. In the definition of the functional $\mathcal{G}^T$ in~\eqref{eq:formula_G_par_2}, for the sake of notational simplicity, we use the compact notation $z=(x,y)$ for the transverse variables $(x,y)$. %

Notice that this functional is defined by subtracting $\Mpar$ in the functional \eqref{eq:formula_F_par.deg_2} introduced in Section \ref{sec:proof_item_2_par_deg_Holder} and by considering different domain and codomain. Moreover, by the same arguments provided in Section \ref{sec:proof_item_2_par_deg_Holder}, it follows that 
\begin{equation}
\label{eq:parabolic_characterization}
    \text{ If }  \quad \mathcal{G}^T(P, \varphi) = 0 \quad \text{ then } \quad  \boldsymbol{\varphi}^T(P) \text{ is an asymptotically parabolic KAM torus.}
\end{equation}

The following proposition contains a series of properties satisfied by the functional $\mathcal{G}^T$. For the rest of this section, we denote by $D_G \mathcal{G}^T$ the differential of $\mathcal{G}^T$ with respect to the matrix $G$.

\begin{proposition}
\label{prop:G_properties_par}
    The operator $\mathcal{G}^T$ given by \eqref{eq:formula_G_par_2} satisfies the following.
    \begin{enumerate}
    \item \label{prop:G_well_defined_par} $\mathcal{G}^T$ and $D_G\mathcal{G}^T$ are well-defined continuous maps, for any $T \geq T_0$. 
    
    \item \label{prop:G_bounded_par} There exists a constant $C$ such that, for any $T \geq T_0$,
\[ \sup \left\{ |\mathcal{G}^T(P, 0)|^T_{\sigma - 1, \mathfrak{D}} \,\left|\, P \in B_{\mathcal{P}^{\mathrm{par}, T}_{\sigma,\Lell^*}}(r) \right\}\right. \leq Cr.\] 

    \item \label{prop:DG_invertible_par} $D_G \mathcal{G}^T(0, 0)$ is invertible, for any $T \geq T_0$. Moreover, there exists a constant $\bar C$, independent of $T$, such that $\|D_G \mathcal{G}^T(0, 0)^{-1}\| < \bar C$. 
        \item \label{prop:DG_limit_par} Let $r, \rho > 0$. Then  
        \[\lim_{T \to +\infty} \| D_G \mathcal{G}^T(P, G) - D_G\mathcal{G}^T(0, 0) \| = 0, \qquad \text{uniformly on } \quad B_{\mathcal{P}^{\mathrm{par}, T}_{\sigma,\Lell^*}}(r) \times B_{\mathcal{S}p^{ T}_{\sigma-1, \omega, \mathcal{D}}}(\rho),\]
where $\| \cdot\|$ denotes the operator norm  and $B_{\mathcal{S}p^{ T}_{\sigma-1, \omega, \mathcal{D}}}(\rho) \subset \mathcal{S}p^{ T}_{\sigma-1, \omega, \mathcal{D}}$ stands for a ball of radius $\rho$ centered at the origin.
    \end{enumerate}
\end{proposition}

Similarly to the proof of Proposition \ref{prop:Cdeg_transverse}, one can prove Proposition \ref{prop:C_transverse} as a direct consequence of Proposition \ref{prop:G_properties_par} and the quantitative version of the Implicit Function Theorem given by Theorem \ref{thm:QIFT}. The rest of this section is dedicated to the proof of Proposition \ref{prop:G_properties_par} %
which is divided into the following four lemmas (Lemmas \ref{lemma:G_well_defined_par}, \ref{lemma:G_bounded_par}, \ref{lemma:DG_invertible_par} and \ref{lem:DG_limit_par}), each devoted to proving one of the properties stated in the proposition.

\begin{lemma}\label{lemma:G_well_defined_par}
       Let $T \geq T_0$. The functional $\mathcal{G}^T$ given by \eqref{eq:formula_G_par_2} and its partial derivative $D_G\mathcal{G}^T$ are well-defined continuous maps.
\end{lemma}
\begin{proof}
    First, we verify that $\mathcal{G}^T$ takes values in $\mathcal{S}ym^{ T}_{\sigma-1, \mathfrak{D}}$. Let us verify that $\mathcal{G}^T$ takes values in $\mathcal{S}ym^{ T}_{\sigma-1, \ell+\dec+1}.$ First, we check that, for any $P, G$, the matrix $\mathcal{G}^T(P, G)(q, t)$ given by \eqref{eq:formula_G_par_2} is symmetric. In the following, if there is no risk of confusion, we omit writing the variables $P, G, q, t$. By \eqref{eq:formula_G_par_2}, it suffices to check that $K^{\top} J_m \Lo K$ is symmetric.
  
By taking derivatives with respect to $q$ and $t$ in $K(q, t) K(q, t)^{-1} = \mathrm{Id}_{2m}$, it easily follows that 
\[ (\Lo K)  K^{-1} = -K (\Lo K^{-1}).\]
Hence, recalling that $K^{\top} J_m = J_m K^{-1}$, we have
\begin{align*}
 (K^{\top} J_m \Lo K)^{\top} & = - (\Lo K^{\top})J_m K = - J_m(\Lo K^{-1}) K = J_mK^{-1} \Lo K = K^{\top} J_m \Lo K,
\end{align*} 
and thus $\mathcal{G}^T(P, G)(q, t)$ is symmetric. 
    We need to show that $\mathcal{G}^T(P, G) \in\mathcal{M}^{ T}_{\sigma-1, \mathfrak{D}}$ (see~\eqref{def:BS_par_M_Sym} for the definition of $\mathcal{M}^{ T}_{\sigma-1, \mathfrak{D}}$). We observe that 
    \begin{equation}
\begin{aligned}\label{eq:der_order_2_Ham_par}
    \partial_p^2 H & = \overline{\overline{M}},\\
    \partial_p\partial_x H &= \partial_x \overline{M}\cdot p + \overline{m}_1,\\
    \partial_p\partial_y H &= \partial_y \overline{M}\cdot p + \partial_y m_1 \cdot x +  \overline{m}_2,\\
    \partial_x^2 H &= \Lambda + d_1 + \partial_x^2 M \cdot p^2 + \partial_x \overline{m}_1 \cdot p + \overline{\overline{L}}_1\cdot x + 2 L_2 \cdot y,\\
    \partial_y\partial_x H &= d_2 + \partial_y\partial_x M \cdot p^2 + \partial_y \overline{m}_1 \cdot p + \partial_y\overline{L}_1\cdot x^2 + 2 \overline{L}_2 \cdot x + \overline{L}_3\cdot y,\\
    \partial_y^2 H &= d_3 + \partial_y^2 M \cdot p^2 + \partial^2_y m_1 \cdot (x,p) + \partial_y \overline{m}_2
\cdot p + \partial_y^2 L_1 \cdot x^3 + \partial_y \overline{L}_2 \cdot x^2\\
&+\overline{\overline{L}}_3\cdot x +  \overline{\overline{L}}_4 \cdot y,
\end{aligned}
\end{equation}
where we recall that $\Lambda =\mathrm{diag}\left(\Lambda_1, \dots, \Lambda_m\right)$ and $\overline{M}$, $\overline{\overline{M}}$, $\overline{m}_1$, $\overline{m}_2$, $\overline{L}_1$, $\overline{L}_2$, $\overline{L}_3$, $\overline{L}_4$, $\overline{\overline{L}}_1$, $\overline{\overline{L}}_3$ and $\overline{\overline{L}}_4$ are defined in Lemma \ref{lemma:def_bar_M_m_L_par}. 

We point out that, under the hypotheses~\eqref{eq:Hyp_decay_par_2}, Proposition \ref{prop:C_existence} provide the existence of a $C^\sigma$ asymptotic KAM torus $\varphi = \boldsymbol{\varphi}^T(P) = (u, v, w_x, w_y)$ satisfying
\begin{equation}\label{uvwxwy_decay_par_trasv}
        u \in \mathcal{S}_{(\sigma, 0), \ell-1}^{T}, \quad  v \in \mathcal{S}_{(\sigma, 0), \ell+ \dec + 5}^{T}, \quad w_x \in \mathcal{S}_{(\sigma, 0), \ell+ \dec + 4}^{T} , \quad w_y \in \mathcal{S}_{(\sigma, 0), \ell+ \dec + 3}^{T}.
    \end{equation}
where we refer to~\eqref{def:S} for the definition of the space $\mathcal{S}^{ T}_{(\sigma, 0), \ell}$.

We recall that $B=(\mathrm{Id}_{2n} +\partial_q  u )^{-\top} \left(\left(\partial_q w_x\right)^\top, \left(\partial_q w_y\right)^\top\right) J_m K$ is a $n \times 2m$ matrix. We rewrite $B$ using the block notation as $B=\left(B_1, B_2\right)$ where $B_1$ and $B_2$ are $n \times m$ matrices. Using~\eqref{uvwxwy_decay_par_trasv}, a straightforward computation shows that 
\begin{equation}\label{B1B2_decay}
    B_1 \in \mathcal{M}^{ T}_{\sigma-1, \ell + \dec +3}, \quad B_2 \in \mathcal{M}^{ T}_{\sigma-1, \ell + \dec +4}, 
\end{equation}
we refer to~\eqref{def:BS_ell_M_Sym} for the definition of the space $\mathcal{M}^{ T}_{\sigma-1, \ell + \dec +2}$.
 Moreover, recalling that $K = \exp(G)$ for each $G \in \mathcal{S}p^{ T}_{\sigma-1, \omega, \mathcal{D}} $, we have that 
 \begin{equation}\label{K_decay_par}
     K - \mathrm{Id}_{2m} - G = G^2 \int_0^1(1 - s)\exp(sG)ds \in \mathcal{M}_{\sigma - 1, \mathcal{D}(2\ell + 2\dec)}^{T},
\end{equation}
where we refer to~\eqref{eq:index_regularity_G_par} for the definition of $\mathcal{D}(2\ell+2\dec)$ and to~\eqref{def:BS_par_M_Sym} for the one of $\mathcal{M}_{\sigma - 1, \mathcal{D}(2\ell + 2\dec)}^{T}$.

In what follows, we shall use the block decomposition
\begin{equation*}
    K=\begin{pmatrix} K_1 & K_2 \\ K_3 & K_4 \end{pmatrix}
\end{equation*}
introduced by~\eqref{def:Matrix_block}. Furthermore, for any $G \in \mathcal{S}p^{ T}_{\sigma-1, \omega, \mathcal{D}}$, we can write 
\begin{equation}\label{G_decay_par}
    G=\begin{pmatrix} G_1 & G_2 \\ G_3 & -G_1^\top \end{pmatrix} \qquad   \begin{array}{ll}  G_1 : \T^n \times I_T \to \mathcal{M}_m(\R), \quad G_2, G_3 : \T^n \times I_T \to  \textup{Sym}(m, \R), \\ G_1 \in \mathcal{M}^{T}_{\sigma-1, \ell+\dec+1}, \,G_2 \in \mathcal{M}^{T}_{\sigma-1, \ell+\dec+2}, \, G_3 \in \mathcal{M}^{T}_{\sigma-1, \ell+\dec} \end{array} 
\end{equation}

We want to verify that each element in the RHS of~\eqref{eq:formula_G_par_2} belongs to $\mathcal{M}^{ T}_{\sigma-1, \mathfrak{D}}$. To this end, using the notation introduced above, we observe that  
\begin{align*}
    B^\top \partial_p^2 H\circ \varphi B &= \begin{pmatrix}  B_1^\top \partial_p^2 H\circ \varphi B_1 & B_1^\top \partial_p^2 H\circ \varphi B_2 \\[0.2cm] B_2^\top \partial_p^2 H\circ \varphi B_1 & B_2^\top \partial_p^2 H\circ \varphi B_2\end{pmatrix},\\
    K^{\top} \partial_p\partial_zH\circ \varphi B &= \begin{pmatrix} K_1^\top \partial_p \partial_x H \circ \varphi B_1 + K_3^\top \partial_p \partial_y H \circ \varphi B_1 & K_1^\top \partial_p \partial_x H \circ \varphi B_2 + K_3^\top \partial_p \partial_y H \circ \varphi B_2 \\[0.2cm] K_2^\top \partial_p \partial_x H \circ \varphi B_1 + K_4^\top \partial_p \partial_y H \circ \varphi B_1 & K_2^\top \partial_p \partial_x H \circ \varphi B_2 + K_4^\top \partial_p \partial_y H \circ \varphi B_2\end{pmatrix}.
\end{align*}
Hence, using~\eqref{eq:der_order_2_Ham_par},~\eqref{B1B2_decay} and~\eqref{K_decay_par}, we prove that $B^\top \partial_p^2 H\circ \varphi B, \, K^{\top} \partial_p\partial_zH\circ \varphi B \in \mathcal{M}^{ T}_{\sigma-1, \mathfrak{D}}$. Similarly, one can see that $B^{\top} \left(\partial_p\partial_zH\circ \varphi\right)^\top K \in \mathcal{M}^{ T}_{\sigma-1, \mathfrak{D}}$. It remains to verify that $K^{\top} \partial_z^2 H\circ \varphi K + K^{\top} J_m \Lo K - \Mpar \in \mathcal{M}^{ T}_{\sigma-1, \mathfrak{D}}$. To this end, we observe that 
\begin{align*}
    &K^\top \partial_z^2 H\circ \varphi K   + K^\top J_m \Lo K - \Mpar \\
    &= K^\top \left(\partial_z^2 H\circ \varphi - \Mpar\right)K \\
    &+K^\top \left[(K - \mathrm{Id}_{2m})^\top \Mpar + \Mpar(K - \mathrm{Id}_{2m}) + J_m \Lo (K - \mathrm{Id}_{2m})\right]\\
    &-(K - \mathrm{Id}_{2m})^\top(K - \mathrm{Id}_{2m})^\top \Mpar.
\end{align*}
On the one hand, using~\eqref{eq:der_order_2_Ham_par},~\eqref{uvwxwy_decay_par_trasv}, and noticing that $K - \mathrm{Id}_{2m} \in \mathcal{M}^{ T}_{\sigma-1, \mathcal{D}}$, one can verify that $ K^\top \left(\partial_z^2 H\circ \varphi - \Mpar\right)K, \, (K - \mathrm{Id}_{2m})^\top(K - \mathrm{Id}_{2m})^\top \Mpar \in \mathcal{M}^{ T}_{\sigma-1, \mathfrak{D}}$. On the other hand, thanks to~\eqref{G_decay_par} and remembering the definition~\eqref{def:BS_par_Sp}, one can prove that 
\begin{equation*}
    G^\top \Mpar + \Mpar G + J_m \Lo G = \begin{pmatrix}  G_1^\top \Lambda + \Lambda G_1 + \Lo G_3 & \Lambda G_2 - \Lo G_1^\top \\ G_2 \Lambda - \Lo G_1 & -\Lo G_2 \end{pmatrix} \in \mathcal{M}^{ T}_{\sigma-1, \mathfrak{D}}.
\end{equation*}
Combining the latter with~\eqref{K_decay_par}, we prove that $K^\top \Big[(K - \mathrm{Id}_{2m})^\top \Mpar + \Mpar(K - \mathrm{Id}_{2m}) + J_m \Lo (K - \mathrm{Id}_{2m})\Big] \in \mathcal{M}^{ T}_{\sigma-1, \mathfrak{D}}$. This concludes the proof that $\mathcal{G}^T(P, G) \in\mathcal{M}^{ T}_{\sigma-1, \mathfrak{D}}$.
The proof that $\mathcal{G}^T$ and $D_G\mathcal{G}^T$ are continuous maps is a straightforward consequence of~\eqref{eq:formula_G_par_2} together with the regularity assumptions on the perturbative terms  $P=(a,b,c_1, c_2,d_1, d_2, d_3) \in \mathcal{P}^{\mathrm{par}, T}_{\sigma,\mathcal{L}^*}$, the components of the $C^\sigma$ asymptotic KAM torus $\varphi =  \boldsymbol{\varphi}^T(P) = (u,v,w_x, w_y) \in \mathcal{E}^{\mathrm{par}, T}_{\sigma, \mathcal{K}^*}$ and the matrix $G \in \mathcal{S}p^{ T}_{\sigma-1, \omega, \mathcal{D}}$.
\end{proof}

\begin{lemma}\label{lemma:G_bounded_par}
There exists a constant $C$, such that for any $T \geq T_0$,
\[ \sup \left\{ |\mathcal{G}^T(P, 0)|^T_{\sigma - 1, \mathfrak{D}} \,\left|\, P \in B_{\mathcal{P}^{\mathrm{par}, T}_{\sigma,\Lell^*}}(r) \right\}\right. \leq Cr.\]
\end{lemma}
\begin{proof}
The proof is similar to that of Lemma \ref{lemma:F_bounded_par}. It consists of a straightforward computation. For this reason, it is omitted. 
\end{proof}

The differential of $\mathcal{G}^T$ with respect to $G$ evaluated at $(0,0)$ is equal to 
\begin{equation}\label{eq:diff_form_G_par_2}
\Function{D_{G}\mathcal{G}(0,0)}{\mathcal{S}p^{ T}_{\sigma-1, \omega, \mathcal{D}}}{ \mathcal{S}ym^{ T}_{\sigma-1, \mathfrak{D}}}{G}{ \mathfrak{L}^{\mathrm{par}}_{\omega, \Lambda} (G) = G^{\top}\Mpar + \Mpar G + J_m \Lo G},
\end{equation}
We recall that, throughout this work, $C(\cdot)$ stands for a generic positive constant depending on the parameter in brackets. 

\begin{lemma}\label{lemma:DG_invertible_par}
     Let $T \geq T_0$. The linear operator $D_G \mathcal{G}^T(0,0)$ given by \eqref{eq:diff_form_G_par_2} is invertible. Moreover, there exits a constant $\bar C$, depending only on $\ell$, $\dec$ and $\Lambda$, such that
    \[ \| D_{G} \mathcal{G}^T(0,0)^{-1}\| < \bar C(\ell, \dec, \Lambda),\]
    where $\| \cdot \|$ denotes the operator norm.
\end{lemma}
\begin{proof}
This follows directly from Proposition \ref{prop:L_w_Lambda}. 
\end{proof} 

\begin{lemma}\label{lem:DG_limit_par}
       Let $P \in  \mathcal{P}^{\mathrm{par}, T}_{\sigma, \Lell^*}$ and $ G_* \in \mathcal{S}p^{T}_{\sigma - 1, \omega, 
    \mathcal{D}}$ with $|P|_{\sigma, \Lell^*}^T < r$. Then there exists a constant $C_*$, depending continuously on $\Upsilon, |P|^T_{\sigma, \Lell^*},$ and $| G_*|^T_{\sigma - 1, \omega, \mathcal{D}}$, such that the operator
    \[  D_G\mathcal{G}^T(P, G_*) - D_G\mathcal{G}^T(0, 0):   \mathcal{S}p^{T}_{\sigma - 1, \omega, 
    \mathcal{D}} \to  \mathcal{S}ym^{ T}_{\sigma - 1, \mathfrak{D}} \]
    satisfies
\begin{equation*}
\label{eq:DG_difference_bound_par}
\| D\mathcal{G}^T(0, 0) -  D\mathcal{G}^T(P, G_*) \| \leq \frac{C_*}{T^{\ell + \dec + 1}},
\end{equation*}
where $\| \cdot\|$ denotes the operator norm. 

In particular, Item \eqref{prop:DG_limit_par} of Proposition \ref{prop:G_properties_par} holds.
\end{lemma}
\begin{proof}
     We can rewrite $\mathcal{G}^T$ as
     \begin{equation*}
         \mathcal{G}^T(P, G) =  e^{G^\top}\beta(P)e^G + e^{G^\top}\Mpar e^G +  e^{G^\top} J_m \Lo  e^{G} - \Mpar,
     \end{equation*}
     where, using the notation~\eqref{def:Matrix_block}, we have that
     \begin{equation}\label{eq:prop_matrix_beta}
         \beta(P)=\begin{pmatrix} \beta_1(P) & \beta_2(P) \\ \beta_3(P) & \beta_4(P) \end{pmatrix} \qquad   \mbox{with \quad $\beta(P) \in \mathcal{M}^{T}_{\sigma-1, \mathfrak{D}}$}.
     \end{equation}
      Recalling the definition of $\mathcal{A}(\cdot, \cdot)$ in~\eqref{eq:adjoint_formula}, we have that 
     \begin{align}
        \label{eq:DG_difference_par}
        (D_G \mathcal{G}^T(P, G_*) & - D_G^T \mathcal{G}(0, 0))G  = \mathcal{A}(G_*, G)^\top e^{G^\top_*}\beta(P)e^{G_*} + e^{G^\top_*}\beta(P)e^{G_*}\mathcal{A}(G_*, G) \nonumber\\
        & + (\mathcal{A}(G_*, G)^\top e^{G^\top_*}\Mpar e^{G_*} - G^\top \Mpar) +  (e^{G^\top_*}\Mpar e^{G_*}\mathcal{A}(G_*, G) - \Mpar G)\\
        & +  \mathcal{A}(G_*, G)^\top e^{G^\top_*} J_m \Lo  e^{G_*} +  (e^{G^\top_*}  J_m \Lo  (e^{G_*} \mathcal{A}(G_*, G)) - J_m \Lo G).\nonumber
    \end{align}
    Using  Proposition \ref{prop:Csigma_prop_norms_matrix} and Lemmas \ref{lem:exp_expansion_bounds}, \ref{prop:exp_bound_par}, a straightforward calculation shows that
    \begin{equation}
    \label{eq:better_decay_DG_par}
    (D_G \mathcal{G}(P, G_*)  - D_G \mathcal{G}(0, 0))G \in \mathcal{S}ym^{T}_{\sigma - 1, \mathfrak{D}(2\ell +2\dec+1)},
    \end{equation}
where $\mathfrak{D}(\ell+\dec)$ is defined in~\eqref{eq:index_regularity_G_par_Sym}. Moreover, its norm is bounded by $C_*|G|^T_{\sigma - 1, \omega, \mathcal{D}}$, where $C_*$ is a constant depending continuously on $\sigma, \ell, \dec, m, n, \Upsilon, \Lambda, |P|^T_{\sigma, \mathcal{L}^*},$ and $|G_*|^T_{\sigma - 1, \omega, \mathcal{D}}$. For this purpose, we need to estimate each term in the RHS of~\eqref{eq:better_decay_DG_par}. In the following, we denote simply by $C$ constants appearing in the inequalities and that depend only on $\sigma, \ell, \dec, m, \Lambda$.

By to Property \ref{it:boound_expFBexpF_par} of Lemma \ref{prop:exp_bound_par}, we have that $e^{G^\top_*}\beta(P)e^{G_*} \in \mathcal{M}^{T}_{\sigma-1, \mathfrak{D}}$ with 
\begin{equation*}
     |e^{G^\top_*}\beta(P)e^{G_*}|^T_{\sigma-1, \mathfrak{D}} \le  |\beta(P)|^T_{\sigma-1, \mathfrak{D}} \left(1 + Ce^{C|G_*|_{\sigma-1,  0}^T}|G_*|_{\sigma-1, \omega, \mathcal{D}}^T  \right).
\end{equation*}
Thanks to the latter, Property \ref{it:bound_BA_par} of Lemma \ref{prop:exp_bound_par} implies that 
\begin{align*}
    &|\mathcal{A}(G_*, G)^\top e^{G^\top_*}\beta(P)e^{G_*}|^T_{\sigma-1, \mathfrak{D}(2\ell+2\dec+1)}\\
     &\qquad \le C|e^{G^\top_*}\beta(P)e^{G_*}|^T_{\sigma-1, \mathfrak{D}}\left(1 + e^{C|G_*|_{\sigma-1,  0}^T}|G_*|_{\sigma-1,\omega,  \mathcal{D}}^T \right) |G|^T_{\sigma-1,\omega,  \mathcal{D}},\\
     &\qquad \le C|\beta(P)|^T_{\sigma-1, \mathfrak{D}}\left(1 +  e^{C|G_*|_{\sigma-1,  0}^T} \left(|G_*|_{\sigma-1,\omega,  \mathcal{D}}^T + \left(|G_*|_{\sigma-1,\omega,  \mathcal{D}}^T \right)^2\right)\right) |G|^T_{\sigma-1,\omega,  \mathcal{D}}.
\end{align*}
The term $e^{G^\top_*}\beta(P)e^{G_*}\mathcal{A}(G_*, G)$ can be treated similarly.

We observe that 
\begin{equation}
\begin{aligned}\label{eq:stima_di_cuore_2_par}
    e^{G^\top_*}\Mpar e^{G_*}\mathcal{A}(G_*, G) - \Mpar G &= ( e^{G_*^\top }- \mathrm{Id}_{2m})\Mpar e^{G_*}\mathcal{A}(G_*, G)\\
    &+ \Mpar ( e^{G_*} - \mathrm{Id}_{2m})\mathcal{A}(G_*, G) + \Mpar \left(\mathcal{A}(G_*, G) - G\right).
\end{aligned}
\end{equation}
On the one hand, using Property \ref{it:bound_A-G_par} of Lemma \ref{prop:exp_bound_par}, we have that 
\begin{align*}
    &|\Mpar \left(\mathcal{A}(G_*, G) - G\right)|^T_{\sigma-1, \mathfrak{D}(2\ell +2\dec +1)}\\
    &\qquad\le Ce^{C|G_*|_{\sigma-1, 0}^T} |G_*|_{\sigma-1, \omega,  \mathcal{D}}^T |G|^T_{\sigma-1, \omega, \mathcal{D}} .
\end{align*}
On the other hand, by Property \ref{prop:crescita_indici_norm_ell} and \ref{prop:bound_norm_product_matrices} of Lemma \ref{prop:Csigma_prop_norms_matrix}, and Property \ref{it:bound_expF} of Lemma \ref{lem:exp_expansion_bounds} a straightforward computation shows that 
\begin{equation*}
    ( e^{G_*^\top }- \mathrm{Id}_{2m})\Mpar e^{G_*}, \,  \Mpar ( e^{G_*} - \mathrm{Id}_{2m})\in \mathcal{M}^{ T}_{\sigma-1, \mathcal{D}}
\end{equation*}
and hence, using Properties \ref{it:bound_expF_par} and \ref{it:bound_MA_par} of Lemma \ref{prop:exp_bound_par}, and Property \ref{it:bound_expF} of Lemma \ref{lem:exp_expansion_bounds}, we obtain that 
\begin{align*}
    &|( e^{G_*^\top }- \mathrm{Id}_{2m})\Mpar e^{G_*} \mathcal{A}(G_*, G)|^T_{\sigma-1, \mathfrak{D}(2\ell + 2\dec +1)} \\
    &\qquad \le C|( e^{G_*^\top }- \mathrm{Id}_{2m})\Mpar e^{G_*}|^T_{\sigma-1, \mathcal{D}}\left(1 + e^{C|G_*|_{\sigma-1, 0}^T} |G_*|_{\sigma-1, \omega, \mathcal{D}}^T\right)|G|^T_{\sigma-1, \omega, \mathcal{D}},\\
    &\qquad \le C e^{C|G_*|_{\sigma-1, 0}^T}\left( |G_*|_{\sigma-1,\omega,  \mathcal{D}}^T + \left(|G_*|_{\sigma-1, \omega, \mathcal{D}}^T \right)^2 \right)|G|^T_{\sigma-1,\omega,  \mathcal{D}},\\
    &|\Mpar ( e^{G_*} - \mathrm{Id}_{2m})\mathcal{A}(G_*,G)|^T_{\sigma-1, \mathfrak{D}(2\ell + 2\dec +1)} \\
    &\qquad\le C|\Mpar ( e^{G_*} - \mathrm{Id}_{2m})|^T_{\sigma-1, \mathcal{D}}\left(1 + e^{C|G_*|_{\sigma-1, 0}^T} |G_*|_{\sigma-1,\omega,  \mathcal{D}}^T \right)|G|^T_{\sigma-1,\omega,  \mathcal{D}},\\
    &\qquad\le  C e^{C|G_*|_{\sigma-1, 0}^T}\left( |G_*|_{\sigma-1,\omega,  \mathcal{D}}^T + \left(|G_*|_{\sigma-1, \omega, \mathcal{D}}^T \right)^2 \right)|G|^T_{\sigma-1,\omega,  \mathcal{D}}.
\end{align*}
   Thanks to the triangular inequality and the above estimates, by~\eqref{eq:stima_di_cuore_2_par}, we prove that 
   \begin{align*}
       &|e^{G^\top_*}\Mpar e^{G_*}\mathcal{A}(G_*, G) - \Mpar G|^T_{\sigma-1, \mathfrak{D}(2\ell +2\dec+1)}\\
       &\qquad \le  Ce^{C|G_*|_{\sigma-1, 0}^T}\bigg[ |G_*|_{\sigma-1,\omega,  \mathcal{D}}^T  + \left(|G_*|_{\sigma-1,\omega,  \mathcal{D}}^T \right)^2 \bigg]|G|^T_{\sigma-1, \omega, \mathcal{D}}.
   \end{align*}
The term $\mathcal{A}(G_*, G)^\top e^{G^\top_*}\Mpar e^{G_*} - G^\top \Mpar$ can be treated similarly. 

We write $\mathcal{A}(G_*, G)^\top e^{G^\top_*} J_m \Lo  e^{G_*}$ as
\begin{equation*}
    \mathcal{A}(G_*, G)^\top e^{G^\top_*} J_m \Lo  e^{G_*} = \mathcal{A}(G_*, G)^\top (e^{G^\top_*}-\mathrm{Id}_{2m}) J_m \Lo  e^{G_*} + \mathcal{A}(G_*, G)^\top  J_m \Lo  e^{G_*}
\end{equation*}
and using Properties \ref{it:bound_expF_par}, \ref{it:bound_NJA(F,G)_par} and \ref{it:bound_L(expG)_par} of Lemma \ref{prop:exp_bound_par}, and Property \ref{it:bound_expF} of Lemma \ref{lem:exp_expansion_bounds} we can verify that 
\begin{align*}
   &|\mathcal{A}(G_*, G)^\top e^{G^\top_*} J_m \Lo  e^{G_*}|^T_{\sigma -1, \mathfrak{D}(2\ell + 2 \dec+1)} \\
   &\qquad \le    C e^{C|G_*|_{\sigma-1, 0}^T}|\Lo e^{G_*}|^T_{\sigma-1, \mathcal{D}(\ell + \dec+1)}\left( 1 +|G_*|_{\sigma-1,\omega,  \mathcal{D}}^T\right)|G|^T_{\sigma-1, \omega, \mathcal{D}}\\
   &\qquad \le  C e^{C|G_*|_{\sigma-1, 0}^T}|\Lo {G_*}|^T_{\sigma-1, \mathcal{D}(\ell +\dec+1)}\left( 1 +|G_*|_{\sigma-1,\omega,  \mathcal{D}}^T + \left(|G_*|_{\sigma-1,\omega,  \mathcal{D}}^T\right)^2\right)|G|^T_{\sigma-1, \omega, \mathcal{D}}\\
   &\qquad \le  C e^{C|G_*|_{\sigma-1, 0}^T}| {G_*}|^T_{\sigma-1, \omega, \mathcal{D}}\left( 1 +|G_*|_{\sigma-1,\omega,  \mathcal{D}}^T + \left(|G_*|_{\sigma-1,\omega,  \mathcal{D}}^T\right)^2\right)|G|^T_{\sigma-1, \omega, \mathcal{D}},
\end{align*}
where $\mathcal{D}(\ell +\dec+1)$ is defined in~\eqref{eq:index_regularity_G_par}.

Finally,  we have that 
\begin{align*}
    e^{G^\top_*}  J_m \Lo  (e^{G_*} \mathcal{A}(G_*, G)) - J_m \Lo G &= \left(e^{G^\top_*} - \mathrm{Id}_{2m}\right)  J_m \Lo  (e^{G_*} \mathcal{A}(G_*, G))  +  J_m \Lo  (e^{G_*}) \mathcal{A}(G_*, G) \\
    &+ J_m  (e^{G_*} - \mathrm{Id}_{2m})\Lo \left(\mathcal{A}(G_*, G)\right) + J_m \Lo \left(\mathcal{A}(G_*, G) - G\right).
\end{align*}
We need to provide an upper bound for each term on the RHS of the latter. For this purpose, using Property \ref{it:bound_JLomega(A-G)} of Lemma \ref{prop:exp_bound_par}, we have that 
\begin{align*}
    &|J_m \Lo \left(\mathcal{A}(G_*, G) - G\right)|^T_{\sigma-1, \mathfrak{D}(2\ell +2\dec+1)}\\
    &\qquad \le C e^{C|G_*|_{\sigma-1, 0}^T} \left(1 + |G_*|_{\sigma-1, \mathcal{D}}^T\right) \bigg[|G|_{\sigma-1, \mathcal{D}}^T|\Lo G_*|_{\sigma-1, \mathcal{D}(2\ell +2\dec +1)}^T + |\Lo G|_{\sigma-1, \mathcal{D}(2\ell +2\dec +1)}^T|G_*|_{\sigma-1, \mathcal{D}}^T \bigg]\\
    &\qquad \le C e^{C|G_*|_{\sigma-1, 0}^T} \left(|G_*|_{\sigma-1, \omega, \mathcal{D}}^T + \left(|G_*|_{\sigma-1, \omega, \mathcal{D}}^T\right)^2\right) |G|_{\sigma-1, \omega, \mathcal{D}}^T.
\end{align*}    
By Properties \ref{it:bound_expF_par} and \ref{it:bound_JMLomegaA_par} of Lemma \ref{prop:exp_bound_par}, we obtain that 
\begin{align*}
    &|J_m  (e^{G_*} - \mathrm{Id}_{2m})\Lo \left(\mathcal{A}(G_*, G)\right)|^T_{\sigma-1, \mathfrak{D}(2\ell +2\dec+1)}\\
    &\qquad \le C |e^{G_*} - \mathrm{Id}_{2m}|^T_{\sigma-1, \mathcal{D}}|\Lo G|_{\sigma-1, \mathcal{D}(\ell+\dec+1)}^T\\
            &\qquad + C e^{C|G_*|_{\sigma-1, 0}^T}|e^{G_*} - \mathrm{Id}_{2m}|^T_{\sigma-1, \mathcal{D}}\left(1 + |G_*|_{\sigma-1, \mathcal{D}}^T\right) \bigg[|G|_{\sigma-1, \mathcal{D}}^T|\Lo G_*|_{\sigma-1, \mathcal{D}(\ell+\dec+1)}^T \\
            &\qquad + |\Lo G|_{\sigma-1, \mathcal{D}(\ell+\dec+1)}^T|G_*|_{\sigma-1, \mathcal{D}}^T\bigg] \\
            &\qquad \le C e^{C|G_*|_{\sigma-1, 0}^T}\left(|G_*|^T_{\sigma-1, \omega, \mathcal{D}} + \left(|G_*|^T_{\sigma-1, \omega, \mathcal{D}}\right)^2 + \left(|G_*|^T_{\sigma-1, \omega, \mathcal{D}}\right)^3\right) |G|^T_{\sigma-1, \omega, \mathcal{D}}.
\end{align*}

and using Properties \ref{it:bound_NJA(F,G)_par} and \ref{it:bound_L(expG)_par} of Lemma \ref{prop:exp_bound_par}, we have that 
\begin{align*}
    &|J_m \Lo  (e^{G_*}) \mathcal{A}(G_*, G)|^T_{\sigma-1, \mathfrak{D}(2\ell +2\dec+1)}\\
    &\qquad \le C|\Lo  (e^{G_*})|^T_{\sigma-1, \mathcal{D}(\ell +  \dec +1)}\left(1 + e^{C|G_*|_{\sigma-1, 0}^T}|G_*|_{\sigma-1, \mathcal{D}}^T \right) |G|^T_{\sigma-1, \omega, \mathcal{D}}\\
    &\qquad \le Ce^{C|G_*|_{\sigma-1, 0}^T}\left(|G_*|^T_{\sigma-1, \omega, \mathcal{D}} + \left( |G_*|^T_{\sigma-1, \omega, \mathcal{D}}\right)^2 + \left( |G_*|^T_{\sigma-1, \omega, \mathcal{D}}\right)^3 \right) |G|^T_{\sigma-1, \omega, \mathcal{D}}.\\
\end{align*}
It remains to analyze
\begin{align*}
    \left(e^{G^\top_*} - \mathrm{Id}_{2m}\right)  J_m \Lo  (e^{G_*} \mathcal{A}(G_*, G))  &= \left(e^{G^\top_*} - \mathrm{Id}_{2m}\right)  J_m \Lo  (e^{G_*}) \mathcal{A}(G_*, G) \\
    &+ \left(e^{G^\top_*} - \mathrm{Id}_{2m}\right)  J_m   e^{G_*}\Lo( \mathcal{A}(G_*, G)).  
\end{align*}
On the one hand, thanks to Property \ref{it:bound_exp_Lomega_exp_par} of Lemma \ref{prop:exp_bound_par}, we know that 
\begin{equation*}
    (e^{F^\top} - \mathrm{Id}_{2m})J_m \Lo (e^F) \in \mathcal{M}^{ T}_{\sigma-1, \mathfrak{D}(2\ell + 2\dec+1)} \subset \mathcal{M}^{ T}_{\sigma-1, \mathfrak{D}(\ell+\dec+1)}
\end{equation*}
and hence, using Properties \ref{it:bound_BA_par} and \ref{it:bound_exp_Lomega_exp_par} of Lemma \ref{prop:exp_bound_par}
\begin{align*}
    &|\left(e^{G^\top_*} - \mathrm{Id}_{2m}\right)  J_m \Lo  (e^{G_*}) \mathcal{A}(G_*, G)|^T_{\sigma-1, \mathfrak{D}(2\ell+2\dec+1)} \\
    &\qquad \le C |\left(e^{G^\top_*} - \mathrm{Id}_{2m}\right)  J_m \Lo  (e^{G_*})|^T_{\sigma-1, \mathfrak{D}(\ell+\dec+1)}\left(1 + e^{C|G_*|_{\sigma-1, 0}^T}|G_*|_{\sigma-1, \omega, \mathcal{D}}^T\right)|G|^T_{\sigma-1, \omega, \mathcal{D}}\\
    &\qquad \le C |\left(e^{G^\top_*} - \mathrm{Id}_{2m}\right)  J_m \Lo  (e^{G_*})|^T_{\sigma-1, \mathfrak{D}(2\ell+2\dec+1)}\left(1 + e^{C|G_*|_{\sigma-1, 0}^T}|G_*|_{\sigma-1, \omega, \mathcal{D}}^T\right)|G|^T_{\sigma-1, \omega, \mathcal{D}}\\
    &\qquad \le Ce^{C|G_*|_{\sigma-1, 0}^T} \left( \left(|G_*|_{\sigma-1, \omega, \mathcal{D}}^T\right)^2 + \left(|G_*|_{\sigma-1, \omega, \mathcal{D}}^T\right)^3 + \left(|G_*|_{\sigma-1, \omega, \mathcal{D}}^T\right)^4\right)|G|^T_{\sigma-1, \omega, \mathcal{D}}.
\end{align*}
On the other hand, by Property \ref{it:bound_expFtopJexpF} of Lemma \ref{prop:exp_bound_par}, we have that $(e^{G_*^\top} - \mathrm{Id}_{2m})J_m e^{G_*} \in \mathcal{M}^{ T}_{\sigma-1, \mathfrak{D}(\ell + \dec -1)}$ and combining it with Property \ref{it:bound_RL_omegaA(F,G)} of Lemma \ref{prop:exp_bound_par}, we obtain 
\begin{align*}
    &|\left(e^{G^\top_*} - \mathrm{Id}_{2m}\right)  J_m   e^{G_*}\Lo( \mathcal{A}(G_*, G))|^T_{\sigma-1, \mathfrak{D}(2\ell+2\dec+1)} \\
    &\qquad \le C \left|\left(e^{G^\top_*} - \mathrm{Id}_{2m}\right)  J_m   e^{G_*}\right|^T_{\sigma-1, \mathfrak{D}(\ell+\dec-1)}|\Lo G|_{\sigma-1, \mathcal{D}(\ell+\dec+1)}^T\\
    &\qquad + C e^{C|G_*|_{\sigma-1, 0}^T}\left|\left(e^{G^\top_*} - \mathrm{Id}_{2m}\right)  J_m   e^{G_*}\right|^T_{\sigma-1, \mathfrak{D}(\ell+\dec-1)}\left(1 + |G_*|_{\sigma-1, \mathcal{D}}^T\right) \bigg[|G|_{\sigma-1, \mathcal{D}}^T|\Lo G_*|_{\sigma-1, \mathcal{D}(\ell+\dec+1)}^T \\
    &\qquad + |\Lo G|_{\sigma-1, \mathcal{D}(\ell+\dec+1)}^T|G_*|_{\sigma-1, \mathcal{D}}^T\bigg]\\
    &\qquad \le Ce^{C|G_*|_{\sigma-1, 0}^T}\left(|G_*|^T_{\sigma-1, \omega, \mathcal{D}} + \left( |G_*|^T_{\sigma-1, \omega, \mathcal{D}}\right)^2 + \left( |G_*|^T_{\sigma-1, \omega, \mathcal{D}}\right)^3 \right) |G|^T_{\sigma-1, \omega, \mathcal{D}}.
\end{align*}

Therefore, there exists $C_*$ depending continuously on $\sigma, \ell, \dec, m, n, \Upsilon, \Lambda, |P|^T_{\sigma, \Lell^*},$ and $|G_*|^T_{\sigma - 1, \omega, \mathcal{D}}$ such that
\begin{align*}
 \big|(D_G\mathcal{G}^T(0,0) - D_G\mathcal{G}^T(P, G_*))G\big|^T_{\sigma - 1, \mathfrak{D}} & \leq  \frac{1}{T^{\ell +\dec+ 1}}\big|(D_G\mathcal{G}^T(0,0) - D_G\mathcal{G}^T(P, G_*))G\big|_{\sigma - 1, \mathfrak{D}(2\ell+2\dec+1)}  \\
 & \leq \frac{C_*}{T^{\ell +\dec + 1}} |G|^T_{\sigma - 1, \omega, \mathcal{D}}.
\end{align*}

\end{proof}

\begin{remark}
We point out that it is possible to prove the existence of a family of matrices $S^t$ satisfying~\eqref{eq:conjugated_cocycles} and $$\lim_{t \to +\infty}\left|S^t - \mathrm{Id}_{2n+2m}\right|_{C^0}=0$$ under the slightly weaker decay assumption
    \begin{equation*}
    \begin{aligned}
    a \in \mathcal{S}^{ 1}_{(\sigma, 2),0}, \quad \partial_q a &\in \mathcal{S}_{(\sigma, 1), \ell+\dec+5}^{1}, \quad b \in \mathcal{S}_{(\sigma, 2), \ell}^{1}, \quad  c_1 \in \mathcal{S}_{(\sigma, 2), \ell+\dec+3}^{ 1}, \quad  c_2 \in \mathcal{S}_{(\sigma, 2), \ell+\dec+4}^{ 1}, \\
         d_1 &\in \mathcal{S}_{(\sigma, 2), \ell+\dec}^{ 1}, \quad d_2 \in \mathcal{S}_{(\sigma, 2), \ell+\dec+1}^{ 1}, \quad d_3 \in \mathcal{S}_{(\sigma, 2), \ell+\dec+2}^{ 1}
         \end{aligned}
    \end{equation*}
    instead of~\eqref{eq:Hyp_decay_par_2}. However, in order to prove also~\eqref{def:trasv_dyn_limit}, the stronger decay assumption~\eqref{eq:Hyp_decay_par_2} is required (see Corollary \ref{cor:criterion}).
\end{remark}

\subsection{Proof of Theorem \ref{Thm:par_analy}}
\label{sc:proof_analytic_par} 
The proof of Theorem \ref{Thm:par_analy} is completely analogous to that of Theorem \ref{Thm:par_Csigma_0_analy}. Indeed, using the same arguments as in Section \ref{sc:proof_analytic_par0}, it suffices to prove the first part of Theorem \ref{Thm:par_analy}. This can be done by following step by step the proof of Item \eqref{thm:C_existence} of Theorem \ref{Thm:par_Csigma}, with the only modification that the domain of the functional $\mathcal{F}^T$ must be adapted as in Section \ref{sc:proof_analytic_par0}, and by considering the spaces of real analytic functions introduced there instead of the spaces of Hölder functions.

\subsection{Proof of Corollary \ref{cor:par}}\label{sec:proof_cor_par}
We recall that the component $a_{32}$ of the matrix $A$ in~\eqref{def:A} in terms of the asymptotic KAM torus $\varphi=(\mathrm{id}_{\T^n} + u,v,w)  = (\mathrm{id}_{\T^n} + u,v,w_x, w_y) $ and the components of the matrix $S$ (see~\eqref{eq:def_A_Hamiltonian}) has the following form
  \begin{equation*}
     a_{32} = e^{G^\top}
\left(\partial_z\partial_pH\circ\varphi\right)
(\mathrm{id}_{\T^n}+\partial_qu)^{-\top}- e^{G^\top} J_m\partial_qw
(\mathrm{id}_{\T^n}+\partial_qu)^{-1}
\left(\partial_p^2H\circ\varphi\right)
(\mathrm{id}_{\T^n}+\partial_qu)^{-\top}
 \end{equation*}
 where for $\ell >3$, $\dec \ge 0$ and $T'' \ge 1$
  \begin{equation}\label{cor_proof:decay_uvwG_par}
     u \in \mathcal{S}_{(\sigma, 0), \ell-1}^{T''}, \hspace{2mm}  v \in \mathcal{S}_{(\sigma, 0), \ell+ \dec + 5}^{T''}, \hspace{2mm} w_x \in \mathcal{S}_{(\sigma, 0), \ell+ \dec+4}^{T''}, \hspace{2mm} w_y \in \mathcal{S}_{(\sigma, 0), \ell+ \dec+3}^{T''}, \hspace{2mm} G \in \mathcal{S}_{(\sigma, 0), \mathcal{D}(\ell +\dec)}^{T''}
 \end{equation}
where we refer to~\eqref{eq:index_regularity_G_par} for the definition of $\mathcal{D}(\ell +\dec)$. Notice that here we assume $\ell >3$ instead of $\ell >1$.

The proof of Corollary \ref{cor:par} is an immediate consequence of 
Proposition \ref{prop:par_1:1_asym_sol},~\eqref{eq:conjugated_cocycles}, and~\eqref{def:trasv_dyn_limit}. We only need to verify whether the hypothesis~\eqref{hyp:limit_aut_sys_par} is satisfied. In this setting, it consists of proving that 
\begin{equation}\label{proof_cor:cond_to_verify_par}
\int_{T}^{+\infty} (s - T) |\Pi_x a_{32}^s - \partial_x\partial_p H^\infty\circ \varphi_0|_{C^0} \, ds < \infty, \quad \int_{T}^{+\infty} |\Pi_y a_{32}^s - \partial_y\partial_p H^\infty\circ \varphi_0|_{C^0} \, ds < \infty       
\end{equation}
where $\varphi_0$ is the trivial embedding in~\eqref{def:varphi0=(q,0,0)} and $\Pi_x$ and $\Pi_y$ the projections defined in~\eqref{not:proj_Pixy}. %
It follows that
 \begin{align*}
        &e^{G^\top}
\left(\partial_z\partial_pH\circ\varphi\right)
(\mathrm{id}_{\T^n}+\partial_qu)^{-\top} - \partial_z\partial_pH^\infty\circ\varphi_0  \\
&\qquad \qquad =\left(\partial_z\partial_pH\right)_0 - \left(\partial_z\partial_pH^\infty\right)_0 + \int_0^1 \partial_q\partial_z\partial_pH(\mathrm{id}_{\T^n} + \tau u, \tau v, \tau w) d \tau \, u\\
        &\qquad \qquad + \int_0^1 \partial_p\partial_z\partial_pH(\mathrm{id}_{\T^n} + \tau u, \tau v, \tau w) d \tau \, v +  \int_0^1 \partial_z\partial_z\partial_pH(\mathrm{id}_{\T^n} + \tau u, \tau v, \tau w) d \tau \, w  \\
&\qquad \qquad +  {G^\top}\int_0^1e^{\tau {G^\top}}\,d\tau \,\left(\partial_z\partial_pH\circ\varphi\right) - \partial_z\partial_pH\circ\varphi \int_0^1 (\mathrm{id}_{\T^n}+\tau\partial_qu)^{-\top}(\partial_qu)^{\top}(\mathrm{id}_{\T^n}+\tau\partial_qu)^{-\top}\,d\tau\\
&\qquad\qquad - {G^\top}\int_0^1e^{\tau {G^\top}}\,d\tau \,\left(\partial_z\partial_pH\circ\varphi\right)\int_0^1 (\mathrm{id}_{\T^n}+\tau\partial_qu)^{-\top}(\partial_qu)^{\top}(\mathrm{id}_{\T^n}+\tau\partial_qu)^{-\top}\,d\tau.
    \end{align*}
    where we used the notation (subscript $0$) introduced in ~\eqref{def:f0}. 
Noticing that 
\begin{equation*}
    \int_0^1 \partial_q\partial_x\partial_pH(\mathrm{id}_{\T^n} + \tau u, \tau v, \tau w) d \tau \, u, \, \int_0^1 \partial_q\partial_y\partial_pH(\mathrm{id}_{\T^n} + \tau u, \tau v, \tau w) d \tau \, u \in \mathcal{S}^{ T^{''}}_{(\sigma, 0), \ell-1}
\end{equation*}
and by~\eqref{cor_proof:decay_uvwG_par}, hypothesis~\eqref{hyp:cor_par_decay} and the properties contained in Section \ref{sc:norm_properties_par}, we prove that 
\begin{equation}\label{proof_cor_par:decay_prop_1}
    \begin{gathered}
       \scalebox{0.97}{$e^{G^\top}
\left(\partial_z\partial_pH\circ\varphi\right)
(\mathrm{id}_{\T^n}+\partial_qu)^{-\top} - \partial_z\partial_pH^\infty\circ\varphi_0 
 - \left(\left(\partial_z\partial_pH\right)_0  - \left(\partial_z\partial_pH^\infty\right)_0\right) \in \mathcal{S}_{(\sigma, 0), \ell-1}^{T''},$}\\
\scalebox{0.95}{$\displaystyle \int_{T''}^{+\infty} (s - T^{''})|\left(\partial_x\partial_pH\right)_0  - \left(\partial_x\partial_pH^\infty\right)_0|_{C^0} ds < \infty, \quad \int_{T''}^{+\infty} |\left(\partial_y\partial_pH\right)_0  - \left(\partial_y\partial_pH^\infty\right)_0|_{C^0} ds < \infty.$}
    \end{gathered}
\end{equation}
Similarly, we have 
\begin{equation}
\begin{aligned}\label{proof_cor_par:decay_prop_2}
    &\Pi_xe^{G^\top} J_m(\partial_qw)^\top
(\mathrm{id}_{\T^n}+\partial_qu)^{-1}
\left(\partial_p^2H\circ\varphi\right)
(\mathrm{id}_{\T^n}+\partial_qu)^{-\top} \in \mathcal{S}_{(\sigma, 0), \ell+\dec+2}^{T''},\\
&\Pi_ye^{G^\top} J_m(\partial_qw)^\top
(\mathrm{id}_{\T^n}+\partial_qu)^{-1}
\left(\partial_p^2H\circ\varphi\right)
(\mathrm{id}_{\T^n}+\partial_qu)^{-\top} \in \mathcal{S}_{(\sigma, 0), \ell+\dec+3}^{T''}.
\end{aligned}
\end{equation}
Combining~\eqref{proof_cor_par:decay_prop_1} and~\eqref{proof_cor_par:decay_prop_2} and remembering that $\ell >3$, we prove~\eqref{proof_cor:cond_to_verify_par} and we conclude the proof of this corollary. 

\subsection{Cohomological equations}\label{sc:cohom_eqs_par}
In this section, we study in detail the cohomological equations appearing in the proof of Theorems \ref{Thm:par_Csigma} and \ref{Thm:par_analy}, namely, the three equations in \eqref{eq:coh_eqs_par} restricted to the Hölder setting (see Sections \ref{sec:proof_item_1_par_Holder}  and \ref{sec:proof_item_2_par_Holder} and the first two equations in \eqref{eq:coh_eqs_par} in the analytic setting (see Section \ref{sc:proof_analytic_par}), respectively.

The first equation was already treated in Section Part I \cite{scarcella_KAMST}.

\begin{proposition}[Proposition 6.13 in \cite{scarcella_KAMST}]\label{prop:Csigma_HEomega} 
       Given $\sigma \ge 0$, $T \ge 1$, and $\ell >0$,  the operator
    \begin{align*}
        \Lo : \mathcal{U}^{ T}_{\sigma, \omega, \ell} \longrightarrow \mathcal{S}^{ T}_{(\sigma, 0), \ell +1}(\R^n)
    \end{align*}
    is well-defined, invertible and satisfies
    \[ \| \Lo ^{-1}\| \leq C(\ell),\]
    for some $C > 0$ depending only on $\ell$.
    
   In addition, if $\sigma \ge 2$, then, for any $u \in \mathcal{U}^{ T}_{\sigma, \omega, \ell}$
    \begin{equation}\label{eq:prop_dist_Lomega_U}
        \partial_q\left(\Lo u\right) = \Lo \partial_q  u .
    \end{equation}
    Moreover, given $\sigma > 0$, $T \ge 1$, and $\ell >0$, the same well-definedness, invertibility, and estimate hold if one replaces the Banach spaces $\mathcal{U}^{ T}_{\sigma, \omega, \ell}$ and $\mathcal{S}^{ T}_{(\sigma, 0), \ell +1}$ by $\mathscr{U}^{ T}_{\sigma, \omega, \ell}$ and $\mathscr{S}^{ T}_{\sigma, \ell +1}$, respectively. Moreover, in this case, the equality~\eqref{eq:prop_dist_Lomega_U} holds for any $\sigma>0$ and $u \in \mathscr{U}^{ T}_{\sigma, \omega, \ell}$.
\end{proposition}

We refer the reader to Part I \cite{scarcella_KAMST} for a proof of the proposition above.

\begin{proposition}\label{prop:L_w_par}
       Given $\sigma \ge 1$, $T \ge 1$, and $\ell>0$,  the operator
    \begin{align*}
         \mathcal{L}^{\mathrm{par}}_{\omega, \Lambda} : \mathcal{U}^{ T}_{\sigma, \omega, \ell + 1} \times  \mathcal{U}^{ T}_{\sigma, \omega, \ell}\longrightarrow \mathcal{S}^{ T}_{(\sigma, 0), \ell +2}(\R^m) \times \mathcal{S}^{ T}_{(\sigma, 0), \ell +1}(\R^m),
    \end{align*}
    as in \eqref{eq:coh_eqs_par}, is well-defined, invertible, and satisfies
    \[ \left \| ({\mathcal{L}^{\mathrm{par}}_{\omega, \Lambda}})^{-1} \right\| \leq C(\ell),\]
    for some $C > 0$ depending only on  $\ell$. 
    
    Moreover, the same well-definedness, invertibility, and estimate hold if one replaces the Banach spaces $\mathcal{U}^{ T}_{\sigma, \omega, \ell}$ and $\mathcal{S}^{ T}_{(\sigma, 0), \ell +1}$ by $\mathscr{U}^{ T}_{\sigma, \omega, \ell}$ and $\mathscr{S}^{ T}_{\sigma, \ell +1}$, respectively.
\end{proposition}

\begin{proof}
Let $(g_x, g_y) \in\mathcal{S}^{ T}_{(\sigma, 0), \ell +2}(\R^m) \times \mathcal{S}^{ T}_{(\sigma, 0), \ell +1}(\R^m).$ Notice that the equation 
\[  \mathcal{L}^{\mathrm{par}}_{\omega, \Lambda}(w_x, w_y) = (g_x, g_y)\]
is equivalent to the system
\[ \left\{ \begin{array}{l} 
-\Lo  w_x = g_x,\\
 -\Lo w_y - \Lambda \cdot w_x = g_y.
 \end{array}\right. \]

The result follows by applying twice Proposition \ref{prop:Csigma_HEomega}.
\end{proof}

Let us recall that $\mathcal{D} = \mathcal{D}(\ell + \dec)$ and $\mathfrak{D} = \mathfrak{D}(\ell + \dec)$ are given by \eqref{eq:index_regularity_G_par} and \eqref{eq:index_regularity_G_par_Sym}, respectively
\begin{proposition}\label{prop:L_w_Lambda} 
    Given $\sigma \ge 1$, $T\ge 1$, and $\ell >0$, the following operator
    \begin{equation*}
        \mathfrak{L}^{\mathrm{par}}_{\omega, \Lambda} : \mathcal{S}p^{ T}_{\sigma, \omega, \mathcal{D}} \longrightarrow \mathcal{S}ym^{ T}_{\sigma, \mathfrak{D}}
    \end{equation*}
    is well-defined and invertible. Moreover, 
    \[
\| (\mathfrak{L}^{\mathrm{par}}_{\omega, \Lambda})^{-1} \| \leq C(\ell, \dec, \Lambda),\]
    for some $C > 0$ depending only on $\ell$, $\dec$ and $\Lambda$.
\end{proposition}

\begin{proof}
     Let $g \in \mathcal{S}ym^{ T}_{\sigma, \mathfrak{D}}$. Hence, it can be written as
\[ g = \begin{pmatrix}
X & Y \\ Y^{\top} & -Z
\end{pmatrix},  \qquad  \begin{array}{ll}  Y : \T^n \times I_T \to \mathcal{M}_m(\R), \quad X, Z : \T^n \times I_T \to  \textup{Sym}(m, \R), \\
X \in \mathcal{M}^{T}_{\sigma, \ell+\dec+1}, \,Y \in \mathcal{M}^{T}_{\sigma, \ell+\dec+2}, \, Z \in \mathcal{M}^{T}_{\sigma, \ell+\dec +3}  \end{array} \]
we refer to~\eqref{def:BS_ell_M_Sym} for the definition of the space $\mathcal{M}^{ T}_{\sigma, \ell + \dec}$.
    We also recall that, if  $G \in \mathcal{S}p^{ T}_{\sigma, \omega, \mathcal{D}}$, then it can be written as in~\eqref{G_decay_par}.
The proof reduces to finding the unique solution $G \in \mathcal{S}p^{ T}_{\sigma, \omega, \mathcal{D}}$ of the equation $G^{\top}\Mpar + \Mpar G + J_m \Lo G =g$, which, in the above notation, reads
\begin{equation}\label{eq:HE_par_tras}
    \left\{ \begin{array}{llll}
G_1^{\top}\Lambda + \Lambda G_1 + \Lo G_3 = X, \\
G_2 \Lambda  - \Lo G_1= Y^\top, \\
 \Lo G_2 = Z.
 \end{array} \right.
 \end{equation}
From this point on, the strategy of the proof is similar to that of Lemma~\ref{lemma:DF_invertible_par}. For this reason, some parts of the argument, especially the computations, are only sketched.

 We observe that the last equation of system~\eqref{eq:HE_par_tras} can be rewritten as a system of $m^2$ decoupled equations
 \begin{equation*}
    \Lo (G_2)_{ij} = Z_{ij}
 \end{equation*}
for $1 \le i,j \le m$. Hence, thanks to Proposition \ref{prop:Csigma_HEomega}, there exists a unique solution $G_2$ of the last equation of system~\eqref{eq:HE_par_tras} such that $G_2 \in \mathcal{M}^{T}_{\sigma, \ell+\dec+2}$, and $\Lo G_2 \in \mathcal{M}^{T}_{\sigma, \ell+\dec+3}$, and satisfying
\begin{equation}\label{proof:lemma_Inv_par_1_est_G2}
   \max\left\{|G_2|^T_{\sigma, \ell + \dec + 2}, \, |\Lo G_2|^T_{\sigma, \ell + \dec + 3}\right\}  \le C(\ell, \dec) |Z|^T_{\sigma, \ell + \dec + 3}.
\end{equation}
We solve the other two equations of~\eqref{eq:HE_par_tras} analogously. Now, in the second equation of~\eqref{eq:HE_par_tras}, $G_2$ is known. Hence, applying Proposition   \ref{prop:Csigma_HEomega} as in the previous case, we prove the existence of a unique solution $G_1$ satisfying $G_1 \in \mathcal{M}^{T}_{\sigma, \ell+\dec+1}$, and $\Lo G_1 \in \mathcal{M}^{T}_{\sigma, \ell+\dec+2}$.Moreover, using the properties contained in Proposition \ref{prop:Csigma_prop_norms_matrix} and~\eqref{proof:lemma_Inv_par_1_est_G2}, we obtain that 
\begin{equation}\label{proof:lemma_Inv_par_1_est_G1}
   \max\left\{|G_1|^T_{\sigma, \ell + \dec+1}, \, |\Lo G_1|^T_{\sigma, \ell + \dec + 2}\right\}  \le C(\ell, \dec, \Lambda) \left(|Y|^T_{\sigma, \ell + \dec + 2} + |Z|^T_{\sigma, \ell + \dec + 3}\right).
\end{equation}
Finally, we can solve the first equation of~\eqref{eq:HE_par_tras} where $G_1$ and $G_2$ are known. Similarly, one can prove the existence of a unique solution $G_3$ of the first equation of~\eqref{eq:HE_par_tras} satisfying $G_3 \in \mathcal{M}^{T}_{\sigma, \ell+\dec}$, and $\Lo G_3 \in \mathcal{M}^{T}_{\sigma, \ell+\dec+1}$, and  
\begin{equation*}
   \max\left\{|G_3|^T_{\sigma, \ell + \dec}, \, |\Lo G_3|^T_{\sigma, \ell + \dec+1}\right\}  \le C(\ell, \dec, \Lambda) \left(|X|^T_{\sigma, \ell + \dec+1} + |Y|^T_{\sigma, \ell + \dec + 2} + |Z|^T_{\sigma, \ell + \dec + 3}\right).
\end{equation*}
    This concludes the proof of this lemma. 
\end{proof}

\subsection{Norm properties}
\label{sc:norm_properties_par}
In this section, we recall/prove several properties for the norms of the Banach spaces introduced in Sections \ref{sc:functional_setting_par} and \ref{sc:proof_analytic_par}.

\begin{proposition}\label{prop:Csigma_prop_norms_pol}
    Given $\sigma$, $\ell$, $d \ge 0$, $T\ge 1$ and a non-negative integer $k$, for all $f \in \mathcal{S}^{T}_{(\sigma, k), \ell}$ and $g \in \mathcal{S}^{ T}_{(\sigma, k), d}$ we have the following properties.
\begin{enumerate}
\item For all $s\ge 0$ and  $\beta \in \N^{2(n + m)}$, if $|\beta| + s \le \sigma + k$, then  
\begin{equation*}
\left|\partial^{\beta}_{(q,p,z)} f \right|^T_{s, \ell} \le C |f|^T_{\sigma +k, \ell}
\end{equation*}
\item For all $\ell' \ge 0$, if $f \in \mathcal{S}^{T}_{(\sigma, k), \ell+\ell'}$ then $|f|^T_{\sigma+k, \ell}  \le T^{-\ell'} |f|^T_{\sigma+k, \ell+\ell'}$,
\item $|fg|^T_{\sigma+k, \ell+d} \le C(\sigma, k)\left(|f|^T_{0,\ell}|g|^T_{\sigma+k,d} + |f|^T_{\sigma+k,\ell}|g|^T_{0,d}\right)$. 
\item We consider $\sigma \ge 1$,  and we assume that $g:\T^n \times B \times I_T \to \T^n \times B$.  %
Letting $\tilde g: \T^n \times B \times I_T \to \T^n \times B \times I_T$ such that $\tilde g(q,p,z,t) = (g(q,p,z,t), t)$, 
then $f \circ  \tilde g \in \mathcal{S}^{ T}_{\sigma+k, \ell}$ and 
\begin{equation*}
|f \circ \tilde g|^T_{\sigma+k, \ell} \le C(\sigma, k) \left(|f|^T_{\sigma+k,\ell}\left(|\partial_{(q,p,z)} g|^T_{0,d}\right)^\sigma + |f|^T_{1,\ell}|\partial_{(q,p,z)}  g|^T_{\sigma+k-1,d} +  |f|^T_{0, \ell}  \right).
 \end{equation*}
\end{enumerate}
\end{proposition}
\begin{proof}
    We refer to~\cite{Sca25} for the proof. 
\end{proof}

 As a direct consequence of Proposition \ref{prop:Csigma_prop_norms_pol}, we have the following.

\begin{proposition}\label{prop:Csigma_prop_norms_matrix}
    Let $\sigma, \ell, d\ge 0$ and $T\ge 1$. For any $F \in  \mathcal{M}^{ T}_{\sigma, \ell}$ and $G \in  \mathcal{M}^{ T}_{\sigma, d}$ we have the following properties.
\begin{enumerate}
\item For all $s\ge 0$ and  $\beta \in \N^{2(n + m)}$, if $|\beta| + s \le \sigma $, then  
\begin{equation*}
\left|\partial^{\beta}_{(q,p,z)} F \right|^T_{s, \ell} \le C |F|^T_{\sigma , \ell}
\end{equation*}
\item \label{prop:crescita_indici_norm_ell} For all $\ell' \ge 0$, if $F \in \mathcal{M}^{ T}_{\sigma, \ell+\ell'}$ then $|F|^T_{\sigma, \ell}  \le T^{-\ell'} |F|^T_{\sigma, \ell+\ell'}$, 
\item \label{prop:bound_norm_product_matrices} $|FG|^T_{\sigma, \ell+d} \le C(\sigma)\left(|F|^T_{0,\ell}|G|^T_{\sigma,d} + |F|^T_{\sigma,\ell}|G|^T_{0,d}\right)$. %
\end{enumerate}
\end{proposition}

Recall that the derivative of the exponential operator $\exp: \mathcal{M}_{2m}(\R) \to \mathcal{M}_{2m}(\R)$ at $F \in \mathcal{M}_{2m}$ and evaluated at $G \in  \mathcal{M}_{2m}$ is given by
\begin{equation}
\label{eq:derivative_exp}
D \exp(F)G = e^{F} \mathcal{A}(F, G),
\end{equation}
where
\begin{equation}
\label{eq:adjoint_formula}
\mathcal{A}(F, G) := \int_0^1e^{-sF}Ge^{sF}ds.
\end{equation}
Similarly, given a differentiable map $s \mapsto G(s)$, the derivative of the map $s \mapsto e^{G(s)}$ is given by
\begin{equation}
\label{eq:derivative_exp_parameter}
\frac{d }{ds} \exp (F(s)) = e^{F(s)}\mathcal{A}(F(s), F'(s )).
\end{equation}

The following properties follow easily from the definitions together with Proposition \ref{prop:Csigma_prop_norms_matrix}.

\begin{lemma}
\label{lem:exp_expansion_bounds}
      Let $\sigma, \ell, d \ge 0$ and $T\ge 1$. For any $B \in \mathcal{M}^{ T}_{\sigma, 0},$ $F \in  \mathcal{M}^{ T}_{\sigma, \ell},$ and $G \in  \mathcal{M}^{ T}_{\sigma, d},$ the following statements hold.

      \begin{enumerate}
          \item \label{it:bound_expF} $|e^B|_{\sigma, 0}^T \leq e^{C(\sigma)|B|_{\sigma, 0}^T}.$
          \item \label{it:bound_expF-id} $|e^F - \mathrm{Id}_{2m}|_{\sigma, \ell}^T \leq  |F|_{\sigma, \ell}^Te^{C(\sigma)|F|_{\sigma, 0}^T}$.       
          \item \label{it:bound_expF-id-F} $|e^F - \mathrm{Id}_{2m} - F|_{\sigma, 2\ell}^T \leq (|F|_{\sigma, \ell}^T)^2e^{C(\sigma)|F|_{\sigma, 0}^T}$.
          \item \label{it:bound_adjoint} $|\mathcal{A}(B, G)|_{\sigma, d}^T \leq C(\sigma)|G|_{\sigma, d}^Te^{C(\sigma)|B|_{\sigma, 0}^T}$.
          \item \label{it:bound_adjoint-G} $|\mathcal{A}(F, G) - G|_{\sigma, \ell + d}^T \leq C(\sigma)|F|_{\sigma, \ell}^T |G|_{\sigma, d}^T e^{C(\sigma)|F|_{\sigma, 0}^T}.$
          \item \label{it:bound_operator_L} If $\Lo B \in \mathcal{M}_{\sigma, \ell}^{ T}$ then $|\Lo (e^B)|_{\sigma, \ell}^T \leq C(\sigma) e^{C(\sigma)|B|^T_{\sigma, 0}}|\Lo B|_{\sigma, \ell}^T$.
          \item \label{it:bound_operator_LA} If $\Lo B \in \mathcal{M}_{\sigma, 1}^{T}$ and $\Lo G \in \mathcal{M}_{\sigma, d + 1}^{T}$ then $$|\Lo (\mathcal{A}(B, G))|_{\sigma, d + 1}^T \leq C(\sigma) e^{C(\sigma)|B|^T_{\sigma, 0}}(|\Lo B|_{\sigma, 1}^T|G|_{\sigma, d}^T + |\Lo G|_{\sigma, d + 1}^T).$$
          \item \label{it:bound_operator_LA-G} If $\Lo F \in \mathcal{M}_{\sigma, \ell + 1}^{T}$ and $\Lo G \in \mathcal{M}_{\sigma, d + 1}^{T}$ then $$|\Lo (\mathcal{A}(F, G) - G)|_{\sigma, \ell + d + 1}^{T} \leq C(\sigma) e^{C(\sigma)|Fe|^T_{\sigma, 0}}(|\Lo F|_{\sigma, \ell + 1}^T|G|_{\sigma, d}^T + |F|_{\sigma, \ell}^T |\Lo G|_{\sigma, d + 1}^T).$$
      \end{enumerate}
      
\end{lemma}

\begin{proposition}\label{prop:properties_analy_pol}
Given $\sigma>0$, $\ell, \, d \ge 0$ and $T\ge 1$,  for all $f \in \mathscr{S}^{ T}_{\sigma, \ell}$ and $g \in \mathscr{S}^{ T}_{\sigma, d}$, we have the following properties    
\begin{enumerate}
    \item For all $0<\sigma'<\sigma$ and $\beta \in \N^{2(n+m)}$, then
    \begin{equation*}
        \left|\partial^{\beta}_{(q,p,z)} f \right|^T_{\sigma', \ell} \le {C \over (\sigma-\sigma')^{|\beta|}} |f|^T_{\sigma, \ell}.
    \end{equation*}
    \item For all $\ell'\ge 0$, if $f \in \mathscr{S}^{ T}_{\sigma, \ell+\ell'}$ then $|f|^T_{\sigma, \ell} \le T^{-\ell'}|f|^T_{\sigma, \ell+\ell'}$.
    \item $|fg|^T_{\sigma, \ell +d} \le |f|^T_{\sigma, \ell}|g|^T_{\sigma, d}$. In particular, if $d >0$ then $|fg|_{\sigma, \ell}^T \to 0$ as $T\to +\infty$.
    \item For all $0<\sigma'<\sigma$, we consider $g:\T^n_{\sigma'} \times B_{\sigma'} \times I_T \to \T^n_\sigma \times B_\sigma$ such that $g \in \mathscr{S}^{ T}_{\sigma',0}$. Letting $\tilde g: \T^n_{\sigma'} \times B_{\sigma'} \times I_T \to \T^n_\sigma \times B_\sigma \times I_T$ such that $\tilde g(q,p,z,t) = (g(q,p,z,t), t)$, 
then $f \circ  \tilde g \in \mathscr{S}^{ T}_{\sigma', \ell}$ and 
\begin{equation*}
    |f\circ \tilde g|^T_{\sigma', \ell} \le |f|^T_{\sigma, \ell}.
\end{equation*}
\end{enumerate}
\end{proposition}

In what follows, we analyze some properties of matrices $G \in \mathcal{M}^{ T}_{\sigma, \mathcal{D}}$ whose components exhibit a prescribed decay in time. For this reason, given $d >1$, for the rest of this section, we define
\begin{equation}\label{def:vector_D_diff_decay}
    \mathcal{D}(d) = (d+1, d+2, d, d+1).
\end{equation}
The following lemma provides a control on the product of matrices in the space  $\mathcal{M}^{ T}_{\sigma, \mathcal{D}(d)}$.

\begin{lemma}\label{lemma:Prod_matrx_diff_decay}     Let $\sigma \ge 0$, $d >1$ and $\mathcal{D}(d)$ as in~\eqref{def:vector_D_diff_decay}. For any $k \ge 1$, and $G \in \mathcal{M}^{ T}_{\sigma, \mathcal{D}(d)}$, there exists a positive constant $C=C(\sigma, m)$, depending only on $\sigma$ and $m$, such that  
    \begin{equation*}
        |G^k|_{\sigma, \mathcal{D}(d)}^T \le C|G|_{\sigma, \mathcal{D}(d)}^T \left(|G|_{\sigma, 0}^T\right)^{k-1}.
    \end{equation*}
\end{lemma}
\begin{proof}
    We prove this lemma by induction on $k$. The statement is trivial for $k=1$. We assume that it holds for some $k\ge 1$. We then prove it for $k+1$. Throughout this proof, for the sake of simplicity, we denote $\mathcal{D}(d)$ simply by $\mathcal {D}$.    
    Using the notation in~\eqref{def:Matrix_block}, we can write the matrix $G \in \mathcal{M}^{ T}_{\sigma, \mathcal{D}}$ as
    \begin{equation*}
        G = \begin{pmatrix}
G_1 & G_2 \\ G_3 & G_4
\end{pmatrix},  \qquad  \begin{array}{ll}  G_j : \T^n \times I_T \to \mathcal{M}_m(\R), \qquad \mbox{for any $1 \le j \le 4$ }, \\
G_1, \, G_4 \in \mathcal{M}^{ T}_{\sigma, d+1}, \quad G_2 \in \mathcal{M}^{ T}_{\sigma, d+2}, \quad G_3 \in \mathcal{M}^{ T}_{\sigma, d} \end{array}
    \end{equation*}
we refer to~\eqref{def:BS_ell_M_Sym} for the definition of the space $\mathcal{M}^{ T}_{\sigma, d}$. We introduce the following notation in order to characterize the iterate $G^k$, which, by the inductive hypothesis, satisfies
\begin{equation*}
        G^k = \begin{pmatrix}
P_k & Q_k \\ R_k & S_k
\end{pmatrix},  \qquad  \begin{array}{ll}  P_k, \, Q_k, \, R_k, \, S_k : \T^n \times I_T \to \mathcal{M}_m(\R),  \\
P_k, \, S_k \in \mathcal{M}^{ T}_{\sigma, d+1}, \quad Q_k \in \mathcal{M}^{ T}_{\sigma, d+2}, \quad R_k \in \mathcal{M}^{ T}_{\sigma, d}. \end{array}
    \end{equation*}
    Using the above notation, we observe that 
    \begin{equation*}
        G^{k+1} = G^k G = \begin{pmatrix} P_k G_1 + Q_k G_3 & P_kG_2+Q_kG_4\\ R_kG_1+S_kG_3 & R_k G_2 + S_kG_4  \end{pmatrix}.
    \end{equation*}
Thanks to the inductive hypotheses, Proposition \ref{prop:Csigma_prop_norms_matrix} and recalling the definition of the norm $|\cdot|^T_{\sigma, \mathcal{D}}$ in~\eqref{def:norm_M_D}, we obtain that 
\begin{align*}
    |P_k G_1 + Q_k G_3|^T_{\sigma, d+1} &\le C(\sigma, m)\left(|P_k|^T_{\sigma, d+1}|G_1|^T_{\sigma, 0} + |Q_k|^T_{\sigma, d+2}|G_3|^T_{\sigma, 0}\right) \le C(\sigma, m)|G|_{\sigma, \mathcal{D}}^T \left(|G|_{\sigma, 0}^T\right)^{k-1}|G|_{\sigma, 0}^T,\\
    |P_k G_2 + Q_k G_4|^T_{\sigma, d+2} &\le C(\sigma, m)\left(|P_k|^T_{\sigma, 0}|G_2|^T_{\sigma, d+2} + |Q_k|^T_{\sigma, d+2}|G_4|^T_{\sigma, 0}\right)\\
    &\le C(\sigma, m)\left(\left(|G|_{\sigma, 0}^T \right)^k|G|_{\sigma, \mathcal{D}}^T + |G|_{\sigma, \mathcal{D}}^T \left(|G|_{\sigma, 0}^T\right)^{k-1}|G|_{\sigma, 0}^T\right),\\
    |R_k G_1 + S_k G_3|^T_{\sigma, d} &\le C(\sigma, m)\left(|R_k|^T_{\sigma, d}|G_1|^T_{\sigma, 0} + |S_k|^T_{\sigma, d+1}|G_3|^T_{\sigma, 0}\right) \le C(\sigma, m)|G|_{\sigma, \mathcal{D}}^T \left(|G|_{\sigma, 0}^T\right)^{k-1}|G|_{\sigma, 0}^T,\\
    |R_k G_2 + S_k G_4|^T_{\sigma, d+1} &\le C(\sigma, m)\left(|R_k|^T_{\sigma, 0}|G_2|^T_{\sigma, d+2} + |S_k|^T_{\sigma, d+1}|G_4|^T_{\sigma, 0}\right)\\
    &\le C(\sigma, m)\left(\left(|G|_{\sigma, 0}^T \right)^k|G|_{\sigma, \mathcal{D}}^T + |G|_{\sigma, \mathcal{D}}^T \left(|G|_{\sigma, 0}^T\right)^{k-1}|G|_{\sigma, 0}^T\right).
\end{align*}
The result follows from the latter and the definition of the norm 
$|\cdot|^T_{\sigma, \mathcal{D}}$ given in~\eqref{def:norm_M_D}.
\end{proof}
Given $d>1$, for the rest of this section, we define the following additional vector $\mathfrak{D}(d)$ which should not be confused with $\mathcal{D}(d)$ in~\eqref{def:vector_D_diff_decay}
\begin{equation}\label{def:vector_frakD_diff_decay}
    \mathfrak{D}(d) = (d+1, d+2, d+2, d+3).
\end{equation}
We have the following properties
\begin{lemma}\label{prop:exp_bound_par}
    Let $\sigma \ge 0$, $d >1$, $J_m$ as in~\eqref{def:J}, $\mathcal{D}(d)$ and $\mathfrak{D}(d)$ as in~\eqref{def:vector_D_diff_decay} and in~\eqref{def:vector_frakD_diff_decay}, respectively. For any $M,G, F \in \mathcal{M}^{ T}_{\sigma, \mathcal{D}(d)}$, $N \in \mathcal{M}^{ T}_{\sigma, \mathcal{D}(d+1)}$ and $B \in \mathcal{M}^{ T}_{\sigma, \mathfrak{D}(d)}$, the following statements hold for some constants $C = C(\sigma, m)$, depending only on $\sigma$ and $m$. 
    \begin{enumerate}
        \item \label{it:bound_expF_par} $e^{F} - \mathrm{Id}_{2m} \in \mathcal{M}^{ T}_{\sigma, \mathcal{D}(d)}$ and  $|e^{F} - \mathrm{Id}_{2m}|^T_{\sigma, \mathcal{D}(d)} \le |F|_{\sigma, \mathcal{D}(d)}^T e^{C|F|_{\sigma, 0}^T}$.
        \item \label{it:boound_expFBexpF_par} $|e^{F^\top}Be^{F}|^T_{\sigma, \mathfrak{D}(d)} \le |B|^T_{\sigma, \mathfrak{D}(d)} \left(1 + C e^{C|F|_{\sigma, 0}^T}|F|_{\sigma, \mathcal{D}(d)}^T  \right)$.
        \item \label{it:bound_expFtopJexpF} $|(e^{F^\top} - \mathrm{Id}_{2m})J_m e^F|^T_{\sigma, \mathfrak{D}(d-1)} \le C |F|_{\sigma, \mathcal{D}(d)}^T e^{C|F|_{\sigma, 0}^T}    $
        \item \label{it:bound_BA_par} $|B\mathcal{A}(F,G)|^T_{\sigma, \mathfrak{D}(2d+1)} \le C|B|^T_{\sigma, \mathfrak{D}(d)}|G|^T_{\sigma, \mathcal{D}(d)}\left(1 + e^{C|F|_{\sigma, 0}^T}|F|_{\sigma, \mathcal{D}(d)}^T\right)$.
        \item \label{it:bound_MA_par} $|M\mathcal{A}(F,G)|^T_{\sigma, \mathfrak{D}(2d+1)} \le C|M|^T_{\sigma, \mathcal{D}(d)}|G|^T_{\sigma, \mathcal{D}(d)}\left(1 + e^{C|F|_{\sigma, 0}^T}|F|_{\sigma, \mathcal{D}(d)}^T \right)$.
        \item \label{it:bound_A(G,N)_par} $|\mathcal{A}(F,N)|^T_{\sigma, \mathcal{D}(d+1)} \le |N|^T_{\sigma, \mathcal{D}(d+1)} \left(1 + Ce^{C|F|_{\sigma, 0}^T}|F|_{\sigma, \mathcal{D}(d)}^T \right). $
        \item \label{it:bound_NJA(F,G)_par} $|N^\top J_m \mathcal{A}(F,G)|^T_{\sigma, \mathfrak{D}(2d+1)} \le C|N|^T_{\sigma, \mathcal{D}(d+1)}|G|^T_{\sigma, \mathcal{D}(d)}\left(1 + e^{C|F|_{\sigma, 0}^T}|F|_{\sigma, \mathcal{D}(d)}^T \right)$, \\
        $|J_m N \mathcal{A}(F,G)|^T_{\sigma, \mathfrak{D}(2d+1)} \le C|N|^T_{\sigma, \mathcal{D}(d+1)}|G|^T_{\sigma, \mathcal{D}(d)}\left(1 + e^{C|F|_{\sigma, 0}^T}|F|_{\sigma, \mathcal{D}(d)}^T \right)$.
        \item \label{it:bound_A-G_par} $|\mathcal{A}(F,G) - G|^T_{\sigma, \mathfrak{D}(2d+1)} \le C |G|^T_{\sigma, \mathcal{D}(d)} |F|_{\sigma, \mathcal{D}(d)}^T e^{C|F|_{\sigma, 0}^T}$.
        \item \label{it:bound_L(expG)_par} If $\Lo F \in \mathcal{M}^{T}_{\sigma, \mathcal{D}(d+1)}$, then $|\Lo e^F|^T_{\sigma, \mathcal{D}(d+1)} \le  C|\Lo F|^T_{\sigma, \mathcal{D}(d+1)} e^{C|F|_{\sigma, 0}^T}\left(1 + |F|_{\sigma, \mathcal{D}(d)}^T\right).$
        \item \label{it:bound_exp_Lomega_exp_par} If $\Lo F \in \mathcal{M}^{T}_{\sigma, \mathcal{D}(d+1)}$, then 
        \begin{align*}
            |(e^{F^\top} - \mathrm{Id}_{2m})J_m \Lo (e^F)|^T_{\sigma, \mathfrak{D}(2d+1)} \le Ce^{C|F|_{\sigma, 0}^T}|\Lo F|^T_{\sigma, \mathcal{D}(d+1)} \left(|F|_{\sigma, \mathcal{D}(d)}^T + \left(|F|_{\sigma, \mathcal{D}(d)}^T\right)^2\right)
        \end{align*}
        \item \label{it:bound_JLomega(A-G)} If $\Lo F, \, \Lo G \in \mathcal{M}^{T}_{\sigma, \mathcal{D}(d+1)}$, then
        \begin{align*}
        &|J_m \Lo \left(\mathcal{A}(F,G)-G\right)|^T_{\sigma, \mathfrak{D}(2d+1)} \\
        &\qquad \le C e^{C|F|_{\sigma, 0}^T} \left(1 + |F|_{\sigma, \mathcal{D}(d)}^T\right) \bigg[|G|_{\sigma, \mathcal{D}(d)}^T|\Lo F|_{\sigma, \mathcal{D}(d+1)}^T + |\Lo G|_{\sigma, \mathcal{D}(d+1)}^T|F|_{\sigma, \mathcal{D}(d)}^T \bigg].
        \end{align*}
        \item \label{it:bound_JMLomegaA_par} If $\Lo F, \, \Lo G \in \mathcal{M}^{T}_{\sigma, \mathcal{D}(d+1)}$, then
        \begin{align*}
            &|J_m M \Lo \left(\mathcal{A}(F,G)\right)|^T_{\sigma, \mathfrak{D}(2d+1)}\\
            &\qquad \le C|M|^T_{\sigma, \mathcal{D}(d)}|\Lo G|_{\sigma, \mathcal{D}(d+1)}^T\\
            &\qquad + C e^{C|F|_{\sigma, 0}^T}|M|^T_{\sigma, \mathcal{D}(d)}\left(1 + |F|_{\sigma, \mathcal{D}(d)}^T\right) \bigg[|G|_{\sigma, \mathcal{D}(d)}^T|\Lo F|_{\sigma, \mathcal{D}(d+1)}^T + |\Lo G|_{\sigma, \mathcal{D}(d+1)}^T|F|_{\sigma, \mathcal{D}(d)}^T\bigg].
        \end{align*}
        \item \label{it:bound_RL_omegaA(F,G)} If $\Lo F, \, \Lo G \in \mathcal{M}^{T}_{\sigma, \mathcal{D}(d+1)}$ and $R \in \mathcal{M}^{T}_{\sigma, \mathfrak{D}(d-1)}$, then
        \begin{align*}
            &|R\Lo \left(\mathcal{A}(F,G)\right)|^T_{\sigma, \mathfrak{D}(2d+1)} \\
            &\qquad \le C |R|^T_{\sigma, \mathfrak{D}(d-1)}|\Lo G|_{\sigma, \mathcal{D}(d+1)}^T\\
            &\qquad + C e^{C|F|_{\sigma, 0}^T}|R|^T_{\sigma, \mathfrak{D}(d-1)}\left(1 + |F|_{\sigma, \mathcal{D}(d)}^T\right) \bigg[|G|_{\sigma, \mathcal{D}(d)}^T|\Lo F|_{\sigma, \mathcal{D}(d+1)}^T + |\Lo G|_{\sigma, \mathcal{D}(d+1)}^T|F|_{\sigma, \mathcal{D}(d)}^T\bigg].
        \end{align*}
    \end{enumerate}
\end{lemma}
\begin{proof}
    To avoid a flood of constants, we denote by $C$ a generic positive constant depending only on $\sigma$ and $m$. In this proof, we widely use the properties contained in Proposition \ref{prop:Csigma_prop_norms_matrix}, especially Properties \ref{prop:crescita_indici_norm_ell} and \ref{prop:bound_norm_product_matrices}, without referring to them each time. 
    
    We observe that 
    \begin{align*}
        |e^F - \mathrm{Id}_{2m}|_{\sigma, \mathcal{D}(d)}^T & \leq  \sum_{k \geq 1}\frac{|F^k|^T_{\sigma, \mathcal{D}(d)}}{k!}  \leq \sum_{k \geq 1}\frac{C|F|^T_{\sigma, \mathcal{D}(d)}|F^{k - 1}|^T_{\sigma, 0}}{k!}  \\
        &\leq  \sum_{k \geq 1}\frac{C^{k-1}|F|^T_{\sigma, \mathcal{D}(d)}(|F|^T_{\sigma, 0})^{k - 1}}{k!}  \leq |F|^T_{\sigma, \mathcal{D}(d)}  e^{C(\sigma,m)|F|^T_{\sigma, 0}}
    \end{align*}
    where in the above inequalities we used Lemma \ref{lemma:Prod_matrx_diff_decay}. This proves~\eqref{it:bound_expF_par}.

    We rewrite $e^{F^\top}Be^{F}$ as
    \begin{equation}\label{eq:prop_dim_item_2_forma}
        e^{F^\top}Be^{F} = (e^{F^\top} - \mathrm{Id}_{2m})Be^{F}  +  B(e^{F} - \mathrm{Id}_{2m}) + B.
    \end{equation}
    We estimate each term on the RHS of the latter separately. For this purpose, using the notation~\eqref{def:Matrix_block} and Property \ref{it:bound_expF_par}, we write $e^{F} - \mathrm{Id}_{2m}$ as follows  
 \begin{equation}\label{eq:prop_matrix_E}
         e^{F} - \mathrm{Id}_{2m} = E =\begin{pmatrix} E_1 & E_2 \\ E_3 & E_4 \end{pmatrix} \qquad   \begin{array}{lll}  E_k : \T^n \times I_T \to \mathcal{M}_m(\R), \quad \mbox{for $1 \le k \le 4$}, \\ E_1 \in \mathcal{M}^{T}_{\sigma, d+1}, \quad E_2 \in \mathcal{M}^{T}_{\sigma, d+2},\\
        E_3 \in \mathcal{M}^{T}_{\sigma, d}, \hspace{8mm} E_4 \in \mathcal{M}^{T}_{\sigma, d +1}. \end{array}  
     \end{equation}
Using the above notation, we can write 
\begin{equation*}
B(e^{F} - \mathrm{Id}_{2m}) = BE=
\begin{pmatrix}
B_1E_1+B_2E_3 & B_1E_2+B_2E_4\\
B_3E_1+B_4E_3 & B_3E_2+B_4E_4
\end{pmatrix}
\end{equation*}
where, by Property \ref{it:bound_expF_par}
\begin{align*}
   |B_1E_1+B_2E_3|^T_{\sigma, d+1} &\le C\left(|B_1|^T_{\sigma, d+1}|E_1|_{\sigma, 0}^T +  |B_2|^T_{\sigma, d+2}|E_3|_{\sigma, 0}^T\right) \le C|B|^T_{\sigma, \mathfrak{D}(d)}|F|_{\sigma, \mathcal{D}(d)}^T e^{C|F|_{\sigma, 0}^T},\\
   |B_1E_2+B_2E_4|^T_{\sigma, d+2} &\le C\left(|B_1|^T_{\sigma, 0}|E_2|_{\sigma, d+2}^T +  |B_2|^T_{\sigma, d+2}|E_4|_{\sigma, 0}^T\right) \le C|B|^T_{\sigma, \mathfrak{D}(d)}|F|_{\sigma, \mathcal{D}(d)}^T e^{C|F|_{\sigma, 0}^T}
\end{align*}
and similarly
\begin{equation*}
        |B_3E_1+B_4E_3|^T_{\sigma, d+2}, \, |B_3E_2+B_4E_4|^T_{\sigma, d+3},\le C |B|^T_{\sigma, \mathfrak{D}(d)}|F|_{\sigma, \mathcal{D}(d)}^T e^{C|F|_{\sigma, 0}^T}.
\end{equation*}
Remembering that $e^F - \mathrm{Id}_{2m}= E$, the above estimates prove that 
\begin{equation*}
    |B(e^F - \mathrm{Id}_{2m})|^T_{\sigma, \mathfrak{D}(d)} \le C |B|^T_{\sigma, \mathfrak{D}(d)}|F|_{\sigma, \mathcal{D}(d)}^T e^{C|F|_{\sigma, 0}^T}.
\end{equation*}
  We stress that $e^{F^\top} - \mathrm{Id}_{2m} = E^\top$ and, as in the previous case, we write
\begin{align*}
&(e^{F^\top} - \mathrm{Id}_{2m})Be^{F}  = E^\top B(E + \mathrm{Id}_{2m}) = E^\top BE + E^\top B , \\ & \quad =\begin{pmatrix}
E_1^\top(B_1E_1+B_2E_3)+E_3^\top(B_3E_1+B_4E_3)
&
E_1^\top(B_1E_2+B_2E_4)+E_3^\top(B_3E_2+B_4E_4)
\\
E_2^\top(B_1E_1+B_2E_3)+E_4^\top(B_3E_1+B_4E_3)
&
E_2^\top(B_1E_2+B_2E_4)+E_4^\top(B_3E_2+B_4E_4)
\end{pmatrix}\\
& + \quad \begin{pmatrix}
E_1^\top B_1+E_3^\top B_3 &
E_1^\top B_2+E_3^\top B_4
\\
E_2^\top B_1+E_4^\top B_3 &
E_2^\top B_2+E_4^\top B_4
\end{pmatrix},
\end{align*}
 using Property \ref{it:bound_expF_par} and the trivial estimate $|E_k|^T_{\sigma, 0} = |e^F - \mathrm{Id}_{2m}|^T_{\sigma, 0} \le 1 + |e^F|^T_{\sigma, 0} \le Ce^{C|F|^T_{\sigma, 0}}$ , for $1 \le k \le 4$, given by Property \ref{it:bound_expF} of Lemma \ref{lem:exp_expansion_bounds}, we have 
\begin{align*}
    &|E_1^\top(B_1E_1+B_2E_3)+E_3^\top(B_3E_1+B_4E_3) + E_1^\top B_1+E_3^\top B_3|^T_{\sigma, d+1} \\
    &\qquad \le C \bigg(|E_1|^T_{\sigma, 0}|B_1|^T_{\sigma, d+1}|E_1|^T_{\sigma,0} + |E_1|^T_{\sigma, d+1}|B_2|^T_{\sigma, 0}|E_3|^T_{\sigma,0} + |E_3|^T_{\sigma, 0}|B_3|^T_{\sigma, 0}|E_1|^T_{\sigma,d+1}\\
    &\qquad +|E_3|^T_{\sigma, 0}|B_4|^T_{\sigma, d+3}|E_3|^T_{\sigma,0} + |E_1|^T_{\sigma, 0}|B_1|^T_{\sigma, d+1} + |E_3|^T_{\sigma, 0}|B_3|^T_{\sigma, d+2}\bigg) \\
    &\qquad \le C |B|^T_{\sigma, \mathfrak{D}(d)}|F|_{\sigma, \mathcal{D}(d)}^T \left( e^{C|F|_{\sigma, 0}^T}\right)^2 \le C |B|^T_{\sigma, \mathfrak{D}(d)}|F|_{\sigma, \mathcal{D}(d)}^T  e^{C|F|_{\sigma, 0}^T}
\end{align*}
and similarly
\begin{align*}
   & |E_1^\top(B_1E_2+B_2E_4)+E_3^\top(B_3E_2+B_4E_4) + E_1^\top B_2+E_3^\top B_4|^T_{\sigma, d+2},\\
   &|E_2^\top(B_1E_1+B_2E_3)+E_4^\top(B_3E_1+B_4E_3) + E_2^\top B_1+E_4^\top B_3|^T_{\sigma, d+2},\\
   & |E_2^\top(B_1E_2+B_2E_4)+E_4^\top(B_3E_2+B_4E_4) + E_2^\top B_2+E_4^\top B_4|^T_{\sigma, d+3} \le C |B|^T_{\sigma, \mathfrak{D}(d)}|F|_{\sigma, \mathcal{D}(d)}^T  e^{C|F|_{\sigma, 0}^T}.
\end{align*}
Combining the above estimates, we prove that 
\begin{equation*}
     |(e^{F^\top} - \mathrm{Id}_{2m})Be^{F}|^T_{\sigma, \mathfrak{D}(d)} \le C  |B|^T_{\sigma, \mathfrak{D}(d)}|F|_{\sigma, \mathcal{D}(d)}^T  e^{C|F|_{\sigma, 0}^T}
\end{equation*}
which concludes the proof of Property \ref{it:boound_expFBexpF_par}.

Using the notation defined by~\eqref{eq:prop_matrix_E}, we write 
\begin{align*}
    &(e^{F^\top} - \mathrm{Id}_{2m})J_m e^F = E^\top J_m (E + \mathrm{Id}_{2m}) = E^\top J_mE + E^\top J_m \\
    & \qquad = \begin{pmatrix}
E_1^\top E_3-E_3^\top E_1
&
E_1^\top E_4-E_3^\top E_2
\\
E_2^\top E_3-E_4^\top E_1
&
E_2^\top E_4-E_4^\top E_2
\end{pmatrix} + \begin{pmatrix}
-E_3^\top&E_1^\top\\
-E_4^\top&E_2^\top
\end{pmatrix}.
\end{align*}
From this decomposition, the proof of Property \ref{it:bound_expFtopJexpF} follows directly from  Property \ref{it:bound_expF_par} and the trivial estimate $|E_k|^T_{\sigma, 0} = |e^F - \mathrm{Id}_{2m}|^T_{\sigma, 0} \le 1 + |e^F|^T_{\sigma, 0} \le Ce^{C|F|^T_{\sigma, 0}}$ , for $1 \le k \le 4$, given by Property \ref{it:bound_expF} of Lemma \ref{lem:exp_expansion_bounds}.

Remembering the definition~\eqref{eq:adjoint_formula}, we observe that 
    \begin{equation}\label{eq:BA(F,G)_est}
        B\mathcal{A}(F,G) = BG + \int_0^1 B(e^{-sF^\top} - \mathrm{Id}_{2m})Ge^{sF} +  BG(e^{sF} - \mathrm{Id}_{2m}) ds.
    \end{equation}
    We need to estimate each term in the RHS of the latter. To this end, using the notation~\eqref{def:Matrix_block}, we can write $BG$ as
    \begin{equation*}
        BG=
\begin{pmatrix}
B_1G_1+B_2G_3 & B_1G_2+B_2G_4\\
B_3G_1+B_4G_3 & B_3G_2+B_4G_4
\end{pmatrix}.
    \end{equation*}
    where, 
    \begin{align*}
        |B_1G_1+B_2G_3|^T_{\sigma, 2d+2} &\le C \left(|B_1|^T_{\sigma, d+1}|G_1|^T_{\sigma, d+1} + |B_2|^T_{\sigma, d+2} |G_3|^T_{\sigma, d}\right)\\
        &\le  C |B|^T_{\sigma, \mathfrak{D}(d)}|G|^T_{\sigma, \mathcal{D}(d)}.
    \end{align*}
    and, similarly, one can verify that 
    \begin{equation*}
        |B_1G_2+B_2G_4|^T_{\sigma, 2d+3}, \, |B_3G_1+B_4G_3|^T_{\sigma, 2d+1}, \, |B_3G_2+B_4G_4|^T_{\sigma, 2d+2},\le C |B|^T_{\sigma, \mathfrak{D}(d)}|G|^T_{\sigma, \mathcal{D}(d)}.
    \end{equation*}
    Combining the above estimates, we prove that 
    \begin{equation}\label{eq:Item_2_prop_norm_par_1}
        |BG|^T_{\sigma, \mathfrak{D}(2d+1)} \le C |B|^T_{\sigma, \mathfrak{D}(d)}|G|^T_{\sigma, \mathcal{D}(d)}.
    \end{equation}
    It remains to estimate the other terms on the RHS of~\eqref{eq:BA(F,G)_est}. To this end,for any $s \in [0,1]$, we denote
 \begin{equation}\label{eq:prop_matrix_E_s}
         e^{\pm sF} - \mathrm{Id}_{2m} = E^{\pm s} =\begin{pmatrix} E^{\pm s}_1 & E^{\pm s}_2 \\[0.2cm] E^{\pm s}_3 & E^{\pm s}_4 \end{pmatrix} 
     \end{equation}
     where $E^{\pm s}_k$ satisfies the same properties of $E_k$ in~\eqref{eq:prop_matrix_E} for all $1 \le k \le 4$. Using the above notation, we can write $BG(e^{sF} - \mathrm{Id}_{2m})$ in~\eqref{eq:BA(F,G)_est} as
     \begin{equation}
        BGE^{+s}=
        \begin{pmatrix}
        (B_1G_1+B_2G_3)E_1^{+s}+(B_1G_2+B_2G_4)E_3^{+s} &
        (B_1G_1+B_2G_3)E_2^{+s}+(B_1G_2+B_2G_4)E_4^{+s}
        \\
        (B_3G_1+B_4G_3)E_1^{+s}+(B_3G_2+B_4G_4)E_3^{+s} &
        (B_3G_1+B_4G_3)E_2^{+s}+(B_3G_2+B_4G_4)E_4^{+s}
        \end{pmatrix}.
    \end{equation}
    As in the previous case, and using Property \ref{it:bound_expF_par}, and remembering that $E^{+s} = e^{sF} - \mathrm{Id}_{2m}$ , one can verify that 
    \begin{equation}\label{eq:Item_2_prop_norm_par_2}
         |BG(e^{sF} - \mathrm{Id}_{2m})|^T_{\sigma, \mathfrak{D}(2d+1)} \le C |B|^T_{\sigma, \mathfrak{D}(d)}|G|^T_{\sigma, \mathcal{D}(d)}|F|_{\sigma, \mathcal{D}(d)}^T e^{C|F|_{\sigma, 0}^T}.
    \end{equation}
    Similarly, using Property \ref{it:bound_expF} of Lemma \ref{lem:exp_expansion_bounds}, a straightforward computation shows that  
    \begin{align*}
        |B(e^{-sF^\top} - \mathrm{Id}_{2m})Ge^{sF}|^T_{\sigma, \mathfrak{D}(2d+1)}, \, |B(e^{-sF} - \mathrm{Id}_{2m})G|^T_{\sigma, \mathfrak{D}(2d+1)} &\le C |B|^T_{\sigma, \mathfrak{D}(d)}|G|^T_{\sigma, \mathcal{D}(d)}|F|_{\sigma, \mathcal{D}(d)}^T e^{C|F|_{\sigma, 0}^T}.
    \end{align*}
    Combining the above estimates with~\eqref{eq:Item_2_prop_norm_par_2},~\eqref{eq:Item_2_prop_norm_par_1} and~\eqref{eq:BA(F,G)_est}, we conclude the proof of Property \ref{it:bound_BA_par}. The proofs of Properties \ref{it:bound_MA_par}, \ref{it:bound_A(G,N)_par} and \ref{it:bound_NJA(F,G)_par} are similar to that of Property \ref{it:bound_BA_par}. For this reason, they are omitted. 

    Remembering~\eqref{eq:adjoint_formula}, we have that 
    \begin{equation*}
        \mathcal{A}(F,G) - G =  \int_0^1 (e^{-sF} - \mathrm{Id}_{2m})Ge^{sF} +   G(e^{sF} - \mathrm{Id}_{2m}) ds. 
    \end{equation*}
We estimate each term in the RHS of the above equality. Using the notation introduced by~\eqref{eq:prop_matrix_E_s}, we observe that
\begin{equation*}
   G(e^{sF} - \mathrm{Id}_{2m}) = \begin{pmatrix}
G_1E_1^{+s}+G_2E_3^{+s} & G_1E_2^{+s}+G_2E_4^{+s}\\
G_3E_1^{+s}+G_4E_3^{+s} & G_3E_2^{+s}+G_4E_4^{+s}
\end{pmatrix}. 
\end{equation*}
and using Property \ref{it:bound_expF_par}, 
\begin{align*}
    |G_1E_1^{+s}+G_2E_3^{+s}|^T_{\sigma, 2d+2} &\le C \left(|G_1|^T_{\sigma, d+1}|E_1^{+s}|^T_{\sigma, d+1} + |G_2|^T_{\sigma, d+2} |E_3^{+s}|^T_{\sigma, d}\right)\\
        &\le  C |G|^T_{\sigma, \mathcal{D}(d)} |F|_{\sigma, \mathcal{D}(d)}^T e^{C|F|_{\sigma, 0}^T}.
\end{align*}
Similarly, one can verify that 
\begin{align*}
    &|G_1E_2+G_2E_4|^T_{\sigma, 2d+3},\, |G_3E_1+G_4E_3|^T_{\sigma, 2d+1}, \\ &\qquad |G_3E_2+G_4E_4|^T_{\sigma, 2d+2},\le  C |G|^T_{\sigma, \mathcal{D}(d)} |F|_{\sigma, \mathcal{D}(d)}^T e^{C|F|_{\sigma, 0}^T},
\end{align*}
and combining the above estimates, we prove that  
\begin{equation*}
    |G(e^{sF} - \mathrm{Id}_{2m})|^T_{\sigma, \mathfrak{D}(2d+1)}\le C |G|^T_{\sigma, \mathcal{D}(d)} |F|_{\sigma, \mathcal{D}(d)}^T e^{C|F|_{\sigma, 0}^T}.
\end{equation*}
Analogously, using Property \ref{it:bound_expF} of Lemma \ref{lem:exp_expansion_bounds}, one can show that 
\begin{equation*}
    |(e^{-sF} - \mathrm{Id}_{2m})Ge^{sF}|^T_{\sigma, \mathfrak{D}(2d+1)}\le C |G|^T_{\sigma, \mathcal{D}(d)} |F|_{\sigma, \mathcal{D}(d)}^T e^{C|F|_{\sigma, 0}^T}.
\end{equation*}
This concludes the proof of Property \ref{it:bound_A-G_par}. 

Noticing that 
\begin{equation*}
    \Lo e^F = e^F \mathcal{A}(F, \Lo E) 
\end{equation*}
and using Properties \ref{it:bound_expF_par}, \ref{it:bound_MA_par} and \ref{it:bound_A(G,N)_par},and Property \ref{it:bound_expF} of Lemma \ref{lem:exp_expansion_bounds}, we conclude the proof of Property \ref{it:bound_L(expG)_par}.

Remembering that $\Lo (e^F) = \Lo (e^F - \mathrm{Id}_{2m})$ and using the notation defined by~\eqref{eq:prop_matrix_E}, we observe that 
\begin{align*}
    (e^{F^\top} - \mathrm{Id}_{2m})J_m \Lo (e^F) =E^\top J_m \Lo E =  \begin{pmatrix}
E_1^\top\Lo E_3-E_3^\top\Lo E_1
&
E_1^\top\Lo E_4-E_3^\top\Lo E_2
\\
E_2^\top\Lo E_3-E_4^\top\Lo E_1
&
E_2^\top\Lo E_4-E_4^\top\Lo E_2
\end{pmatrix}
    \end{align*}
and by Properties \ref{it:bound_expF_par} and \ref{it:bound_L(expG)_par} we have that 
\begin{align*}
    &|E_1^\top\Lo E_3-E_3^\top\Lo E_1|^T_{\sigma, 2d+2} \le C\left(|E_1^\top|^T_{\sigma, d+1}|\Lo E_3|^T_{\sigma, d+1} + |E_3^\top|^T_{\sigma, d}|\Lo E_1|^T_{\sigma, d+2}\right)\\
    &\qquad \le C|e^{F^\top}-\mathrm{Id}_{2m}|^T_{\sigma, \mathcal{D}(d)}|\Lo e^F|^T_{\sigma, \mathcal{D}(d+1)} \le Ce^{C|F|_{\sigma, 0}^T}|\Lo F|^T_{\sigma, \mathcal{D}(d+1)} \left(|F|_{\sigma, \mathcal{D}(d)}^T + \left(|F|_{\sigma, \mathcal{D}(d)}^T\right)^2\right)
\end{align*}
and similarly
\begin{align*}
    &|E_1^\top\Lo E_4-E_3^\top\Lo E_2|^T_{\sigma, 2d+3}, \, |E_2^\top\Lo E_3-E_4^\top\Lo E_1|^T_{\sigma, 2d+1}, \, |E_2^\top\Lo E_4-E_4^\top\Lo E_2|^T_{\sigma, 2d+2}, \\
    &\qquad \le Ce^{C|F|_{\sigma, 0}^T}|\Lo F|^T_{\sigma, \mathcal{D}(d+1)} \left(|F|_{\sigma, \mathcal{D}(d)}^T + \left(|F|_{\sigma, \mathcal{D}(d)}^T\right)^2\right).
\end{align*}
This concludes the proof of Property \ref{it:bound_exp_Lomega_exp_par}.

    We observe that 
    \begin{equation*}
        J_m \Lo \left(\mathcal{A}(F,G)-G\right) = J_m \Lo \left(\int_0^1 \left(e^{-sF} -\mathrm{Id}_{2m}\right) G e^{sF} + G\left(e^{sF} - \mathrm{Id}_{2m}\right)ds\right).
    \end{equation*}
    We analyze each term on the RHS of the latter separately. To this end, using the notation~\eqref{eq:prop_matrix_E_s}, we have that 
    \begin{align*}
        J_m \Lo \left(G\left(e^{sF} - \mathrm{Id}_{2m}\right)\right) &= J_m \Lo \left(GE^{+s}\right) = J_m \Lo \left(G\right)E^{+s} + J_m G\Lo\left(E^{+s}\right)\\
        & = \begin{pmatrix}
\Lo(G_3)E^{+s}_1+\Lo(G_4)E^{+s}_3
&
\Lo(G_3)E^{+s}_2+\Lo(G_4)E^{+s}_4
\\
-\Lo(G_1)E^{+s}_1-\Lo(G_2)E^{+s}_3
&
-\Lo(G_1)E^{+s}_2-\Lo(G_2)E^{+s}_4
\end{pmatrix} \\
&+ \begin{pmatrix}
G_3\Lo(E^{+s}_1)+G_4\Lo(E^{+s}_3)
&
G_3\Lo(E^{+s}_2)+G_4\Lo(E^{+s}_4)
\\
-G_1\Lo(E^{+s}_1)-G_2\Lo(E^{+s}_3)
&
-G_1\Lo(E^{+s}_2)-G_2\Lo(E^{+s}_4)
\end{pmatrix}.
    \end{align*}
 Using and using Properties \ref{it:bound_expF_par} and \ref{it:bound_L(expG)_par}, and noticing that $\Lo e^F = \Lo (e^F-\mathrm{Id}_{2m})$, we obtain that 
 \begin{align*}
    & |\Lo(G_3)E^{+s}_1+\Lo(G_4)E^{+s}_3 + G_3\Lo(E^{+s}_1)+G_4\Lo(E^{+s}_3)|^T_{\sigma, 2d+2}\\
    &\qquad \le C\bigg( |\Lo G_3|^T_{\sigma, d+1} |E^{+s}_1|^T_{\sigma, d+1} + |\Lo G_4|^T_{\sigma, d+2} |E^{+s}_3|^T_{\sigma, d} + |G_3|^T_{\sigma, d} |\Lo E^{+s}_1|^T_{\sigma, d+2} \\
    &\qquad + |G_4|^T_{\sigma, d+1} |\Lo E^{+s}_3|^T_{\sigma, d+1} \bigg) \\
    &\qquad \le C e^{C|F|^T_{\sigma, 0}} \left(1 + |F|^T_{\sigma, \mathcal{D}(d)}\right)\left(|\Lo G|^T_{\sigma, \mathcal{D}(d+1)} |F|^T_{\sigma, \mathcal{D}(d)} + |G|^T_{\sigma, \mathcal{D}(d)} |\Lo F|^T_{\sigma, \mathcal{D}(d+1)}\right).
 \end{align*}
 Similarly, one can verify that 
 \begin{align*}
     & |\Lo(G_3)E^{+s}_2+\Lo(G_4)E^{+s}_4 + G_3\Lo(E^{+s}_2)+G_4\Lo(E^{+s}_4)|^T_{\sigma, 2d+3},\\
     & |\Lo(G_1)E^{+s}_1+\Lo(G_2)E^{+s}_3 +G_1\Lo(E^{+s}_1)+G_2\Lo(E^{+s}_3)|^T_{\sigma, 2d+3},\\
     & |\Lo(G_1)E^{+s}_2+\Lo(G_2)E^{+s}_4 +G_1\Lo(E^{+s}_2)+G_2\Lo(E^{+s}_4)|^T_{\sigma, 2d+4},\\
     &\qquad \le C e^{C|F|^T_{\sigma, 0}} \left(1 + |F|^T_{\sigma, \mathcal{D}(d)}\right)\left(|\Lo G|^T_{\sigma, \mathcal{D}(d+1)} |F|^T_{\sigma, \mathcal{D}(d)} + |G|^T_{\sigma, \mathcal{D}(d)} |\Lo F|^T_{\sigma, \mathcal{D}(d+1)}\right). 
 \end{align*}
     Combining the above estimates, we prove that 
     \begin{align*}
         &\left|J_m \Lo \left(\int_0^1 G\left(e^{sF} - \mathrm{Id}_{2m}\right)ds\right)\right|^T_{\sigma, \mathfrak{D}(2d+1)}\\
         &\qquad \le C e^{C|F|^T_{\sigma, 0}} \left(1 + |F|^T_{\sigma, \mathcal{D}(d)}\right)\left(|\Lo G|^T_{\sigma, \mathcal{D}(d+1)} |F|^T_{\sigma, \mathcal{D}(d)} + |G|^T_{\sigma, \mathcal{D}(d)} |\Lo F|^T_{\sigma, \mathcal{D}(d+1)}\right)
     \end{align*}
     and by Property \ref{it:bound_expF} of Lemma \ref{lem:exp_expansion_bounds} and reasoning as in the previous case, one can verify that 
     \begin{align*}
         &\left|J_m \Lo \left(\int_0^1 \left(e^{-sF} -\mathrm{Id}_{2m}\right) G e^{sF}ds\right)\right|^T_{\sigma, \mathfrak{D}(2d+1)}\\
         &\qquad \le C e^{C|F|^T_{\sigma, 0}} \left(1 + |F|^T_{\sigma, \mathcal{D}(d)}\right)\left(|\Lo G|^T_{\sigma, \mathcal{D}(d+1)} |F|^T_{\sigma, \mathcal{D}(d)} + |G|^T_{\sigma, \mathcal{D}(d)} |\Lo F|^T_{\sigma, \mathcal{D}(d+1)}\right).
     \end{align*}
     This concludes the proof of Property \ref{it:bound_JLomega(A-G)}. 

     We can write $J_m M \Lo \left(\mathcal{A}(F,G)\right)$ as
     \begin{align*}
         &J_m M \Lo \left(\mathcal{A}(F,G)\right) =  J_m M \Lo G + J_m M \Lo \left(\int_0^1 (e^{-sF} - \mathrm{Id}_{2m})G e^{sF} + G(e^{sF} - \mathrm{Id}_{2m})ds\right)
     \end{align*}
     and, using the notation~\eqref{def:Matrix_block}, we observe that 
     \begin{equation*}
J_mM\Lo G
=
\begin{pmatrix}
M_3\Lo G_1+M_4\Lo G_3
&
M_3\Lo G_2+M_4\Lo G_4
\\
-M_1\Lo G_1-M_2\Lo G_3
&
-M_1\Lo G_2-M_2\Lo G_4
\end{pmatrix}
\end{equation*}
     with
     \begin{align*}
         |M_3\Lo G_1+M_4\Lo G_3|^T_{\sigma, 2d+2} &\le C\left(|M_3|^T_{\sigma, d}|\Lo G_1|^T_{\sigma, d+2} + |M_4|^T_{\sigma, d+1}|\Lo G_3|^T_{\sigma, d+1}\right),\\
         &\le C|M|^T_{\sigma, \mathcal{D}(d)}|\Lo G|^T_{\sigma, \mathcal{D}(d+1)}
     \end{align*}
and similarly 
\begin{align*}
  &|M_3\Lo G_2+M_4\Lo G_4|^T_{\sigma, 2d+3}, \,  |M_1\Lo G_1+M_2\Lo G_3|^T_{\sigma, 2d+3}  ,\\
  & \qquad  |M_1\Lo G_2+M_2\Lo G_4|^T_{\sigma, 2d+4}  \le C|M|^T_{\sigma, \mathcal{D}(d)}|\Lo G|^T_{\sigma, \mathcal{D}(d+1)}.
\end{align*}
The rest of the proof of Property \ref{it:bound_JMLomegaA_par} follows by combining  Properties \ref{it:bound_expF_par} and \ref{it:bound_L(expG)_par} and the arguments used in the proofs of the previous properties.

We write  $R\Lo \left(\mathcal{A}(F,G)\right)$ as
     \begin{align*}
         &R\Lo \left(\mathcal{A}(F,G)\right) =  R \Lo G + R \Lo \left(\int_0^1 (e^{-sF} - \mathrm{Id}_{2m})G e^{sF} + G(e^{sF} - \mathrm{Id}_{2m})ds\right)
     \end{align*}
     where 
    \begin{equation*}
R\Lo G
=
\begin{pmatrix}
R_1\Lo G_1+R_2\Lo G_3
&
R_1\Lo G_2+R_2\Lo G_4
\\
R_3\Lo G_1+R_4\Lo G_3
&
R_3\Lo G_2+R_4\Lo G_4
\end{pmatrix}
\end{equation*}
with 
     \begin{align*}
         |R_1\Lo G_1+R_2\Lo G_3|^T_{\sigma, 2d+2} &\le C\left(|R_1|^T_{\sigma, d}|\Lo G_1|^T_{\sigma, d+2} + |R_2|^T_{\sigma, d+1}|\Lo G_3|^T_{\sigma, d+1}\right),\\
         &\le C|R|^T_{\sigma, \mathfrak{D}(d-1)}|\Lo G|^T_{\sigma, \mathcal{D}(d+1)}
     \end{align*}
and similarly 
\begin{align*}
  &|R_1\Lo G_2+R_2\Lo G_4|^T_{\sigma, 2d+3}, \,  |R_3\Lo G_1+R_4\Lo G_3|^T_{\sigma, 2d+3}  ,\\
  & \qquad  |R_3\Lo G_2+R_4\Lo G_4|^T_{\sigma, 2d+4}  \le C|R|^T_{\sigma, \mathfrak{D}(d-1)}|\Lo G|^T_{\sigma, \mathcal{D}(d+1)}.
\end{align*}
The rest of the proof of Property \ref{it:bound_RL_omegaA(F,G)} follows along the same lines as the previous one, combining Properties \ref{it:bound_expF_par} and \ref{it:bound_L(expG)_par}.
\end{proof}

\appendix
\section{Hölder class of functions}\label{app:Holder}
In this section, we recall the definition of  Hölder classes of functions and some of their well-known properties. For this purpose, given an open subset $E$ of $\R^n$ and a positive integer $k \ge 0$,  we denote by $C^k(E)$ the spaces of functions $f: E \to \R$ with continuous partial derivatives $\partial^\alpha f \in C^0(E)$ for all $\alpha \in \N^n$ with $|\alpha|= |\alpha|_1 = \alpha_1+...+\alpha_n \le k$. Furthermore, for all $f \in C^k(E)$, we define the following norm
\begin{equation*}
|f|_{C^k} = \sup_{|\alpha|\le k}|\partial^\alpha f|_{C^0},
\end{equation*}
where $|\partial^\alpha f|_{C^0} = \sup_{x \in E}|\partial^\alpha f(x)|$. 

Let $\sigma = k + \mu$, where $k \in \Z$, $k \ge 0$ and $0 < \mu <1$.  We define by $C^\sigma(E)$ the space of Hölder functions $f\in C^k(E)$ verifying 
\begin{equation}
\label{def:Holder_norm}
|f|_{C^\sigma} = \sup_{|\alpha|\le k}|\partial^\alpha f|_{C^0} + \sup_{x,y\in E\\, x\ne y, |\alpha| = k}{|\partial^\alpha f(x) - \partial^\alpha f(y)| \over |x-y|^\mu}<\infty.
\end{equation}
It is well-known that $C^\sigma(E)$ endowed with the norm~\eqref{def:Holder_norm} is a Banach space. We adopt the same notation for Real-valued functions and matrices. We say that a Real-valued function or a matrix belongs to $C^\sigma(E)$ if each of its components does. In this case, the corresponding norm is defined as the maximum of the norms of its components.

We recall that $C(\cdot)$ stands for a constant depending on the parameters in brackets. The following proposition summarizes some properties of the norm~\eqref{def:Holder_norm}. 
\begin{proposition}
\label{prop: prop_Holder_norms}
We consider $f$, $g \in C^\sigma(E)$ and $\sigma \ge 0$.
\begin{enumerate}
\item For all $\beta \in \N^{n}$ and $s\ge 0$, if $|\beta| + s \le \sigma$ then  $\left|{\partial^{|\beta|} \over \partial{x_1}^{\beta_1}... \partial{x_n}^{\beta_n}} f \right|_{C^s} \le C|f|_{C^\sigma}$.\\
\item  $|fg|_{C^\sigma} \le C(\sigma)\left(|f|_{C^0}|g|_{C^\sigma} + |f|_{C^\sigma}|g|_{C^0}\right)$. 
\end{enumerate}
 Let $E_1$ be an open subset of $\R^n$ and $z:E_1\to E$ a function taking values in the domain of $f$.  In what follows $\partial z$ stands for the partial derivatives of $z$.
\begin{enumerate}
\item[(3)] If $\sigma < 1$, $f \in C^1(E)$, $z \in C^\sigma (E_1)$ then $f\circ z \in C^\sigma(E_1)$ and
$$
|f \circ z|_{C^\sigma} \le C(|f|_{C^1}|z|_{C^\sigma}+ |f|_{C^0}).
$$
\end{enumerate} 
\begin{enumerate}
\item[(4)] If $\sigma < 1$, $f \in C^\sigma(E)$, $z \in C^1 (E_1)$ then $f\circ z \in C^\sigma(E_1)$ 
$$
|f \circ z|_{C^\sigma} \le C(|f|_{C^\sigma}|\partial z|^\sigma_{C^0}+ |f|_{C^0}).
$$
\end{enumerate}
\begin{enumerate}
\item[(5)] If $\sigma \ge 1$ and $f \in C^\sigma (E)$, $z \in C^\sigma (E_1)$ then $f\circ z \in C^\sigma(E_1)$ 
$$
|f \circ z|_{C^\sigma} \le C(\sigma) \left(|f|_{C^\sigma}|\partial z|^\sigma_{C^0} + |f|_{C^1}|\partial z|_{C^{\sigma-1}}+ |f|_{C^0}\right).
$$
\end{enumerate}
\end{proposition}
\begin{proof}
    We refer to~\cite{Hor76} and~\cite{Sca22} for the proof. 
\end{proof}

The following is a quantitative version of the classical Implicit Function Theorem (see, e.g., \cite{ChiAM2}). 

\begin{theorem}
\label{thm:QIFT}
    Let $(X, |\cdot |_X), (Y, |\cdot|_Y)$, $(Z, |\cdot|_Z)$ be Banach spaces and $U \subseteq X,$ $V \subseteq Y$ be open subsets containing $0$. Let $F:U \times V \subseteq X \times Y \to Z$ be a $C^1$ map such that $F(0, 0) = 0$ and $D_2 F(0, 0): Y 
    \to Z$ is invertible. 
    
    Suppose there exist $M, r, \rho > 0$ satisfying the following.
    \begin{enumerate}
        \item   $B_X(r) \times B_Y(\rho) \subseteq U \times V$.
        \item $\| D_2 F(0, 0)^{-1}\|_{L(Z, Y)} \leq M$.
        \item $|F(x, 0)|_Y \leq \tfrac{\rho}{2M}$, for all $x \in B_X(r) $.
        \item $\|D_2 F(x, y) - D_2F(0, 0)\|_{L(Y, Z)} \leq \frac{1}{2M}$, for all $x \in B_X(r)$ and all $y \in B_Y(\rho)$.
    \end{enumerate}
    Then, there exists a $C^1$ map $g: B_X(r) \subseteq X \to B_Y(\rho) \subseteq Y$ such that
    \[ F(x, g(x)) = 0, \qquad \text{ for all } x \in B_X(r).  \]
\end{theorem}

\section*{Acknowledgments} D.S. have been partially supported by the grant PID2024-158570NB-I00  funded by \\
MCIN/AEI/10.13039/501100011033 and “ERDF A way of making Europe”.  D.S. also acknowledges partial support from the ICREA Acadèmia 2023 grant awarded to Dr. Marcel Guardia Munàrriz. 

\bibliographystyle{amsalpha}
\bibliography{ref}
\end{document}